%% file: main.tex
\documentclass[12pt]{article} 
\usepackage[margin=1in]{geometry}
\usepackage{amsfonts,latexsym,amsthm,amssymb,amsmath,amscd,esint,dsfont}
\usepackage{graphicx}
\graphicspath{ {./figure/} }
\usepackage[margin=10pt,font=small,labelfont=bf]{caption}

\usepackage{cases}			
\usepackage{textcomp, gensymb}
\usepackage{xcolor}

\usepackage{tabularx}
\usepackage{array}

\makeatletter
\def\namedlabel#1#2{\begingroup
	\def\@currentlabel{#2}
	\label{#1}\endgroup
}
\makeatother

\usepackage[parfill]{parskip}
\usepackage{fancyhdr}

\renewcommand{\AA}{\mathbb{A}}

\newcommand{\CC}{\mathbb{C}}

\newcommand{\PP}{\mathbb{P}}

\newcommand{\RR}{\mathbb{R}}

\newcommand{\ZZ}{\mathbb{Z}}

\newcommand{\cF}{\mathcal{F}}

\newcommand{\cO}{\mathcal{O}}
\newcommand{\cP}{\mathcal{P}}
\newcommand{\cQ}{\mathcal{Q}}

\newcommand{\cS}{\mathcal{S}}

\newcommand{\GL}{\operatorname{GL}}

\newcommand{\PGL}{\operatorname{PGL}}

\newcommand{\Int}{\operatorname{int}}
\newcommand{\Aff}{\operatorname{Aff}}
\newcommand{\sgn}{\operatorname{sgn}}

\newtheorem{theorem}{Theorem}[section]
\newtheorem{corollary}[theorem]{Corollary}
\newtheorem{lemma}[theorem]{Lemma}

\newtheorem{proposition}[theorem]{Proposition}

\newtheorem{conjecture}[theorem]{Conjecture}

\theoremstyle{definition}
\newtheorem{definition}[theorem]{Definition}

\newtheorem{remark}[theorem]{Remark}
\newtheorem{problem}[theorem]{Problem}

\usepackage{hyperref}
\hypersetup{
    colorlinks=true,
    linkcolor=blue,
    filecolor=magenta,      
    urlcolor=cyan,
    pdftitle={Overleaf Example},
    pdfpagemode=FullScreen,
}

\usepackage[
backend=bibtex,
style=alphabetic,
sorting=ynt
]{biblatex}
\title{Spirals, tic-tac-toe partition, and deep diagonal daps}
\author{Zhengyu Zou}
\date{\today}

\begin{document}
\maketitle

\begin{abstract}
    The deep diagonal map $T_k$ acts on planar polygons by connecting the $k$-th diagonals and intersecting them successively. The map $T_2$ is the pentagram map, and $T_k$ is a generalization.
    We study the action of $T_k$ on two subsets of the so-called twisted polygons, which we term \textit{type-$\alpha$ and type-$\beta$ $k$-spirals}. For $k \geq 2$, $T_{k}$ preserves both types of $k$-spirals. In particular, we show that for $k = 2$ and $k = 3$, both types of $k$-spirals have precompact forward and backward $T_k$-orbits modulo projective transformations. We derive a rational formula for $T_3$, which generalizes the $y$-variables transformation formula of the corresponding quiver mutation by M. Glick and P. Pylyavskyy. We also present four algebraic invariants of $T_3$. These special orbits in the moduli space are partitioned into cells of a $3 \times 3$ tic-tac-toe grid. This establishes the action of $T_k$ on $k$-spirals as a geometric generalization of $T_2$ on convex polygons. 
\end{abstract}

\input{1intro.tex}
\input{2background.tex}
\input{3spiral.tex}
\input{4partition.tex}

\input{5formula.tex}

\input{6precompact.tex}
\input{7twospiral.tex}

\input{8appendix.tex}

\newpage

\printbibliography

\end{document}

%% file: 1intro.tex
\section{Introduction} \label{sec:intro}

\subsection{Context and Motivation}

Given a polygon $P$ in the real projective plane, let $T_k$ be the map that connects its $k$-th diagonals and intersects them successively to form another polygon $P'$ whose vertices are given by the following formula:
\begin{equation} \label{eqn:delta k,1 formula}
    P'_i = P_i P_{i+k} \cap P_{i+1} P_{i+k+1}.
\end{equation} 
Figure \ref{fig:delta demo} demonstrates an example of the action of $T_2$ on a convex heptagon.
The map $T_2$ is called the \textit{pentagram map}, a well-studied discrete dynamical system (see \cite{Schwartz1992, Schwartz2001recurrent, schwartz2008discretemonodromypentagramsmethod, Ovsienko2010}). A well-known result is that $T_2$ preserves convexity.\footnote{A projective polygon is \textit{convex} if some projective transformation maps it to a planar convex polygon in the affine patch} The $T_2$-orbit of a convex polygon sits on a flat torus in the moduli space of projective equivalent convex polygons. On the other hand, the geometry of the map $T_k$ is less well-behaved. For $k \geq 3$, the $T_k$ images of convex polygons may not even be embedded. See Figure \ref{fig:delta demo} for an example of $T_3$ taking a convex heptagon to a polygon that is not even embedded. 

\begin{figure}[ht]
    \centering
    \scriptsize
    
    \def\svgwidth{0.6\columnwidth}
    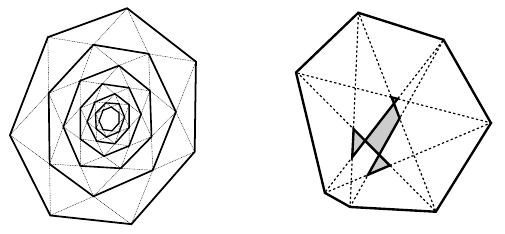

    \caption{Left: The iterative images of a convex heptagon under the action of $T_2$. Right: A convex heptagon whose image under $T_3$ is not even embedded.}
    \label{fig:delta demo}
\end{figure}

Previous results of $T_k$ often had an algebraic and combinatorial flavor, motivated by two branches of studies. The first one was a sequence of works \cite{schwartz2008discretemonodromypentagramsmethod, Ovsienko2010, Soloviev_2013, Ovsienko2013Integrability} that established that the $T_2$ action on the moduli space of projective convex polygons is a discrete completely integrable system; the second one was M.\ Glick's discovery in \cite{glick2011Ypatterns} of the connection between $T_2$ and cluster algebras. 
In \cite{gekhtman2012higherpentagrammapsweighted}, M.\ Gekhtman, M.\ Shapiro, S.\ Tabachnikov, and A.\ Vainshtein generalized the cluster transformations in \cite{glick2011Ypatterns} to the map $T_k$ acting on so-called ``corrugated polygons,'' which are polygonal curves in $\RR\PP^{k}$ satisfying certain coplanarity conditions. \cite{gekhtman2012higherpentagrammapsweighted} showed that $T_k$ is a discrete integrable system. There are numerous integrability results for these higher-dimensional analogs. See \cite{Khesin2013, Beffa2013AGDflows, Beffa2015IntegrableGeneralizations, Khesin2016, izosimov2023longdiagonalpentagrammaps}. 
These led to many applications and connections of $T_k$ to other fields, such as 
octahedral recurrence 
    \cite{schwartz2008discretemonodromypentagramsmethod, difrancesco2012tsystemsboundariesnetworksolutions}, 
the condensation method of computing determinants 
    \cite{schwartz2008discretemonodromypentagramsmethod, Glick2020limit}, 
cluster algebras 
    \cite{glick2011Ypatterns, gekhtman2012higherpentagrammapsweighted,GP2016ymeshes, difrancesco2012tsystemsboundariesnetworksolutions}, 
Poisson Lie groups 
    \cite{Fock2016, Izosimov2022poisson}, 
$T$-systems 
    \cite{KEDEM2015233, difrancesco2012tsystemsboundariesnetworksolutions}, 
Grassmannians 
    \cite{Felipe2019Grassmannians}, 
algebraically closed fields 
    \cite{Weinreich2023}, 
Poncelet polygons 
    \cite{Schwartz2007Poncelet, schwartz2021textbookcasepentagramrigidity, Izosimov2022poncelet, schwartz2024pentagramrigiditycentrallysymmetric}, 
and integrable partial differential equations
    \cite{schwartz2008discretemonodromypentagramsmethod, Ovsienko2010, Nackan2021kdv}. 

The geometric aspects of $T_k$ and other deep diagonal maps on planar polygons remain underexplored. There are only a few studies on the geometries of $T_k$ that focused on small $k$ or polygons with many symmetries. See \cite{schwartz2021textbookcasepentagramrigidity, schwartz2024pentagramrigiditycentrallysymmetric}. There is no established general framework on the type of geometric properties preserved under $T_k$ for $k \geq 3$ that is analogous to convexity under $T_2$. Even less is known for geometric objects that have precompact orbits under $T_k$.

The most relevant result to this endeavor is the discovery of \textit{$k$-birds} under the map $\Delta_{k}$ in \cite{schwartz2024flappingbirdspentagramzoo}. A $k$-bird $P$ is a planar $n$-gon with $n > 3k$, such that there exists a continuous path of polygons $P^{(t)}$ connecting $P$ to the regular $n$-gon where the four lines 
\begin{equation*}
    P^{(t)}_{i} P^{(t)}_{i-k-1}, \hspace{10pt}
    P^{(t)}_{i} P^{(t)}_{i-k}, \hspace{10pt}
    P^{(t)}_{i} P^{(t)}_{i+k}, \hspace{10pt}
    P^{(t)}_{i} P^{(t)}_{i+k+1}   
\end{equation*}
are distinct for all $i = 1, \ldots, n$ and $t \in I$. 
The map $\Delta_{k}$ connects the $(k+1)$-th diagonal of a polygon and intersects the diagonals that are $k$ clicks apart. See Figure \ref{fig:birds example} for the action of $\Delta_2$ on 2-birds. In \cite{schwartz2024flappingbirdspentagramzoo}, R.\ Schwartz showed that the $k$-birds are invariant under both $\Delta_{k}$ and $\Delta_{k}^{-1}$. Experimentally, the $k$-birds seem to have toroidal orbits under $\Delta_k$, which highly resembles the orbit of convex $n$-gons under $T_2$. Schwartz also showed that the $k$-birds have precompact forward $\Delta_k$-orbits modulo affine transformations---a property satisfied by convex $n$-gons under $T_2$. 
\begin{figure}[ht]
    \centering
    
    \def\svgwidth{0.5\columnwidth}
    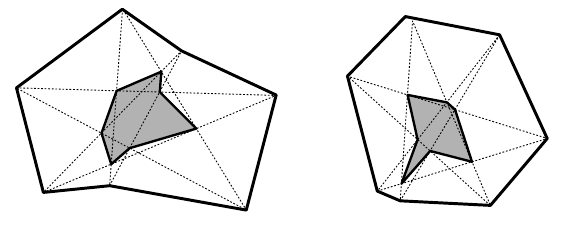

    \caption{Action of $\Delta_2$ on two heptagons that are 2-birds.}
    \label{fig:birds example}
\end{figure}

This paper has two main results. The first one is the discovery of two classes of geometric objects called type-$\alpha$ and type-$\beta$ $k$-spirals that are preserved under $T_k$ for all $k \geq 2$. These two classes of objects are subsets of \textit{twisted polygons}: bi-infinite sequences $P: \ZZ \rightarrow \RR\PP^2$ such that no three consecutive points are collinear, and $P_{i+n} = \phi(P_i)$ for some fixed projective transformation $\phi$ called the \textit{monodromy}. The moduli space of projective equivalent twisted $n$-gons is conventionally denoted by $\cP_n$. The type-$\alpha$ and type-$\beta$ $k$-spirals are the first discovered classes of geometric constructions of $T_k$ that generalize the pentagram map, which provides crucial evidence for a more general understanding of geometrically preserved classes under $T_k$. 

The second result is the precompactness of both forward and backward $T_k$-orbits of type-$\alpha$ and type-$\beta$ $k$-spirals modulo projective transformations for $k = 2$ and $3$, a key property satisfied by convex polygons under the pentagram map discovered by Schwartz in \cite{Schwartz1992}. 
We first examine the action of $T_3$ on type-$\alpha$ and type-$\beta$ 3-spirals. We show that one can characterize type-$\alpha$ and type-$\beta$ 3-spirals via linear constraints on the corner invariants. We also derive a birational formula of $T_3$ for the corner invariants, which is a generalization of the combinatorial formulas developed by \cite{GP2016ymeshes}. Then, we present four global invariants under $T_3$, which we use to prove the precompactness of $T_3$-orbits modulo projective transformations. 
For the case $k = 2$, we show that there exists no type-$\alpha$ 2-spirals and that the type-$\beta$ 2-spirals are distinct from closed convex polygons. We use the Casimir functions of the $T_2$-invariant Poisson structure developed in \cite{schwartz2008discretemonodromypentagramsmethod} and \cite{Ovsienko2010} to show that type-$\beta$ 2-spirals have precompact $T_2$-orbits modulo projective transformations. 

\subsection{The \texorpdfstring{$k$}{k}-Spirals under the Map \texorpdfstring{$T_k$}{T(k)}} \label{subsec:spiral polygons}

Here we describe the geometric picture of a $k$-spiral. For the formal definition, see $\S$\ref{subsec:construct Sk}. Geometrically, $[P] \in \cP_n$ is a \textit{$k$-spiral} if for all $N \in \ZZ$, we can find a representative $P$ such that $\{P_{i}\}_{i \geq N}$ lies on the affine patch, and the triangles $(P_i, P_{i+1}, P_{i+2})$ and $(P_i, P_{i+1}, P_{i+k})$ have positive orientation for all $i \geq N$. We call $P$ an \textit{$N$-representative} of $[P]$.

\begin{figure}[ht]
    \centering
    
    \def\svgwidth{0.5\columnwidth}
    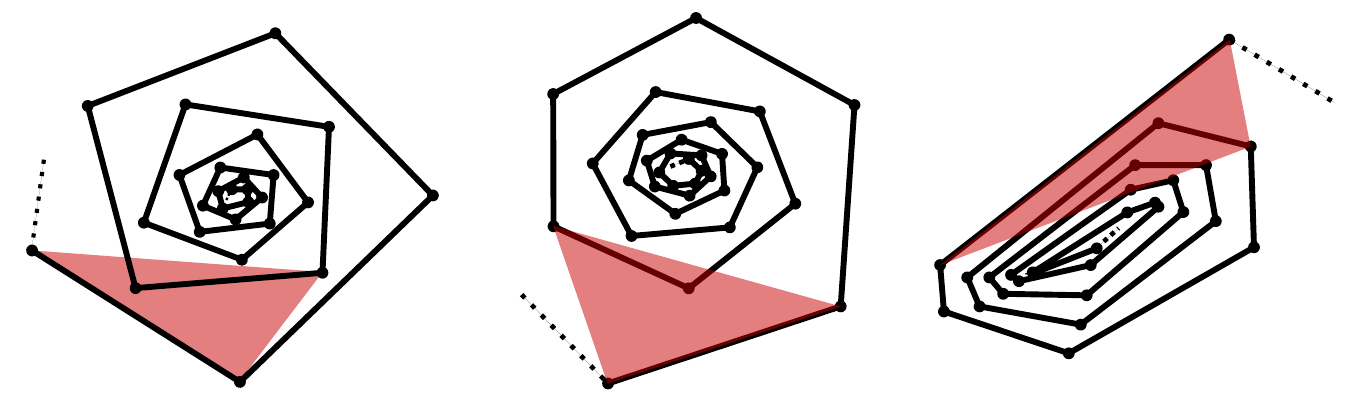

    \caption{A gallery of $5$-spirals. Left: $\cS_{5,3}^\alpha$. Middle: $\cS_{5,3}^\beta$. Right: $\cS_{5,20}^\beta$. The red-shaded triangles indicate the defining orientations and containment relations of type-$\alpha$ and type-$\beta$ $k$-spirals..}
    \label{fig:spiral example}
\end{figure}

We are mainly interested in two types of $k$-spirals, which we term \textit{type-$\alpha$} and \textit{type-$\beta$} (although there certainly exist many more types of spirals, we only consider these two types here). They are $k$-spirals with additional constraints on the arrangement of the four points $P_{i}$, $P_{i+1}$, $P_{i+k}$, $P_{i+k+1}$. For type-$\alpha$ spirals, we require $P_{i+k}$ to be contained in the interior of the triangle $(P_{i}, P_{i+1}, P_{i+k+1})$. For type-$\beta$ spirals, $P_{i+k+1}$ needs to be contained in the interior of $(P_{i}, P_{i+1}, P_{i+k})$. We say $P$ is a \textit{type-$\alpha$ or type-$\beta$ $N$-representative}. A class of twisted polygons $[P]$ is a type-$\alpha$ $k$-spiral (resp.\ $\beta$) if and only if it admits a type-$\alpha$ (resp.\ $\beta$) $N$-representative for all $N \in \ZZ$. 
Let $\cS_{k,n}^\alpha$ and $\cS_{k,n}^\beta$ denote the space of type-$\alpha$ and type-$\beta$ $k$-spirals modulo projective equivalence. We will see in $\S$\ref{subsec:construct Sk} that they are both open in $\cP_n$ and hence have dimension $2n$. Figure \ref{fig:spiral example} illustrates three examples of representatives of $\cS_{5,n}^\alpha$ for $n = 3$, and $20$. 

It turns out that $\cS_{k,n}^\alpha$ and $\cS_{k,n}^\beta$ are invariant under both $T_{k}$ and $T_k^{-1}$. Figure \ref{fig:T5 action on S5} shows the inward half of a representative $P$ of $[P] \in \cS_{5,3}^\beta$, with the red arc representing $P' = T_5(P)$. On the right we have five polygonal arcs by joining vertices of $P$ that are 5 clicks apart. We call them \textit{the transversals of $P$}. One way to distinguish type-$\alpha$ and type-$\beta$ spirals is by looking at the orientations of transversals. The transversals of type-$\alpha$ spirals are counterclockwise, whereas those of type-$\beta$ are clockwise (See Figure \ref{fig:spiral witness flags}). In $\S$\ref{sec:spiral}, we use the orientations of these transversals to prove the following main theorem. 

\begin{theorem} \label{thm:spiral polygon invariance}
    For all $n \geq 2$ and $k \geq 2$, we have $T_k(\cS_{k,n}^\alpha) = \cS_{k,n}^\alpha$. The same is true for type-$\beta$. 
    \begin{figure}[ht]
        \centering
        
    \def\svgwidth{0.5\columnwidth}
    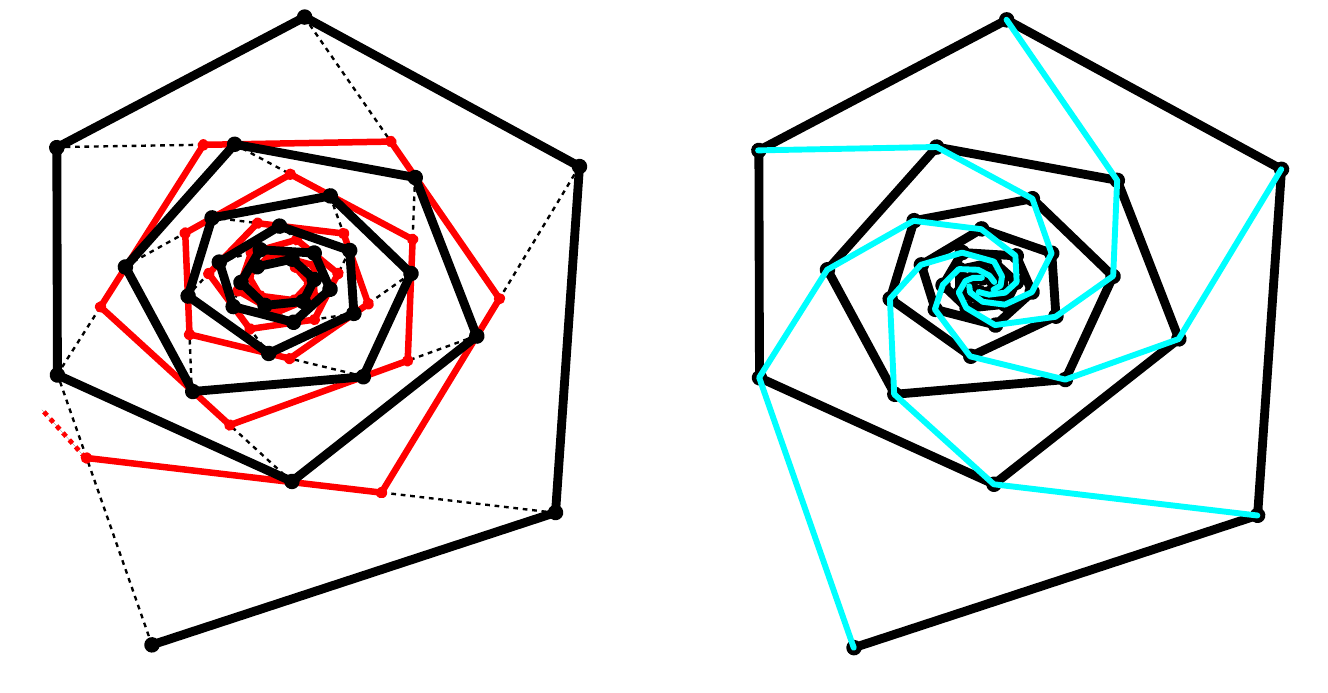

        \caption{Left: $T_5$ acting on a representative $P$ of $[P] \in \cS_{5,3}^\beta$. Right: Transversals of $P$.}
        \label{fig:T5 action on S5}
    \end{figure}
\end{theorem}

A key property satisfied by convex polygons under the pentagram map is that the forward and backward orbits of any convex polygon under the pentagram map are precompact modulo projective tranformations. See \cite[Lemma 3.2]{Schwartz1992}. Experimental results suggest that the $k$-birds also have precompact $\Delta_k$-orbits. In \cite[Conjecture 8.2]{schwartz2024flappingbirdspentagramzoo} Schwartz conjectured that the $k$-birds have precompact forward $\Delta_k$-orbits modulo affine transformations. We observed experimentally that $\cS_{k,n}^\alpha$ and $\cS_{k,n}^\beta$ behave analogously under $T_k$. 

\begin{conjecture} \label{conj:spiral precompact}
    For $n \geq 2$ and $k \geq 2$, the forward and backward $T_k$-orbit of any $[P] \in \cS_{k,n}^{\alpha}$ is precompact in $\cP_n$. The same holds for type-$\beta$.
\end{conjecture}

In $\S$\ref{sec:precompact} and $\S$\ref{sec:two spiral}, we prove Conjecture \ref{conj:spiral precompact} for $k = 2$ and $k = 3$. 

\subsection{Tic-Tac-Toe Partition and Precompact \texorpdfstring{$T_3$}{T(3)} Orbits} \label{subsec:tic tac toe}

Our main focus will be the case $k = 3$, which we prove in $\S$\ref{subsec:precompactness of T3}.

\begin{theorem} \label{thm:3-spiral precompact}
    For $n \geq 2$, the forward and backward $T_3$-orbit of any $[P] \in \cS_{3,n}^{\alpha}$ is precompact in $\cP_n$. The same holds for type-$\beta$.
\end{theorem}

We discovered several interesting properties of the two types of $k$-spirals and the map $T_3$ along our way to prove Theorem \ref{thm:3-spiral precompact}. One major discovery is that the sets $\cS_{3,n}^\alpha$ and $\cS_{3,n}^\beta$ fit well with a local parameterization of $\cP_n \rightarrow \RR^{2n}$ introduced by \cite{Schwartz1992} called \textit{corner invariants} (See $\S$\ref{subsec:twisted polygons corner inv}). The invariant sets of $\cP_n$ under $T_3$ are partitioned by linear boundaries in the parameter space. The boundary lines give a grid pattern that resembles the board of the game ``tic-tac-toe.'' Each of the four ``side-squares'' is invariant under $T_3$. 
\begin{figure}[ht]
    \centering
    
    \def\svgwidth{0.4\columnwidth}
    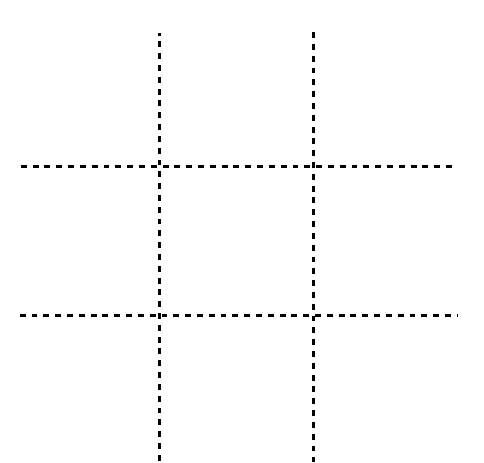

    \caption{The partition of $\RR^2$ into a $3 \times 3$ grid, and the four side-squares of our interest.}
    \label{fig:tic-tac-toe grid}
\end{figure}

To construct the tic-tac-toe board, consider the three intervals $I, J, K$ of $\RR$ given by $I = (-\infty, 0)$, $J = (0, 1)$, $K = (1, \infty)$. The squares are of the form $I \times I$, $I \times J$, $I \times K$, $J \times I$, etc..\ We mark the four side-squares $S_n(I,J)$, $S_n(J,I)$, $S_n(K,J)$, $S_n(J,K)$. See Figure \ref{fig:tic-tac-toe grid} for a visualization of the tic-tac-toe grid. Given $[P] \in \cP_n$, we say $[P] \in S_n(I, J)$ if all even corner invariants of $[P]$ are in $I$, and all odd ones are in $J$. This means if we plot all $n$ pairs of corner invariants $(x_{2i}, x_{2i+1})$ onto $\RR^2$, we would see $n$ points lying in $I \times J$. The other three side squares are defined analogously. 

\begin{figure}[ht]
    \centering
    \scriptsize
    
    \def\svgwidth{0.35\columnwidth}
    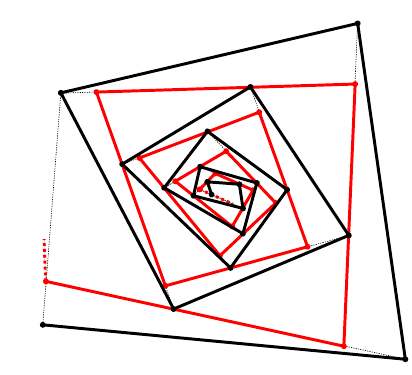

    \ \ \ \ 
    \includegraphics[width=0.35\columnwidth]{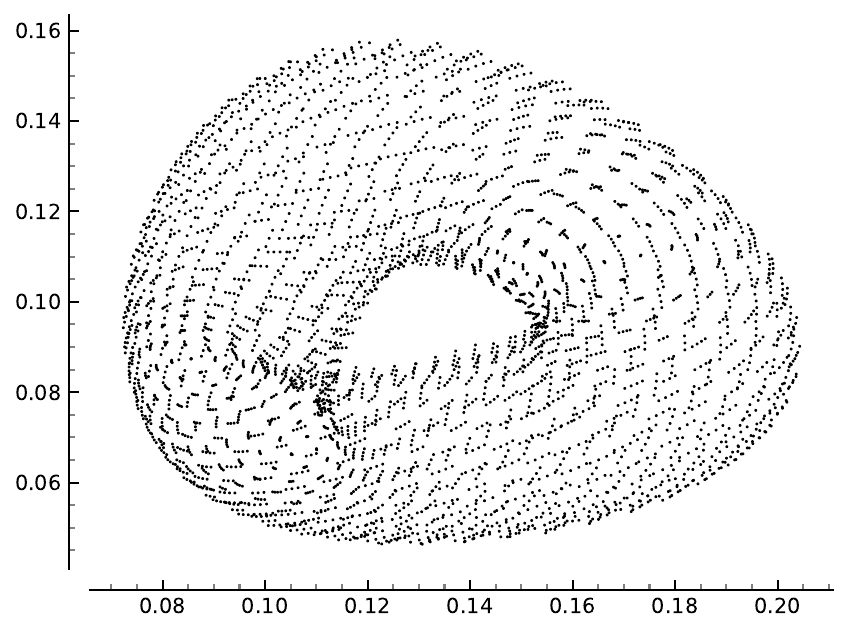}
    \caption{Left: $T_3$ acting on a representative of $[P] \in S_4(K,J)$. Right: The orbit of $P_3^{(m)}$ in $\AA^2$ by fixing $P_{-2} = (0,0)$, $P_{-1} = (1,0)$, $P_0 = (1,1)$, $P_1 = (0,1)$.}
    \label{fig:orbit demo}
\end{figure}

Figure \ref{fig:orbit demo} shows vertices of a representative $P$ of $[P] \in S_4(K,J)$ and the image $P' = T_3(P)$. On the right, we have the projection of the first $2^{11}$ iterations of the orbit of $P$ under $T_3$. Each point corresponds to $P^{(m)}_3$ after normalizing $(P_{-2}^{(m)}, P_{-1}^{(m)}, P_{0}^{(m)}, P_1^{(m)})$ to the unit square (here $P^{(m)} = T_3^m(P)$). We speculate that the orbit lies on a flat torus, where the map $T_3$ acts as a translation on the flat metric. 

Twisted polygons that are assigned to these squares have geometric properties. For example, the closed convex polygons always lie in the center square; two of the side-squares are $\cS_{3,n}^\alpha$ and $\cS_{3,n}^\beta$; the other two side-squares are obtained by reverting the indexing of vertices of these two types of $k$-spirals. These facts will be proved in $\S$\ref{sec:partition}. 

The proof of Theorem \ref{thm:3-spiral precompact} is algebraic. In $\S$\ref{sec:formula} I show that $T_3$ is a birational map on the corner invariants, which generalizes a direct application of \cite[Theorem 1.6]{GP2016ymeshes}. For the explicit formulas, see Equation \eqref{eqn:31 corner}. In $\S$\ref{sec:precompact}, I derive four algebraic invariants of $T_3$, which allow me to show boundedness of the corner invariants of the $T_3$-orbits, thereby proving Theorem \ref{thm:3-spiral precompact}. This approach is reminiscent of Schwartz's second proof of precompactness of $T_2$-orbits of convex polygons in \cite[Section 3B \& 3C]{Schwartz2001recurrent}.

\subsection{The \texorpdfstring{Type-$\beta$}{Type-Beta} 2-Spirals and Precompact \texorpdfstring{$T_2$}{T2} Orbits}

We now proceed to the case $k = 2$, where the map $T_2$ is the classical pentagram map. 
In $\S$\ref{subsec:construct Sk} we show that there exist no type-$\alpha$ 2-spirals (so Conjecture \ref{conj:spiral precompact} is vacuously true for type-$\alpha$ 2-spirals). On the other hand, type-$\beta$ 2-spirals are nontrivial geometric constructions that are distinct from convex polygons. In $\S$\ref{subsec:type beta 2 spirals corner inv}, we show that the corner invariants of type-$\beta$ 2-spirals are also partitioned by linear boundaries, and in particular $\cS_{3,n}^\alpha \subset \cS_{2,n}^\beta$. 

We point out that the type-$\beta$ 2-spirals are not related to the pentagram spirals in \cite{Schwartz2013spirals}. The latter requires $P$ to be a relabeling of $T_2^m(P)$ for some positive integer $m$.

In $\S$\ref{subsec:precompactness of T2 orbits}, we use the Casimir functions of the $T_2$-invariant Poisson structure on $\cP_n$ from \cite{schwartz2008discretemonodromypentagramsmethod} and \cite{Ovsienko2010} to prove Conjecture \ref{conj:spiral precompact} for $k = 2$. 

\begin{theorem} \label{thm:2-spiral precompact}
    For $n \geq 2$, the forward and backward $T_2$-orbit of any $[P] \in \cS_{2,n}^{\beta}$ is precompact in $\cP_n$.
\end{theorem}

\subsection{Obstacles for \texorpdfstring{$k > 3$}{k > 3} and Future Directions}

Our algebraic method of proving Theorem \ref{thm:3-spiral precompact} and \ref{thm:2-spiral precompact} requires a complete characterization of the corner invariants of $\cS_{k,n}^\alpha$ and $\cS_{k,n}^\beta$ and enough algebraic invariants of $T_k$ that uniformly bound the corner invariants away from the boundaries of $\cS_{k,n}^\alpha$ and $\cS_{k,n}^\beta$. However, the corner invariants seem to be not partitioned by linear boundaries for $k > 3$, which makes it difficult to analyze the boundaries of the corner invariants of $\cS_{k,n}^\alpha$ and $\cS_{k,n}^\beta$. Moreover, the map $T_k$ for the corner invariants seems not birational from computer algebra. This makes it difficult to algebraically characterize the corner invariants. 

One future direction is to look at the cross-ratio of different combinations of points other than the ones involved in the definition of corner invariants. In $\S$\ref{sec:appendix} we present a conjecture on a potential algebraic invariant of $T_k$, which can be interpreted as a Casimir function of a Poisson structure over the $y$-parameteris of a quiver $Q_S$. The quiver $Q_S$ is associated to a $Y$-mesh of type $S$ from \cite{GP2016ymeshes} and is isomorphic to the quiver in \cite{gekhtman2012higherpentagrammapsweighted}, which corresponds geometrically to the map $T_k$. 

Another direction is to analyze the two types of $k$-spirals geometrically. There are yet many open problems on the geometry of the two types of $k$-spirals that could hint at the behavior of their $T_k$-orbits. For open problems, see the end of $\S$\ref{subsec:construct Sk}. Answering these geometric problems may provide a new approach to tackle Conjecture \ref{conj:spiral precompact}. 

Finally, for the case $k = 3$, the birational formula for $T_3$ could be applied to other settings such as the action of $T_3$ on Poncelet polygons \cite{schwartz2024pentagramrigiditycentrallysymmetric} or discovering $T_3$-compatible Poisson structures on $\cP_n$ that generalizes the one in \cite{gekhtman2012higherpentagrammapsweighted} for corrugated polygons. 

\subsection{Accompanying Program}

I wrote a web-based program to visualize the orbits of twisted polygons under $T_k$. Readers can access it from the following link:
\begin{center}
    \href{https://zzou9.github.io/pentagram-map/spiral.html}{\textbf{https://zzou9.github.io/pentagram-map/spiral.html}}
\end{center}
When reaching the website, you will see a representative of a twisted polygon displayed in the middle of the screen. You can click on the user manual button for instructions on how to use the program. 
I discovered most of the results by computer experiments using this program. The paper contains rigorous proofs of the beautiful pictures I observed from it. 

\subsection{Acknowledgements}

This work was supported by a Brown University SPRINT/UTRA summer research program grant. I would like to thank my advisor, Richard Evan Schwartz, for introducing the concept of deep diagonal maps, providing extensive insights throughout the project, and offering guidance during the writing process. I would like to thank Anton Izosimov and Sergei Tabachnikov for their insightful discussions on the tic-tac-toe partition. Finally, I am grateful to the referees for their helpful comments and for highlighting the connections between the birational formula and the work of Glick and Pylyavskyy.

%% file: figure/delta_demo.pdf_tex
\begingroup%
  \makeatletter%
  \providecommand\color[2][]{%
    \errmessage{(Inkscape) Color is used for the text in Inkscape, but the package 'color.sty' is not loaded}%
    \renewcommand\color[2][]{}%
  }%
  \providecommand\transparent[1]{%
    \errmessage{(Inkscape) Transparency is used (non-zero) for the text in Inkscape, but the package 'transparent.sty' is not loaded}%
    \renewcommand\transparent[1]{}%
  }%
  \providecommand\rotatebox[2]{#2}%
  \newcommand*\fsize{\dimexpr\f@size pt\relax}%
  \newcommand*\lineheight[1]{\fontsize{\fsize}{#1\fsize}\selectfont}%
  \ifx\svgwidth\undefined%
    \setlength{\unitlength}{241.97331466bp}%
    \ifx\svgscale\undefined%
      \relax%
    \else%
      \setlength{\unitlength}{\unitlength * \real{\svgscale}}%
    \fi%
  \else%
    \setlength{\unitlength}{\svgwidth}%
  \fi%
  \global\let\svgwidth\undefined%
  \global\let\svgscale\undefined%
  \makeatother%
  \begin{picture}(1,0.46348012)%
    \lineheight{1}%
    \setlength\tabcolsep{0pt}%
    \put(0,0){\includegraphics[width=\unitlength,page=1]{delta_demo.pdf}}%
    \put(0.60304045,0.37567518){\color[rgb]{0,0,0}\makebox(0,0)[lt]{\lineheight{1.25}\smash{\begin{tabular}[t]{l}$P$\end{tabular}}}}%
    \put(0.73087307,0.17587762){\color[rgb]{0,0,0}\makebox(0,0)[lt]{\lineheight{1.25}\smash{\begin{tabular}[t]{l}$P'$\end{tabular}}}}%
    \put(0.08761983,0.40879498){\color[rgb]{0,0,0}\makebox(0,0)[lt]{\lineheight{1.25}\smash{\begin{tabular}[t]{l}$P$\end{tabular}}}}%
    \put(0.11698607,0.34643624){\color[rgb]{0,0,0}\makebox(0,0)[lt]{\lineheight{1.25}\smash{\begin{tabular}[t]{l}$P'$\end{tabular}}}}%
    \put(0.17232685,0.32344348){\color[rgb]{0,0,0}\makebox(0,0)[lt]{\lineheight{1.25}\smash{\begin{tabular}[t]{l}$P''$\end{tabular}}}}%
  \end{picture}%
\endgroup%

%% file: figure/birds_demo.pdf_tex
\begingroup%
  \makeatletter%
  \providecommand\color[2][]{%
    \errmessage{(Inkscape) Color is used for the text in Inkscape, but the package 'color.sty' is not loaded}%
    \renewcommand\color[2][]{}%
  }%
  \providecommand\transparent[1]{%
    \errmessage{(Inkscape) Transparency is used (non-zero) for the text in Inkscape, but the package 'transparent.sty' is not loaded}%
    \renewcommand\transparent[1]{}%
  }%
  \providecommand\rotatebox[2]{#2}%
  \newcommand*\fsize{\dimexpr\f@size pt\relax}%
  \newcommand*\lineheight[1]{\fontsize{\fsize}{#1\fsize}\selectfont}%
  \ifx\svgwidth\undefined%
    \setlength{\unitlength}{270.31373752bp}%
    \ifx\svgscale\undefined%
      \relax%
    \else%
      \setlength{\unitlength}{\unitlength * \real{\svgscale}}%
    \fi%
  \else%
    \setlength{\unitlength}{\svgwidth}%
  \fi%
  \global\let\svgwidth\undefined%
  \global\let\svgscale\undefined%
  \makeatother%
  \begin{picture}(1,0.39969038)%
    \lineheight{1}%
    \setlength\tabcolsep{0pt}%
    \put(0,0){\includegraphics[width=\unitlength,page=1]{birds_demo.pdf}}%
  \end{picture}%
\endgroup%

%% file: figure/spiral_gallery.pdf_tex
\begingroup%
  \makeatletter%
  \providecommand\color[2][]{%
    \errmessage{(Inkscape) Color is used for the text in Inkscape, but the package 'color.sty' is not loaded}%
    \renewcommand\color[2][]{}%
  }%
  \providecommand\transparent[1]{%
    \errmessage{(Inkscape) Transparency is used (non-zero) for the text in Inkscape, but the package 'transparent.sty' is not loaded}%
    \renewcommand\transparent[1]{}%
  }%
  \providecommand\rotatebox[2]{#2}%
  \newcommand*\fsize{\dimexpr\f@size pt\relax}%
  \newcommand*\lineheight[1]{\fontsize{\fsize}{#1\fsize}\selectfont}%
  \ifx\svgwidth\undefined%
    \setlength{\unitlength}{650.81303634bp}%
    \ifx\svgscale\undefined%
      \relax%
    \else%
      \setlength{\unitlength}{\unitlength * \real{\svgscale}}%
    \fi%
  \else%
    \setlength{\unitlength}{\svgwidth}%
  \fi%
  \global\let\svgwidth\undefined%
  \global\let\svgscale\undefined%
  \makeatother%
  \begin{picture}(1,0.30175134)%
    \lineheight{1}%
    \setlength\tabcolsep{0pt}%
    \put(0,0){\includegraphics[width=\unitlength,page=1]{spiral_gallery.pdf}}%
  \end{picture}%
\endgroup%

%% file: figure/spiral_action.pdf_tex
\begingroup%
  \makeatletter%
  \providecommand\color[2][]{%
    \errmessage{(Inkscape) Color is used for the text in Inkscape, but the package 'color.sty' is not loaded}%
    \renewcommand\color[2][]{}%
  }%
  \providecommand\transparent[1]{%
    \errmessage{(Inkscape) Transparency is used (non-zero) for the text in Inkscape, but the package 'transparent.sty' is not loaded}%
    \renewcommand\transparent[1]{}%
  }%
  \providecommand\rotatebox[2]{#2}%
  \newcommand*\fsize{\dimexpr\f@size pt\relax}%
  \newcommand*\lineheight[1]{\fontsize{\fsize}{#1\fsize}\selectfont}%
  \ifx\svgwidth\undefined%
    \setlength{\unitlength}{637.12355246bp}%
    \ifx\svgscale\undefined%
      \relax%
    \else%
      \setlength{\unitlength}{\unitlength * \real{\svgscale}}%
    \fi%
  \else%
    \setlength{\unitlength}{\svgwidth}%
  \fi%
  \global\let\svgwidth\undefined%
  \global\let\svgscale\undefined%
  \makeatother%
  \begin{picture}(1,0.52150044)%
    \lineheight{1}%
    \setlength\tabcolsep{0pt}%
    \put(0,0){\includegraphics[width=\unitlength,page=1]{spiral_action.pdf}}%
    \put(0.04832379,0.00663476){\color[rgb]{0,0,0}\makebox(0,0)[lt]{\lineheight{1.25}\smash{\begin{tabular}[t]{l}$P$\end{tabular}}}}%
    \put(-0.01107419,0.15299737){\color[rgb]{1,0,0}\makebox(0,0)[lt]{\lineheight{1.25}\smash{\begin{tabular}[t]{l}$P'$\end{tabular}}}}%
    \put(0,0){\includegraphics[width=\unitlength,page=2]{spiral_action.pdf}}%
    \put(0.56582546,0.00555316){\color[rgb]{0,0,0}\makebox(0,0)[lt]{\lineheight{1.25}\smash{\begin{tabular}[t]{l}$P$\end{tabular}}}}%
  \end{picture}%
\endgroup%

%% file: figure/tic-tac-toe_grid.pdf_tex
\begingroup%
  \makeatletter%
  \providecommand\color[2][]{%
    \errmessage{(Inkscape) Color is used for the text in Inkscape, but the package 'color.sty' is not loaded}%
    \renewcommand\color[2][]{}%
  }%
  \providecommand\transparent[1]{%
    \errmessage{(Inkscape) Transparency is used (non-zero) for the text in Inkscape, but the package 'transparent.sty' is not loaded}%
    \renewcommand\transparent[1]{}%
  }%
  \providecommand\rotatebox[2]{#2}%
  \newcommand*\fsize{\dimexpr\f@size pt\relax}%
  \newcommand*\lineheight[1]{\fontsize{\fsize}{#1\fsize}\selectfont}%
  \ifx\svgwidth\undefined%
    \setlength{\unitlength}{232.01900933bp}%
    \ifx\svgscale\undefined%
      \relax%
    \else%
      \setlength{\unitlength}{\unitlength * \real{\svgscale}}%
    \fi%
  \else%
    \setlength{\unitlength}{\svgwidth}%
  \fi%
  \global\let\svgwidth\undefined%
  \global\let\svgscale\undefined%
  \makeatother%
  \begin{picture}(1,0.95659788)%
    \lineheight{1}%
    \setlength\tabcolsep{0pt}%
    \put(0,0){\includegraphics[width=\unitlength,page=1]{tic-tac-toe_grid.pdf}}%
    \put(0.37491616,0.74337563){\color[rgb]{0,0,0}\makebox(0,0)[lt]{\lineheight{1.25}\smash{\begin{tabular}[t]{l}$S_n(J,K)$\end{tabular}}}}%
    \put(0.70078972,0.43753834){\color[rgb]{0,0,0}\makebox(0,0)[lt]{\lineheight{1.25}\smash{\begin{tabular}[t]{l}$S_n(K,J)$\end{tabular}}}}%
    \put(0.38782233,0.13019886){\color[rgb]{0,0,0}\makebox(0,0)[lt]{\lineheight{1.25}\smash{\begin{tabular}[t]{l}$S_n(J,I)$\end{tabular}}}}%
    \put(0.07430404,0.43951645){\color[rgb]{0,0,0}\makebox(0,0)[lt]{\lineheight{1.25}\smash{\begin{tabular}[t]{l}$S_n(I,J)$\end{tabular}}}}%
    \put(0.31688315,0.90947798){\color[rgb]{0,0,0}\makebox(0,0)[lt]{\lineheight{1.25}\smash{\begin{tabular}[t]{l}0\end{tabular}}}}%
    \put(0.63285383,0.91147283){\color[rgb]{0,0,0}\makebox(0,0)[lt]{\lineheight{1.25}\smash{\begin{tabular}[t]{l}1\end{tabular}}}}%
    \put(0.00175977,0.59488486){\color[rgb]{0,0,0}\makebox(0,0)[lt]{\lineheight{1.25}\smash{\begin{tabular}[t]{l}1\end{tabular}}}}%
    \put(-0.00284891,0.28955785){\color[rgb]{0,0,0}\makebox(0,0)[lt]{\lineheight{1.25}\smash{\begin{tabular}[t]{l}0\end{tabular}}}}%
  \end{picture}%
\endgroup%

%% file: figure/T3_action_on_S4.pdf_tex
\begingroup%
  \makeatletter%
  \providecommand\color[2][]{%
    \errmessage{(Inkscape) Color is used for the text in Inkscape, but the package 'color.sty' is not loaded}%
    \renewcommand\color[2][]{}%
  }%
  \providecommand\transparent[1]{%
    \errmessage{(Inkscape) Transparency is used (non-zero) for the text in Inkscape, but the package 'transparent.sty' is not loaded}%
    \renewcommand\transparent[1]{}%
  }%
  \providecommand\rotatebox[2]{#2}%
  \newcommand*\fsize{\dimexpr\f@size pt\relax}%
  \newcommand*\lineheight[1]{\fontsize{\fsize}{#1\fsize}\selectfont}%
  \ifx\svgwidth\undefined%
    \setlength{\unitlength}{197.71681658bp}%
    \ifx\svgscale\undefined%
      \relax%
    \else%
      \setlength{\unitlength}{\unitlength * \real{\svgscale}}%
    \fi%
  \else%
    \setlength{\unitlength}{\svgwidth}%
  \fi%
  \global\let\svgwidth\undefined%
  \global\let\svgscale\undefined%
  \makeatother%
  \begin{picture}(1,0.9081816)%
    \lineheight{1}%
    \setlength\tabcolsep{0pt}%
    \put(0,0){\includegraphics[width=\unitlength,page=1]{T3_action_on_S4.pdf}}%
    \put(0.03097556,0.11702933){\color[rgb]{0,0,0}\makebox(0,0)[lt]{\lineheight{1.25}\smash{\begin{tabular}[t]{l}$P$\end{tabular}}}}%
    \put(0.13346639,0.24479564){\color[rgb]{1,0,0}\makebox(0,0)[lt]{\lineheight{1.25}\smash{\begin{tabular}[t]{l}$P'$\end{tabular}}}}%
    \put(0,0){\includegraphics[width=\unitlength,page=2]{T3_action_on_S4.pdf}}%
  \end{picture}%
\endgroup%

%% file: 2background.tex
\section{Background} \label{sec:background}

\subsection{Projective Geometry}

The real projective plane $\RR\PP^2$ is the space of 1-dimensional subspaces of $\RR^3$. \textit{Points} of $\RR\PP^2$ are lines in $\RR^3$ that go through the origin. We say that $[x:y:z]$ is a \textit{homogeneous coordinate} of $V \in \RR\PP^2$ if the vector $\tilde V = (x, y, z)$ is a representative of $V$. 
Given two distinct points $V_1, V_2 \in \RR\PP^2$, the \textit{line} $l = V_1V_2$ connecting $V_1$ and $V_2$ is the 2-dimensional hyperplane spanned by the two 1-dimensional subspaces. 
Let $l_1,l_2$ be two lines in $\RR\PP^2$. The \textit{point of intersection} $l_1 \cap l_2$ is the 1-dimensional line given by the intersection of the two 2-dimensional subspaces. In $\RR\PP^2$, there exists a unique line connecting each pair of distinct points and a unique point of intersection given two distinct lines. We call a collection of points $V_1, V_2, \ldots, V_n \in \RR\PP^2$ \textit{in general position} if no three of them are collinear.

The \textit{affine patch} $\AA^2$ consists of points in $\RR\PP^2$ with homogeneous coordinate $[x:y:1]$. We call this canonical choice of coordinate $(x,y,1)$ the \textit{affine coordinate} of a point $V \in \AA^2$. There is a diffeomorphism $\Phi: \RR^2 \rightarrow \AA^2$ given by $\Phi(x,y) = [x:y:1]$. We often identify $\AA^2$ as a copy of $\RR^2$ in $\RR\PP^2$.
The line $\RR\PP^2 - \AA^2$ is called the \textit{line at infinity}.

A map $\phi: \RR\PP^2 \rightarrow \RR\PP^2$ is a \textit{projective transformation} if it maps points to points and lines to lines and is bijective. Algebraically, the group of projective transformations is $\PGL_3(\RR) = \GL_3(\RR) / \RR^* I$, where we are modding by the subgroup $\RR^* I = \{\lambda I: \lambda \in \RR^*\}$ and $I$ is the $3 \times 3$ identity matrix. We state a classical result regarding projective transformations below with its proof omitted.

\begin{theorem} \label{thm:proj transform}
    Given two 4-tuples of points $(V_1,V_2,V_3,V_4)$ and $(W_1,W_2,W_3,W_4)$ in $\RR\PP^2$, both in general position, there exists a unique $\phi \in \PGL_3(\RR)$ such that $\phi(V_i) = W_i$. 
\end{theorem}

The group of \textit{affine transformations} $\Aff_2(\RR)$ on $\AA^2$ is the subgroup of projective transformations that fixes the line at infinity. It is isomorphic to a semidirect product of $\GL_2(\RR)$ and $\RR^2$. Elements of $\Aff_2(\RR)$ can be uniquely expressed as a tuple $(M',v)$ where $M' \in \GL_2(\RR)$ and $v \in \RR^2$. Let $\Aff_2^+(\RR)$ denote the subgroup of $\Aff_2^+(\RR)$ where $(M',v) \in \Aff_2^+(\RR)$ iff $\det(M') > 0$. These are orientation-preserving affine transformations. 

\subsection{Orientation of Affine Triangles} \label{subsec:triangle orientation}

Given an ordered 3-tuple $(V_1,V_2,V_3)$ of points in $\RR^2$ or $\AA^2$, let $\Int(V_1, V_2, V_3)$ denote the interior of the affine triangle with vertices $V_1, V_2, V_3$. There is a canonical way to define the orientation of an ordered 3-tuple. Let $\tilde V_i$ be the affine coordinate of $V_i$. We consider the \textit{signed area} $\cO(V_1,V_2,V_3)$ of the oriented triangle, which can be computed as
\begin{equation} \label{eqn:triangle orientation}
    \cO(V_1,V_2,V_3) = \det(\tilde V_1, \tilde V_2, \tilde V_3).
\end{equation}
The determinant is evaluated on the $3 \times 3$ matrix with column vectors $\tilde V_i$. We say an ordered 3-tuple $(V_1, V_2, V_3)$ is \textit{positive} if $\cO(V_1,V_2,V_3)>0$.  Figure \ref{fig:triangle orientation} shows an example of a positive 3-tuple.

\begin{figure}[ht]
    \centering
    \footnotesize
    
    \def\svgwidth{0.8\columnwidth}
    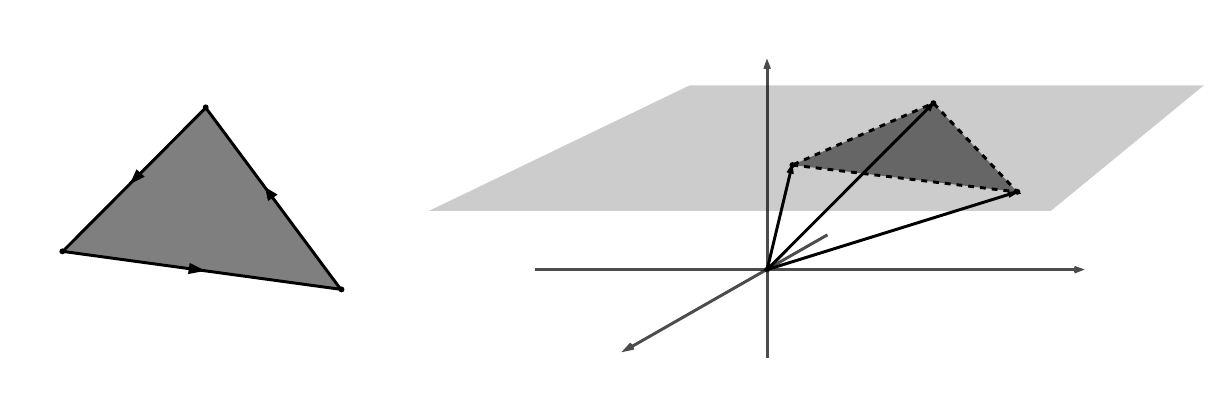

    \caption{A positive 3-tuple of affine points $(V_1, V_2, V_3)$.}
    \label{fig:triangle orientation}
\end{figure}

Here is another way to compute $\cO$ using the $\RR^2$ coordinates of $V_1, V_2, V_3$:
\begin{equation} \label{eqn:vector orientations}
    \begin{aligned}
        \cO(V_1,V_2,V_3) 
        &= \det(V_1,V_2) + \det(V_2,V_3) + \det(V_3,V_1) \\
        &= \det(V_i - V_{i-1}, V_{i+1} - V_i) \ \text{ for $i = 1, 2, 3$}
    \end{aligned}
\end{equation}
where the determinant is evaluated on the $2\times2$ matrix.

$\cO$ interacts with the action of $\Aff_2^+(\RR)$ and the symmetric group $S_3$ on planar/affine triangles in the following way: Given $M \in \Aff_2^+(\RR)$, let $V'_i = M(V_i)$. One can show that $(V_1, V_2, V_3)$ is positive iff $(V_1', V_2', V_3')$ is positive. 
On the other hand, for all $\sigma \in S_3$, $\cO(V_{\sigma(1)}, V_{\sigma(2)}, V_{\sigma(3)})
= \sgn(\sigma) \cO(V_1, V_2, V_3)$, so $\cO(V_{\sigma(1)}, V_{\sigma(2)}, V_{\sigma(3)})
= \cO(V_1, V_2, V_3)$ when $\sigma$ is a 3-cycle.

Below are useful equivalence conditions for the positivity of $(V_1, V_2, V_3)$. The proof is elementary, so we will omit it.

\begin{proposition} \label{prop:orientation lemma}
    Given $V_1, V_2, V_3 \in \RR^2$ in general position, and $W \in \Int(V_1,V_2,V_3)$, the following are equivalent:
    \begin{enumerate}
        \item $(V_1, V_2, V_3)$ is positive.
        \item $(V_i, V_{i+1}, W)$ is positive for some $i \in \{1,2,3\}$.
        \item $(V_i, V_{i+1}, W)$ is positive for all $i \in \{1,2,3\}$.
        \item $\det(V_i - V_{i-1}, V_{i+1} - W) > 0$ for some $i \in \{1,2,3 \}$. 
        \item $\det(V_i - V_{i-1}, V_{i+1} - W) > 0$ for all $i \in \{1,2,3 \}$. 
    \end{enumerate}
\end{proposition}

\subsection{The Cross-Ratio}

The cross-ratio is used to construct a projective-invariant parametrization of the $k$-spirals. There are multiple ways to define the cross-ratio of four collinear points on the projective plane, each using its own permutation of the points. We follow the convention used in \cite{Schwartz1992}. Given four collinear points $A, B, C, D$ on a line $\omega \subset \RR\PP^2$, we choose a projective transformation $\psi$ that maps $\omega$ to the $x$-axis of $\AA^2$. Let $a, b, c, d$ be the $x$-coordinates of $\psi(A)$, $\psi(B)$, $\psi(C)$, $\psi(D)$. We define the \textit{cross-ratio} to be the following quantity:
\begin{equation} \label{eqn:chi definition}
    \chi(A, B, C, D) := \frac{(a - b)(c - d)}{(a - c)(b - d)} .
\end{equation}
If $A$ lies on the line at infinity, we let $\chi(A, B, C, D) = \frac{c - d}{b - d}$. One can check that given any $\phi \in \PGL_3(\RR)$, 
\begin{equation*}
    \chi(A, B, C, D) = \chi(\phi(A), \phi(B), \phi(C), \phi(D)) . 
\end{equation*}

We also define the cross-ratio for four projective lines. Let $l, m, n, k$ be four lines intersecting at a common point $O$. Normalize with a projective transformation so that $l,m,n,k \subset \AA^2$ with slopes $s_l, s_m, s_n, s_k$. We define 
\begin{equation} \label{eqn:chi definition lines}
    \chi(l, m, n, k) = \frac{(s_l - s_m) (s_n - s_k)}{(s_l - s_n) (s_m - s_k)}
\end{equation}
with $\chi(l,m,n,k) = \frac{s_n - s_k}{s_m - s_k}$ if $s_l = \infty$. 

If $\omega$ is a line that does not go through $O$ and intersects $l, m, n, k$ at $A, B, C, D$ respectively, we have 
\begin{equation} \label{eqn:chi prop}
    \chi(l, m, n, k) = \chi(A, B, C, D) .
\end{equation}
See Figure \ref{fig:cross-ratio} for the configuration. The proof is elementary, so we will omit it. 

\begin{figure}[ht]
    \centering
    \scriptsize
    
    \def\svgwidth{0.4\columnwidth}
    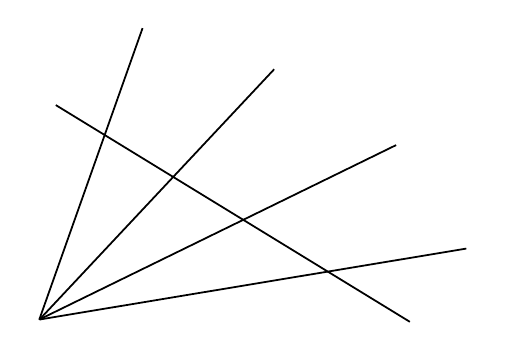

    \caption{The configuration in Equation \eqref{eqn:chi prop}.}
    \label{fig:cross-ratio}
\end{figure}

\subsection{Twisted Polygons, Corner Invariants} \label{subsec:twisted polygons corner inv}

Introduced in \cite{schwartz2008discretemonodromypentagramsmethod}, a \textit{twisted $n$-gon} is a bi-infinite sequence $P : \ZZ \rightarrow \RR\PP^2$, along with a projective transformation $M \in \PGL_3(\RR)$ called the \textit{monodromy}, such that every three consecutive points of $P$ are in general position, and $P_{i+n} = M(P_i)$ for all $i \in \ZZ$. When $M$ is the identity, we get an ordinary closed $n$-gon. Two twisted $n$-gons $P, Q$ are \textit{equivalent} if there exists $\phi \in \PGL_3(\RR)$ such that $\phi(P_i) = Q_i$ for all $i \in \ZZ$. The two monodromies $M_p$ and $M_q$ satisfy $M_q = \phi M_p \phi^{-1}$. Let $\cP_n$ denote the space of twisted $n$-gons modulo projective equivalence. 

The cross-ratio allows us to parameterize $\cP_n$ with coordinates in $\RR^{2n}$. 
Given a twisted $n$-gon $P$, the \textit{corner invariants} of $P$ is a coordinate system $x_{0}(P), \ldots, x_{2n-1}(P)$ given by 
\begin{equation} \label{eqn:corner invariant def}
    \left\{ 
        \begin{aligned}
            x_{2i}(P) = \chi(P_{i-2}, \, P_{i-1}, \, P_{i-2} P_{i-1} \cap P_{i} P_{i+1}, \, P_{i-2} P_{i-1} \cap P_{i+1} P_{i+2}) ;  \\
            x_{2i+1}(P) = \chi(P_{i+2}, \, P_{i+1}, \, P_{i+2} P_{i+1} \cap P_{i} P_{i-1}, \, P_{i+2} P_{i+1} \cap P_{i-1} P_{i-2}).
        \end{aligned}
    \right.
\end{equation}

\begin{figure}[ht]
    \centering
    \scriptsize
    
    \def\svgwidth{0.9\columnwidth}
    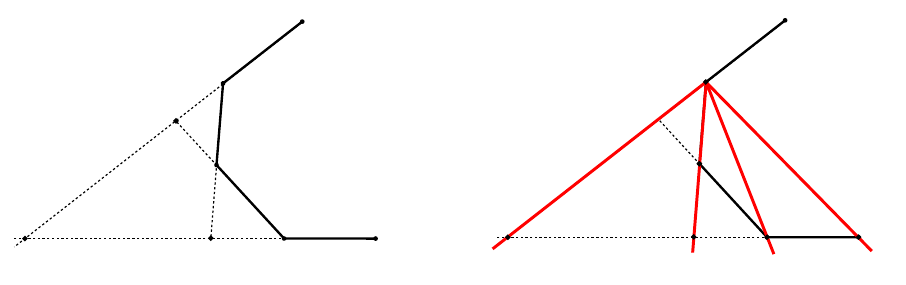

    \caption{Left: The corner invariants $x_{2i}(P) = \chi(P_{i-2}, P_{i-1}, A, O)$ computed using Equation \eqref{eqn:corner invariant def}. Right: $x_{2i}(P) = \chi(l_{1,-2}, l_{1,-1}, l_{1,0}, l_{1,2})$ computed using Equation \eqref{eqn:corner invariant lines def}.}
    \label{fig:corner inv}
\end{figure}

See the left side of Figure \ref{fig:corner inv} for a geometric interpretation of the corner invariants. Let $l_{a,b} = P_{i+a} P_{i+b}$. By Equation \eqref{eqn:chi definition lines}, the corner invariants can be computed by 
\begin{equation} \label{eqn:corner invariant lines def}
    \left\{ 
        \begin{aligned}
            x_{2i}(P) &= \chi(l_{1,-2}, l_{1,-1}, l_{1,0}, l_{1,2}) ;  \\
            x_{2i+1}(P) &= \chi(l_{-1,2}, l_{-1,1}, l_{-1,0}, l_{-1,-2}) .
        \end{aligned}
    \right.
\end{equation}
See the right side of Figure \ref{fig:corner inv} for the line configurations. 

Since $\chi$ is invariant under projective transformations, for all $j$ we have $x_j(P) = x_{j+2n}(P)$, so a $2n$-tuple of corner invariants is enough to fully determine the projective equivalence class of a twisted $n$-gon. We use $x_j(P)$ to denote the corner invariants of $[P] \in \cP_n$ without adding square brackets around $P$. 
To obtain the corner invariants of $[P] \in \cP_n$, one can simply choose an arbitrary representative $P$ and compute its corner invariants. \cite[Equation (19) \& (20)]{schwartz2008discretemonodromypentagramsmethod} showed that one can also revert the process and obtain a representative twisted polygon of the equivalence class given its corner invariants. 

%% file: figure/triangle_orientation.pdf_tex
\begingroup%
  \makeatletter%
  \providecommand\color[2][]{%
    \errmessage{(Inkscape) Color is used for the text in Inkscape, but the package 'color.sty' is not loaded}%
    \renewcommand\color[2][]{}%
  }%
  \providecommand\transparent[1]{%
    \errmessage{(Inkscape) Transparency is used (non-zero) for the text in Inkscape, but the package 'transparent.sty' is not loaded}%
    \renewcommand\transparent[1]{}%
  }%
  \providecommand\rotatebox[2]{#2}%
  \newcommand*\fsize{\dimexpr\f@size pt\relax}%
  \newcommand*\lineheight[1]{\fontsize{\fsize}{#1\fsize}\selectfont}%
  \ifx\svgwidth\undefined%
    \setlength{\unitlength}{583.68518739bp}%
    \ifx\svgscale\undefined%
      \relax%
    \else%
      \setlength{\unitlength}{\unitlength * \real{\svgscale}}%
    \fi%
  \else%
    \setlength{\unitlength}{\svgwidth}%
  \fi%
  \global\let\svgwidth\undefined%
  \global\let\svgscale\undefined%
  \makeatother%
  \begin{picture}(1,0.3281553)%
    \lineheight{1}%
    \setlength\tabcolsep{0pt}%
    \put(0,0){\includegraphics[width=\unitlength,page=1]{triangle_orientation.pdf}}%
    \put(0.0076785,0.10176179){\color[rgb]{0,0,0}\makebox(0,0)[lt]{\lineheight{1.25}\smash{\begin{tabular}[t]{l}$V_1$\end{tabular}}}}%
    \put(0.2873087,0.06255523){\color[rgb]{0,0,0}\makebox(0,0)[lt]{\lineheight{1.25}\smash{\begin{tabular}[t]{l}$V_2$\end{tabular}}}}%
    \put(0.15948835,0.25339845){\color[rgb]{0,0,0}\makebox(0,0)[lt]{\lineheight{1.25}\smash{\begin{tabular}[t]{l}$V_3$\end{tabular}}}}%
    \put(0.63459857,0.21336707){\color[rgb]{0,0,0}\makebox(0,0)[lt]{\lineheight{1.25}\smash{\begin{tabular}[t]{l}$\tilde V_1$\end{tabular}}}}%
    \put(0.8507718,0.16078231){\color[rgb]{0,0,0}\makebox(0,0)[lt]{\lineheight{1.25}\smash{\begin{tabular}[t]{l}$\tilde V_2$\end{tabular}}}}%
    \put(0.75809646,0.25736482){\color[rgb]{0,0,0}\makebox(0,0)[lt]{\lineheight{1.25}\smash{\begin{tabular}[t]{l}$\tilde V_3$\end{tabular}}}}%
    \put(0.44903037,0.21749611){\color[rgb]{0,0,0}\makebox(0,0)[lt]{\lineheight{1.25}\smash{\begin{tabular}[t]{l}$\AA^2$\end{tabular}}}}%
  \end{picture}%
\endgroup%

%% file: figure/inverse_cross.pdf_tex
\begingroup%
  \makeatletter%
  \providecommand\color[2][]{%
    \errmessage{(Inkscape) Color is used for the text in Inkscape, but the package 'color.sty' is not loaded}%
    \renewcommand\color[2][]{}%
  }%
  \providecommand\transparent[1]{%
    \errmessage{(Inkscape) Transparency is used (non-zero) for the text in Inkscape, but the package 'transparent.sty' is not loaded}%
    \renewcommand\transparent[1]{}%
  }%
  \providecommand\rotatebox[2]{#2}%
  \newcommand*\fsize{\dimexpr\f@size pt\relax}%
  \newcommand*\lineheight[1]{\fontsize{\fsize}{#1\fsize}\selectfont}%
  \ifx\svgwidth\undefined%
    \setlength{\unitlength}{243.97059884bp}%
    \ifx\svgscale\undefined%
      \relax%
    \else%
      \setlength{\unitlength}{\unitlength * \real{\svgscale}}%
    \fi%
  \else%
    \setlength{\unitlength}{\svgwidth}%
  \fi%
  \global\let\svgwidth\undefined%
  \global\let\svgscale\undefined%
  \makeatother%
  \begin{picture}(1,0.66338242)%
    \lineheight{1}%
    \setlength\tabcolsep{0pt}%
    \put(0,0){\includegraphics[width=\unitlength,page=1]{inverse_cross.pdf}}%
    \put(0.81663594,0.00366548){\color[rgb]{0,0,0}\makebox(0,0)[lt]{\lineheight{1.25}\smash{\begin{tabular}[t]{l}$\omega$\end{tabular}}}}%
    \put(0.26140333,0.62505241){\color[rgb]{0,0,0}\makebox(0,0)[lt]{\lineheight{1.25}\smash{\begin{tabular}[t]{l}$l$\end{tabular}}}}%
    \put(0.54087678,0.54797684){\color[rgb]{0,0,0}\makebox(0,0)[lt]{\lineheight{1.25}\smash{\begin{tabular}[t]{l}$m$\end{tabular}}}}%
    \put(0.76768995,0.39358987){\color[rgb]{0,0,0}\makebox(0,0)[lt]{\lineheight{1.25}\smash{\begin{tabular}[t]{l}$n$\end{tabular}}}}%
    \put(0.93916981,0.17800728){\color[rgb]{0,0,0}\makebox(0,0)[lt]{\lineheight{1.25}\smash{\begin{tabular}[t]{l}$k$\end{tabular}}}}%
    \put(0.15600762,0.44681075){\color[rgb]{0,0,0}\makebox(0,0)[lt]{\lineheight{1.25}\smash{\begin{tabular}[t]{l}$A$\end{tabular}}}}%
    \put(0.30492705,0.35643769){\color[rgb]{0,0,0}\makebox(0,0)[lt]{\lineheight{1.25}\smash{\begin{tabular}[t]{l}$B$\end{tabular}}}}%
    \put(0.45719561,0.25948612){\color[rgb]{0,0,0}\makebox(0,0)[lt]{\lineheight{1.25}\smash{\begin{tabular}[t]{l}$C$\end{tabular}}}}%
    \put(0.63564887,0.15462514){\color[rgb]{0,0,0}\makebox(0,0)[lt]{\lineheight{1.25}\smash{\begin{tabular}[t]{l}$D$\end{tabular}}}}%
    \put(0.02027219,0.00207006){\color[rgb]{0,0,0}\makebox(0,0)[lt]{\lineheight{1.25}\smash{\begin{tabular}[t]{l}$O$\end{tabular}}}}%
  \end{picture}%
\endgroup%

%% file: figure/corner_inv.pdf_tex
\begingroup%
  \makeatletter%
  \providecommand\color[2][]{%
    \errmessage{(Inkscape) Color is used for the text in Inkscape, but the package 'color.sty' is not loaded}%
    \renewcommand\color[2][]{}%
  }%
  \providecommand\transparent[1]{%
    \errmessage{(Inkscape) Transparency is used (non-zero) for the text in Inkscape, but the package 'transparent.sty' is not loaded}%
    \renewcommand\transparent[1]{}%
  }%
  \providecommand\rotatebox[2]{#2}%
  \newcommand*\fsize{\dimexpr\f@size pt\relax}%
  \newcommand*\lineheight[1]{\fontsize{\fsize}{#1\fsize}\selectfont}%
  \ifx\svgwidth\undefined%
    \setlength{\unitlength}{434.80867641bp}%
    \ifx\svgscale\undefined%
      \relax%
    \else%
      \setlength{\unitlength}{\unitlength * \real{\svgscale}}%
    \fi%
  \else%
    \setlength{\unitlength}{\svgwidth}%
  \fi%
  \global\let\svgwidth\undefined%
  \global\let\svgscale\undefined%
  \makeatother%
  \begin{picture}(1,0.31353949)%
    \lineheight{1}%
    \setlength\tabcolsep{0pt}%
    \put(0,0){\includegraphics[width=\unitlength,page=1]{corner_inv.pdf}}%
    \put(0.74031288,0.12078145){\color[rgb]{0,0,0}\makebox(0,0)[lt]{\lineheight{1.25}\smash{\begin{tabular}[t]{l}$P_i$\end{tabular}}}}%
    \put(0.7441856,0.23640807){\color[rgb]{0,0,0}\makebox(0,0)[lt]{\lineheight{1.25}\smash{\begin{tabular}[t]{l}$P_{i+1}$\end{tabular}}}}%
    \put(0.84845486,0.30169009){\color[rgb]{0,0,0}\makebox(0,0)[lt]{\lineheight{1.25}\smash{\begin{tabular}[t]{l}$P_{i+2}$\end{tabular}}}}%
    \put(0.84967007,0.06136404){\color[rgb]{0,0,0}\makebox(0,0)[lt]{\lineheight{1.25}\smash{\begin{tabular}[t]{l}$P_{i-1}$\end{tabular}}}}%
    \put(0.9439024,0.06081043){\color[rgb]{0,0,0}\makebox(0,0)[lt]{\lineheight{1.25}\smash{\begin{tabular}[t]{l}$P_{i-2}$\end{tabular}}}}%
    \put(0.20720333,0.11935315){\color[rgb]{0,0,0}\makebox(0,0)[lt]{\lineheight{1.25}\smash{\begin{tabular}[t]{l}$P_i$\end{tabular}}}}%
    \put(0.211076,0.23497978){\color[rgb]{0,0,0}\makebox(0,0)[lt]{\lineheight{1.25}\smash{\begin{tabular}[t]{l}$P_{i+1}$\end{tabular}}}}%
    \put(0.31534531,0.3002618){\color[rgb]{0,0,0}\makebox(0,0)[lt]{\lineheight{1.25}\smash{\begin{tabular}[t]{l}$P_{i+2}$\end{tabular}}}}%
    \put(0.31147267,0.02862218){\color[rgb]{0,0,0}\makebox(0,0)[lt]{\lineheight{1.25}\smash{\begin{tabular}[t]{l}$P_{i-1}$\end{tabular}}}}%
    \put(0.41105542,0.02917545){\color[rgb]{0,0,0}\makebox(0,0)[lt]{\lineheight{1.25}\smash{\begin{tabular}[t]{l}$P_{i-2}$\end{tabular}}}}%
    \put(0.02268273,0.02917547){\color[rgb]{0,0,0}\makebox(0,0)[lt]{\lineheight{1.25}\smash{\begin{tabular}[t]{l}$O$\end{tabular}}}}%
    \put(0.22904027,0.02751576){\color[rgb]{0,0,0}\makebox(0,0)[lt]{\lineheight{1.25}\smash{\begin{tabular}[t]{l}$A$\end{tabular}}}}%
    \put(0.52490539,0.01398595){\color[rgb]{1,0,0}\makebox(0,0)[lt]{\lineheight{1.25}\smash{\begin{tabular}[t]{l}$l_{1,2}$\end{tabular}}}}%
    \put(0.73899098,0.01289942){\color[rgb]{1,0,0}\makebox(0,0)[lt]{\lineheight{1.25}\smash{\begin{tabular}[t]{l}$l_{1,0}$\end{tabular}}}}%
    \put(0.83528377,0.01243544){\color[rgb]{1,0,0}\makebox(0,0)[lt]{\lineheight{1.25}\smash{\begin{tabular}[t]{l}$l_{1,-1}$\end{tabular}}}}%
    \put(0.94725207,0.01243544){\color[rgb]{1,0,0}\makebox(0,0)[lt]{\lineheight{1.25}\smash{\begin{tabular}[t]{l}$l_{1,-2}$\end{tabular}}}}%
  \end{picture}%
\endgroup%

%% file: 3spiral.tex
\section{The Spirals and \texorpdfstring{$T_k$}{Tk}-Orbit Invariance} \label{sec:spiral}

In this section, we explore the geometric properties of type-$\alpha$ and type-$\beta$ $k$-spirals and prove Theorem \ref{thm:spiral polygon invariance}.
In $\S$\ref{subsec:construct Sk}, we give rigorous definitions of the two types of $k$-spirals and discuss their geometric properties. In $\S$\ref{subsec:transversals}, we introduce a construct associated to the two types of $k$-spirals called the transversals. In $\S$\ref{subsec:proof of forward invariance} and $\S$\ref{subsec:proof of backward invariance}, we prove Theorem \ref{thm:spiral polygon invariance} using geometric properties of the transversals.

\subsection{The Geometry of \texorpdfstring{$k$}{k}-Spirals} \label{subsec:construct Sk}

Here we give the formal definition of a $k$-spiral and its two subsets called type-$\alpha$ and type-$\beta$. We then explore their geometric properties and present some open problems. 

\begin{definition} \label{def:spiral polygon}
    Given integers $k \geq 2$, $n \geq 2$, we say that $[P] \in \cP_n$ is a \textit{$k$-spiral} if for all $N \in \ZZ$, there exists a representative $P$ that satisfies the following: For all $i \geq N$, $P_i$ lies in $\AA^2$, $(P_{i}, P_{i+1}, P_{i+2})$ is positive, and $(P_{i}, P_{i+1}, P_{i+k})$ is positive. Such a representative is called an \textit{$N$-representative}. Saying that $[P]$ is a $k$-spiral means that $[P]$ admits an $N$-representative for all $N \in \ZZ$. 
\end{definition}

\begin{remark} \label{rmk:why need N reps}
    The idea of considering an $N$-representative for each $N \in \ZZ$ is new to the literature and may at first seem superfluous. Readers will see in $\S$\ref{sec:partition} that this condition is natural when we examine the corner invariants of the two types of $k$-spirals. See the end of this section for open problems related to the geometry of $N$-representatives. 

    In practice, since $[P]$ is a twisted $n$-gon, it suffices to find a single $N_0$-representative $P_0$ for some $N_0 \in \ZZ$. One can then obtain other $N$-representatives for $N < N_0$ by applying the $m$-th power of the monodromy of $[P]$ to $P_0$, where $m > \frac{N_0 - N}{k} + 1$. 
\end{remark}

\begin{definition} \label{def:k spiral of type alpha and beta}
    A $k$-spiral $[P] \in \cP_n$ is of \textit{type-$\alpha$} or \textit{type-$\beta$} if for all $N \in \ZZ$, it has an $N$-representative $P$ that satisfies the following conditions:
    \begin{itemize}
        \item $[P]$ is of type-$\alpha$ if $P_{i+k} \in \Int(P_i, P_{i+1}, P_{i+k+1})$ for all $i \geq N$;
        \item $[P]$ is of type-$\beta$ if $P_{i+k+1} \in \Int(P_i, P_{i+1}, P_{i+k})$ for all $i \geq N$.
    \end{itemize}
\end{definition}

\begin{figure}[ht]
    \centering
    
    \def\svgwidth{0.6\columnwidth}
    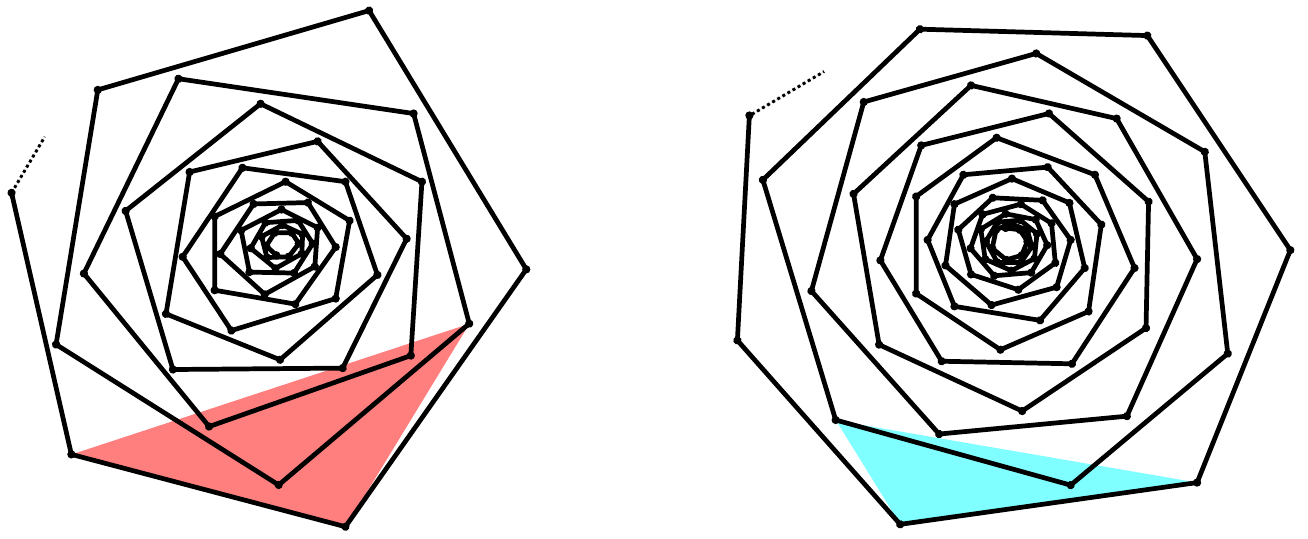

    \caption{Left: The inward half of a 0-representative $P$ of a type-$\alpha$ 6-spiral. The red triangle is joined by $(P_i, P_{i+1}, P_{i+k+1})$, which is positive by Proposition \ref{prop:type alpha k spiral orientation} and contains $P_{i+k}$ in its interior. 
    Right: The inward half of a 0-representative $P$ of a type-$\beta$ 6-spiral. The cyan triangle is joined by $(P_i, P_{i+1}, P_{i+k})$, which is positive and contains $P_{i+k+1}$ in its interior.}
    \label{fig:spiral witness}
\end{figure}

See Figure \ref{fig:spiral witness} for 0-representatives of type-$\alpha$ and type-$\beta$ 6-spirals.
For the type-$\alpha$ $k$-spirals, we show that positivity of $(P_i, P_{i+1}, P_{i+k})$ is equivalent to positivity of $(P_i, P_{i+1}, P_{i+k+1})$. The latter condition turns out to be more convenient for showing $T_k$ invariance. 

\begin{proposition} \label{prop:type alpha k spiral orientation}
    $[P] \in \cP_n$ is a type-$\alpha$ $k$-spiral if and only if for all $N \in \ZZ$, there exists a representative $P$ that satisfies the following: for all $i \geq N$, $P_i$ lies in $\AA^2$, $(P_{i}, P_{i+1}, P_{i+2})$ is positive, $(P_i, P_{i+1}, P_{i+k+1})$ is positive, and $P_{i+k} \in \Int(P_i, P_{i+1}, P_{i+k+1})$.
\end{proposition}

\begin{proof}
    Since $P_{i+k} \in \Int(P_i, P_{i+1}, P_{i+k+1})$, we see that $\Int(P_i, P_{i+1}, P_{i+k+1})$ is nonempty, so the three points $P_i$, $P_{i+1}$, $P_{i+k+1}$ are in general position. It then follows from Proposition \ref{prop:orientation lemma} that $(P_i, P_{i+1}, P_{i+k})$ is positive iff $(P_i, P_{i+1}, P_{i+k+1})$ is positive. 
\end{proof}

\begin{corollary} \label{cor:no type alpha 2 spirals}
    There exists no type-$\alpha$ 2-spirals. 
\end{corollary}

\begin{proof}
    It suffices to show that there exists no configuration of four points $A, B, C, D \in \AA^2$ such that $(A, B, D)$, $(B, C, D)$ are both positive and $C \in \Int(A, B, D)$. If $(A, B, D)$ is positive and $C \in \Int(A, B, D)$, then Proposition \ref{prop:orientation lemma} implies $(B, D, C)$ is positive, but that contradicts $(B, C, D)$ positive because $\cO(B, C, D) = -\cO(B, D, C)$. 
\end{proof}

On the other hand, type-$\beta$ 2-spirals do exist. Geometrically, their $N$-representatives look like triangular spirals. See $\S$\ref{sec:two spiral} for a more thorough discussion on type-$\beta$ 2-spirals.

\begin{remark}
    One may attempt to define the two types of $k$-spirals on bi-infinite sequences of points in $\RR\PP^2$ with no periodicity constraints. The results in this section hold true for this more general definition. We restrict our attention to twisted polygons because it's a finite-dimensional space, which allows us to more easily keep track of the $T_k$-orbits.
\end{remark}

We now proceed to discuss some geometric properties of type-$\alpha$ and type-$\beta$ $k$-spirals. 
A twisted polygon $P$ is called \textit{$k$-nice} if the four points $P_i, P_{i+1}, P_{i+k}, P_{i+k+1}$ are in general position for all $i \in \ZZ$. The $k$-nice condition is projective invariant. Let $\cP_{k,n}$ denote the space of $k$-nice twisted $n$-gons modulo projective equivalence. 

\begin{proposition} \label{prop:Pkn open in Pn}
    For all $k \geq 2$, $\cP_{k,n}$ is open in $\cP_n$, so it has dimension $2n$. 
\end{proposition}

\begin{proof}
    The condition that four points $P_i, P_{i+1}, P_{i+k}, P_{i+k+1}$ are in general position remains true if we perturb one of the points in a small enough neighborhood of $\RR\PP^2$. The dimension of $\cP_{k,n}$ comes from the fact that $\cP_n$ has dimension $2n$, which is shown in \cite[Lemma 2.2]{Ovsienko2010}.
\end{proof}

\begin{proposition} \label{prop:spirals are k-nice}
    Both type-$\alpha$ and type-$\beta$ $k$-spirals are $k$-nice.
\end{proposition}

\begin{proof}
    We give a proof to the type-$\alpha$ case. The type-$\beta$ case is analogous, so we will omit it. Given a type-$\alpha$ $k$-spiral $[P]$ and an integer $i \in \ZZ$, let $P$ be an $i$-representative of $[P]$. Since $(P_i, P_{i+1}, P_{i+k+1})$ is positive, these three points cannot be collinear. Also, since $P_{i+k} \in \Int(P_i, P_{i+1}, P_{i+k+1})$, $P_{i+k}$ does not lie in any of the lines joined by two of the three vertices $P_i, P_{i+1}, P_{i+k+1}$. This shows that $P_i, P_{i+1}, P_{i+k}, P_{i+k+1}$ are in general position. 
\end{proof}

As stated in $\S$\ref{subsec:spiral polygons}, we let $\cS_{k,n}^\alpha$ and $\cS_{k,n}^\beta$ denote the space of type-$\alpha$ and type-$\beta$ $k$-spirals (By Corollary \ref{cor:no type alpha 2 spirals}, $\cS_{2,n}^\alpha = \emptyset$ for all $n \geq 2$ ). 

\begin{proposition} \label{prop:spirals are open in Pkn}
    Both $\cS_{k,n}^\alpha$ and $\cS_{k,n}^\beta$ are open in $\cP_{k,n}$, so they both have dimension $2n$. 
\end{proposition}

\begin{proof}
    The positivity conditions of $(P_{i}, P_{i+1}, P_{i+2})$ and $(P_i, P_{i+1}, P_{i+k})$ are open conditions from continuity of the determinant function. The condition $P_{i+k} \in \Int(P_i, P_{i+1}, P_{i+k+1})$ for type-$\alpha$ (or $P_{i+k+1} \in \Int(P_i, P_{i+1}, P_{i+k})$ for type-$\beta$) is equivalent to the positivity of certain determinants by Proposition \ref{prop:orientation lemma}, so this is also an open condition. Finally, $\cS_{k,n}^\alpha \subset \cP_{k,n}$ and $\cS_{k,n}^\beta \subset \cP_{k,n}$ follows from Proposition \ref{prop:spirals are k-nice}.
\end{proof}

A twisted polygon $P$ is \textit{closed} if there exists some positive integer $n$ such that $P_{i+n} = P_i$, or $[P] \in \cP_n$ with identity monodromy. We show that neither type-$\alpha$ nor type-$\beta$ $k$-spirals are closed. 

\begin{proposition} \label{prop:k spirals are not closed}
    For all $k \geq 2$ and $n \geq 2$, if $[P] \in \cS_{k,n}^\alpha$, then $[P]$ is not closed. The same holds for $\cS_{k,n}^{\beta}$. 
\end{proposition}

\begin{proof}
    Given any closed $n$-gon $P$ on $\AA^2$, let $C$ be the convex hull of the vertices of $P$. Since $P$ has finitely many vertices, there exists a vertex $P_i$ such that $P_i \not\in \Int(C)$. Then, since $\Int(P_{i-k}, P_{i-k+1}, P_{i+1}) \subset \Int(C)$, we must have $P_i \not\in \Int(P_{i-k}, P_{i-k+1}, P_{i+1})$. It follows that $P$ is not an $N$-representative of type-$\alpha$  $k$-spiral for any $N$ or $k$. The proof for type-$\beta$ is similar, so we omit it. 
\end{proof}

The two types of $k$-spirals seem to possess rich geometric properties. We will present some open problems. In the discussion below, $[P]$ denotes a type-$\alpha$ or type-$\beta$ $k$-spiral. 

\begin{problem}
    For all $N \in \ZZ$, is it always possible to find $N$-representatives $P$ such that for all $j > i+1$, $(P_i, P_{i+1}, P_{j})$ is positive (in other words, $P_j$ always lies on the same side of the line $P_i P_{i+1}$)? 
\end{problem}

\begin{problem}
    Let $P$ be an arbitrary representative of $[P]$. Is there a minimal $N \in \ZZ$ such that $P$ is an $N$-representative on some affine patch of $\RR\PP^2$? Does there exist $P$ that is an $N$-representative for all $N \in \ZZ$?
\end{problem}

\begin{problem}
    Given an $N$-representative $P$, does $P_i$ converge to a point in $\AA^2$ as $i \rightarrow \infty$?
\end{problem}

\subsection{Transversals of the Spirals} \label{subsec:transversals}

In this section, we prove our remark in $\S$\ref{subsec:spiral polygons} that transversals for type-$\alpha$ spirals are oriented counterclockwise, whereas transversals for type-$\beta$ are oriented clockwise. Recall that the \textit{transversals} of an $N$-representative $P$ of a $k$-spiral are $k$ polygonal arcs joined by vertices $P_i, P_{i+k}, P_{i+2k}, \ldots$ for $i = N, \ldots, N+k-1$. See Figure \ref{fig:spiral witness flags} for one of the $k$ transversals of the two representatives from Figure \ref{fig:spiral witness}. 

\begin{figure}[ht]
    \centering
    
    \def\svgwidth{0.6\columnwidth}
    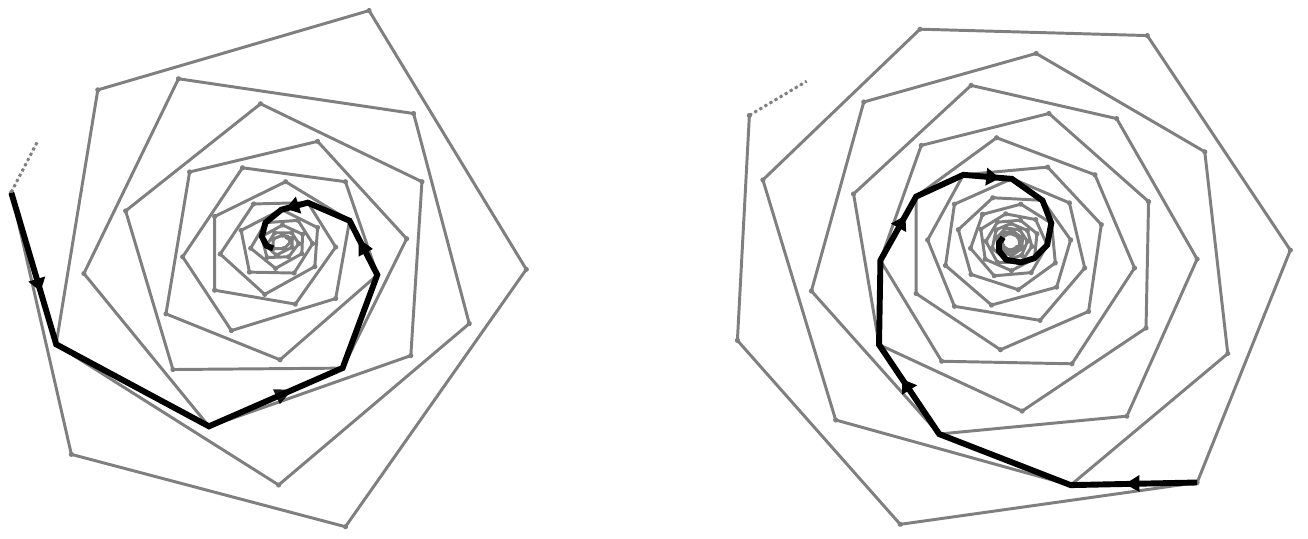

    \caption{Transversals of two representatives from Figure \ref{fig:spiral witness}.}
    \label{fig:spiral witness flags}
\end{figure}

\begin{lemma} \label{lem:orientation of pentagon}
    Given $O,A,B,C,D \in \AA^2$ \textnormal{(}See Figure \ref{fig:quadrilateral config}\textnormal{)} such that $(A,O,B)$, $(A,O,D)$, $(B,O,C)$, $(C,O,D)$ are all positive. Then, $(A,O,C)$ is positive iff $(B,O,D)$ is positive.
\end{lemma}

\begin{proof}
    For the forward direction, normalize with $\Aff_2^+(\RR)$ so that $O = (0,0)$ and $A = (-1,0)$. Let $B = (x_b, y_b)$, $C = (x_c, y_c)$, and $D = (x_d,y_d)$. Since $(A,O,B)$ is positive, Equation \eqref{eqn:vector orientations} gives us
    \begin{equation*}
        \cO(A,O,B) = \det(O - A,B - O) = \det(-A, B) = y_b > 0.
    \end{equation*}
    Similarly, positivity of $(A,O,C)$ and $(A,O,D)$ give us $y_c,y_d > 0$. Next, observe that
    \begin{equation*}
        \begin{aligned}
            \cO(B,O,C) = \det(-B, C) = -x_by_c + x_cy_b; \\
            \cO(B,O,D) = \det(-B, D) = -x_by_d + x_dy_b; \\
            \cO(C,O,D) = \det(-C, D) = -x_cy_d + x_dy_c.
        \end{aligned}
    \end{equation*}
    Since $y_b,y_c,y_d > 0$, we have $\frac{y_b}{y_c}, \frac{y_d}{y_c} > 0$, which implies 
    \begin{equation*}
        \cO(B, O, D) 
        = -x_by_d + x_d y_b 
        = \frac{y_b}{y_c} \, \cO(C,O,D) + \frac{y_d}{y_c} \, \cO(B,O,C)
        > 0.
    \end{equation*}
    This shows positivity of $(B,O,D)$.

    The proof for the backward direction is analogous. Normalize so that $O = (0,0)$ and $D = (1,0)$. Let $A = (x_a,y_a)$, $B = (x_b,y_b)$, $C = (x_c,y_c)$. Positivity of $(A,O,D)$, $(B, O, D)$, and $(C, O, D)$ implies $y_a,y_b,y_c > 0$. One can then check that 
    \begin{equation*}
        \cO(A, O, C) 
        = -x_ay_c + x_cy_a 
        = \frac{y_c}{y_b} \, \cO(A,O,B) + \frac{y_a}{y_b} \, \cO(B,O,C)
        > 0.
    \end{equation*}
    This shows positivity of $(A,O,C)$.
\end{proof}

\begin{figure}[ht]
    \centering
    \tiny
    
    \def\svgwidth{0.5\columnwidth}
    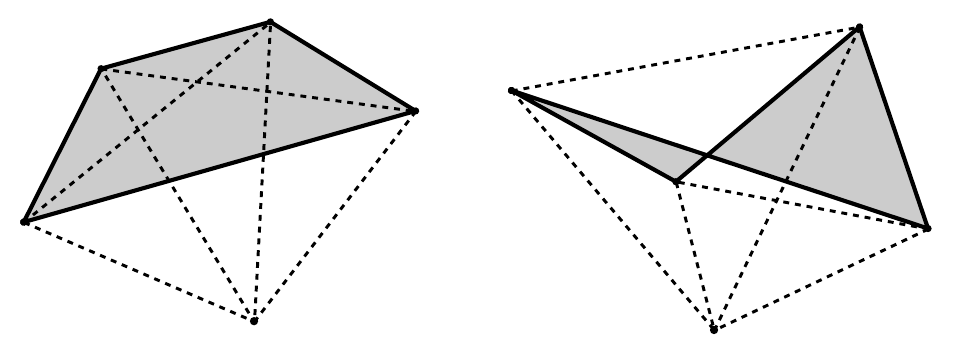

    \caption{Examples of $O, A, B, C, D$ in Lemma \ref{lem:orientation of pentagon}.}
    \label{fig:quadrilateral config}
\end{figure}

The next proposition formalizes our claim on the orientation of transversals. 

\begin{proposition} \label{prop:whirlpool}
    Let $P$ be an $N$-representative of a $k$-spiral $[P]$. For all $i > N$, if $[P]$ is type-$\alpha$, then $(P_i, P_{i+k}, P_{i+2k})$ is positive; if $[P]$ is type-$\beta$, then $(P_{i+2k}, P_{i+k}, P_{i})$ is positive.
\end{proposition}

\begin{figure}[ht]
    \centering
    \footnotesize
    
    \def\svgwidth{0.8\columnwidth}
    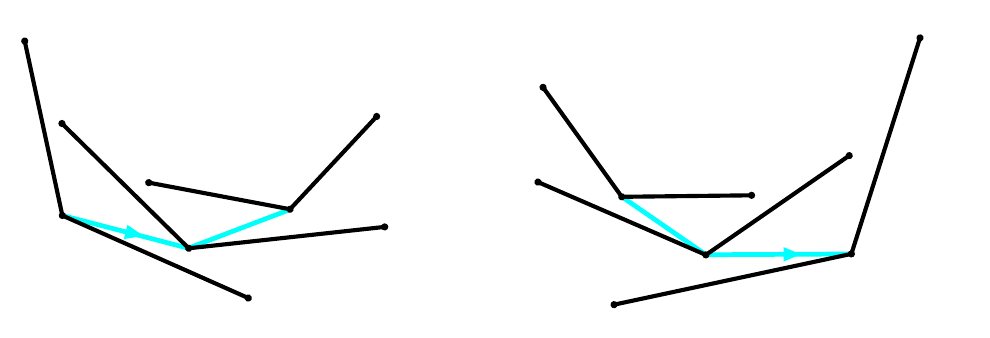

    \caption{Left: $\cS_{k,n}^\alpha$ configuration. Right: $\cS_{k,n}^\beta$ configuration.}
    \label{fig:spiral flags}
\end{figure}

\begin{proof}
    The proof applies Lemma \ref{lem:orientation of pentagon} with suitable choices of $O,A,B,C,D$. See Figure \ref{fig:spiral flags} for the configuration of points involved. 

    We start with $P$ of type-$\alpha$. 
    Consider the following choices of vertices:
    \begin{equation*} 
        O = P_{i+k}; \  
        A = P_{i}; \ 
        B = P_{i+k-1}; \ 
        C = P_{i+2k}; \ 
        D = P_{i+k+1}.
    \end{equation*}

    It follows immediately from the definition of a type-$\alpha$ $N$-representative that $(B,O,C)$ and $(B,O,D)$ are positive. The other conditions follow from applications of Proposition \ref{prop:orientation lemma}.
    Apply Proposition \ref{prop:orientation lemma} with $(P_{i-1}, P_{i}, P_{i+k})$ positive and $P_{i+k-1} \in \Int(P_{i-1}, P_{i}, P_{i+k})$ to get positivity of $(A, O, B)$.
    Apply Proposition \ref{prop:orientation lemma} with $(P_{i}, P_{i+1}, P_{i+k+1})$ positive and $P_{i+k} \in \Int(P_{i}, P_{i+1}, P_{i+k+1})$ to get positivity of $(A, O, D)$.
    Apply Proposition \ref{prop:orientation lemma} with $(P_{i+k}, P_{i+k+1}, P_{i+2k+1})$ positive and $P_{i+2k} \in \Int(P_{i+k}, P_{i+k+1}, P_{i+2k+1})$ to get positivity of $(C, O, D)$.
    Then, the backward direction of Lemma \ref{lem:orientation of pentagon} implies $(P_{i}, P_{i+k}, P_{i+2k})$ is positive.
    
    The proof for type-$\beta$ is analogous. 
    Consider the following choices of vertices:
    \begin{equation*}
        O = P_{i+k}; \  
        A = P_{i+k-1}; \ 
        B = P_{i+2k}; \ 
        C = P_{i+k+1}; \ 
        D = P_{i}.
    \end{equation*}

    Positivity of $(A,O,C)$ and $(B,O,C)$ follows from the definition of a type-$\beta$ $N$-representative. A similar application of Proposition \ref{prop:orientation lemma} as in the case of type-$\alpha$ gives positivity of $(A, O, B)$, $(A, O, D)$, and $(C, O, D)$, which we will omit.
    Finally, the forward direction of Lemma \ref{lem:orientation of pentagon} implies $(P_{i+2k}, P_{i+k}, P_{i})$ is positive.
\end{proof}

\subsection{Invariance of Forward Orbit} \label{subsec:proof of forward invariance}

In this section, we prove that $\cS_{k,n}^\alpha$ and $\cS_{k,n}^\beta$ are $T_k$-invariant. We will use Equation \eqref{eqn:delta k,1 formula} for our labeling convention. See Figure \ref{fig:twisted k nice}. 

\begin{figure}[ht]
    \centering
    \footnotesize
    
    \def\svgwidth{0.5\columnwidth}
    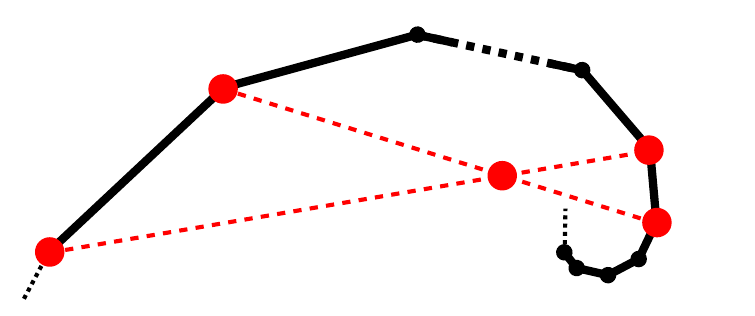

    \caption{The labeling convention of the map $T_k$ from Equation \eqref{eqn:delta k,1 formula}.}
    \label{fig:twisted k nice}
\end{figure}

If $P$ is $k$-nice, then $P'$ is always well-defined. In particular, Proposition \ref{prop:spirals are k-nice} implies $T_k$ is well-defined on $\cS_{k,n}^\alpha$ and $\cS_{k,n}^\beta$. 

\begin{remark} \label{rmk:k-nice and Tk}
    $T_k$ doesn't necessarily send $k$-nice twisted polygons to $k$-nice twisted polygons. Here is an example provided by the anonymous referee: Fix $r \in (0,1)$. Consider the function $P: \ZZ \rightarrow \CC \cong \RR^2$ mapping $z \mapsto r^z \exp (z \pi i / k)$. One can check that $P$ is a $k$-nice twisted $n$-gon for any $n \geq 2$ with monodromy that is a scale-rotation, but $T_k(P)$ is the zero function and hence not $k$-nice. What we will show is that in the case of type-$\alpha$ and type-$\beta$ $k$-spirals, $T_k$ does preserve $k$-niceness. This is a direct consequence of Theorem \ref{thm:spiral polygon invariance} and Proposition \ref{prop:spirals are k-nice}. 
\end{remark}

We proceed to prove the $T_k$-invariance of $\cS_{k,n}^\alpha$ and $\cS_{k,n}^\beta$ separately. We start with the following lemma.

\begin{lemma} \label{lem:triangle lemma}
    Given four points $A, B, C, D$ in $\RR^2$ in general position with $D \in \Int(A,B,C)$. Let $O = AB \cap CD$. There exist $s \in (0,1)$ and $t \in (1,\infty)$ such that 
    \begin{equation*}
        O = (1-s)A + sB = (1-t)C + tD .
    \end{equation*}
\end{lemma}

\begin{proof}
    Since $D \in \Int(A,B,C)$, there exists $\lambda_1, \lambda_2, \lambda_3 \in (0,1)$ such that 
    \begin{equation*}
        \lambda_1+\lambda_2+\lambda_3 = 1; \ \ 
        D = \lambda_1 A + \lambda_2 B + \lambda_3 C.
    \end{equation*}
    Taking $s = \frac{\lambda_2}{1 - \lambda_3}$ and $t = \frac{1}{1 - \lambda_3}$ gives us the desired result.
\end{proof}

\begin{proposition} \label{prop:Sk alpha invariance}
    For all $k \geq 2$ and $n \geq 2$, $T_k(\cS_{k,n}^\alpha) \subset \cS_{k,n}^\alpha$. 
\end{proposition}

\begin{proof}
    Given an $N$-representative $P$ of some $[P] \in \cS_{k,n}^\alpha$, we will show that $P' = T_k(P)$ is a type-$\alpha$ $N$-representative of $[T_k(P)]$ by proving that for all $i \geq N$, $(P'_{i}, P'_{i+1}, P'_{i+2})$ is positive, $(P'_{i}, P'_{i+1}, P'_{i+k+1})$ is positive, and $P'_{i+k} \in \Int (P'_i, P'_{i+1}, P'_{i+k+1})$. See the left side of Figure \ref{fig:forward invariance} for configurations of relevant vertices of $P$ and $P'$. 

    Let $i \geq N$ be fixed. 
    Since $P$ is a type-$\alpha$ $N$-representative, $P_{j+k} \in \Int(P_{j}, P_{j+1}, P_{j+k+1})$ for all $j \geq N$. Applying Lemma \ref{lem:triangle lemma} with Equation \eqref{eqn:delta k,1 formula} on $P'_{j}$ for $j \in \{i,i+1,i+2,i+k,i+k+1\}$ gives us 
    \begin{equation} \label{eqn:Sk alpha invariance}
        \begin{aligned}
            & P'_{i} = (1-s_1) P_{i+1} + s_1 P_{i+k+1}; &
            & P'_{i+1} = (1-t_1) P_{i+1} + t_1 P_{i+k+1}; \\
            & P'_{i+1} = (1-s_2) P_{i+2} + s_2 P_{i+k+2}; &
            & P'_{i+2} = (1-t_2) P_{i+2} + t_2 P_{i+k+2}; \\
            & P'_{i+k} = (1-s_3) P_{i+k+1} + s_3 P_{i+2k+1}; &
            & P'_{i+k+1} = (1-t_3) P_{i+k+1} + t_3 P_{i+2k+1},
        \end{aligned}
    \end{equation}
    where $s_1,s_2,s_3 \in (0,1)$ and $t_1,t_2,t_3 \in (1,\infty)$. In particular, this shows $P'_{i+k+1} \not\in P'_i P'_{i+1}$, so the three points $P'_{i}, P'_{i+1}, P'_{i+k+1}$ are in general position. 
    
    To see that $(P'_{i}, P'_{i+1}, P'_{i+2})$ is positive, Equation \eqref{eqn:vector orientations} and \eqref{eqn:Sk alpha invariance} give us  
    \begin{equation} \label{eqn:Sk alpha invariance 2}
        \begin{aligned}
            \cO(P'_{i}, P'_{i+1}, P'_{i+2})
            &= \det(P'_{i+1} - P'_{i}, P'_{i+2} - P'_{i+1}) \\
            &= \det ((s_1 - t_1) P_{i+1} + (t_1 - s_1) P_{i+k+1}, (s_2 - t_2) P_{i+2} + (t_2 - s_2) P_{i+k+2}) \\
            &= (t_1 - s_1)(t_2 - s_2) \det(P_{i+k+2} - P_{i+2}, P_{i+1} - P_{i+k+1}).
        \end{aligned}
    \end{equation}
    Then, since $\cO(P_{i+1}, P_{i+2}, P_{i+k+2}) > 0$ and $P_{i+k+1} \in \Int (P_{i+1}, P_{i+2}, P_{i+k+2})$, Proposition \ref{prop:orientation lemma} implies $\det(P_{i+k+2} - P_{i+2}, P_{i+1} - P_{i+k+1}) > 0$, so $\cO(P'_{i}, P'_{i+1}, P'_{i+2}) > 0$. 

    Next, we show that $P'_{i+k} \in \Int(P'_{i}, P'_{i+1}, P'_{i+k+1})$. 
    Let $r_1 = \frac{1 - s_1}{t_1 - s_1}$ and $r_2 = \frac{s_3}{t_3}$. \eqref{eqn:Sk alpha invariance} implies $r_1, r_2 \in (0,1)$ and 
    \begin{equation*}
        \begin{aligned}
            P'_{i+k} 
            &= (1 - s_3) P_{i+k+1} + s_3 P_{i+2k+1} \\
            &= \frac{(1 - s_3)(t_1 - 1)}{t_1 - s_1} P'_{i} + \frac{(1 - s_3)(1 - s_1)}{t_1 - s_1} P'_{i+1} + \frac{s_3(t_3 - 1)}{t_3 - s_3} P'_{i+k} + \frac{s_3(1 - s_3)}{t_3 - s_3} P'_{i+k+1}.
        \end{aligned}
    \end{equation*}
    It follows that 
    \begin{equation*}
        \begin{aligned}
            P'_{i+k}
            &= \frac{t_3 - s_3}{t_3(s_3 - 1)} \left( \frac{(1 - s_3)(t_1 - 1)}{t_1 - s_1} P'_{i} + \frac{(1 - s_3)(1 - s_1)}{t_1 - s_1} P'_{i+1} + \frac{s_3(1 - s_3)}{t_3 - s_3} P'_{i+k+1} \right) \\
            &= \frac{(t_3 - s_3)(1 - t_1)}{t_3(t_1 - s_1)} P'_i + \frac{(t_3 - s_3)(s_1 - 1)}{t_3(t_1 - s_1)} P'_{i+1} + \frac{s_3}{t_3} P'_{i+k+1} \\
            &= (1 - r_2)(1 - r_1) P'_{i} + (1 - r_2)r_1 P'_{i+1} + r_2 P'_{i+k+1}.
        \end{aligned}
    \end{equation*}
    Observe that the coefficients $(1-r_2)(1-r_1)$, $(1-r_2)r_1$, $r_2$ are all in $(0,1)$ and sum up to $1$, so $P'_{i+k} \in \Int(P'_{i}, P'_{i+1}, P'_{i+k+1})$.
    
    Finally, using Equation \eqref{eqn:vector orientations} and \eqref{eqn:Sk alpha invariance}, we have 
    \begin{equation} \label{eqn:Sk alpha invariance 3}
        \begin{aligned}
            \det(P'_{i+1} - P'_i, P'_{i+k+1} - P'_{i+k}) 
            &= \det((t_1 - s_1) (P_{i+k+1} - P_{i+1}), (t_3 - s_3) (P_{i+2k+1} - P_{i+k+1})) \\
            &= (t_1 - s_1)(t_3 - s_3) \det(P_{i+k+1} - P_{i+1}, P_{i+2k+1} - P_{i+k+1}) \\
            &= (t_1 - s_1)(t_3 - s_3) \cO(P_{i+1}, P_{i+k+1}, P_{i+2k+1}).
        \end{aligned}
    \end{equation}
    Proposition \ref{prop:whirlpool} implies $\cO(P_{i+1}, P_{i+k+1}, P_{i+2k+1}) > 0$, so $\det(P'_{i+1} - P'_i, P'_{i+k+1} - P'_{i+k})  > 0$ Since $P'_{i}, P'_{i+1}, P'_{i+k+1}$ are in general position and $P'_{i+k} \in \Int(P'_{i}, P'_{i+1}, P'_{i+k+1})$, Proposition \ref{prop:orientation lemma} and Equation \eqref{eqn:Sk alpha invariance 3} imply $\cO(P'_{i}, P'_{i+1}, P'_{i+k+1}) > 0$. We conclude that $P'$ is a type-$\alpha$ $N$-representative. 
\end{proof}

\begin{figure}[ht]
    \centering
    \scriptsize
    
    \def\svgwidth{\columnwidth}
    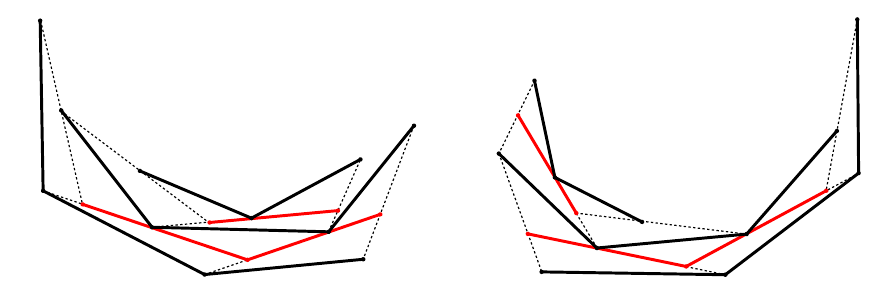

    \caption{Left: Proposition \ref{prop:Sk alpha invariance} configuration. Right: Proposition \ref{prop:Sk beta invariance} configuration.}
    \label{fig:forward invariance}
\end{figure}

\begin{proposition} \label{prop:Sk beta invariance}
    For all $k \geq 2$ and $n \geq 2$, $T_k(\cS_{k,n}^{\beta}) \subset \cS_{k,n}^{\beta}$. 
\end{proposition}

\begin{proof}
    The proof is analogous to the one for Proposition \ref{prop:Sk alpha invariance}. Replacing $\alpha$ with $\beta$, we may work with the setup in the proof of Proposition \ref{prop:Sk alpha invariance}. See the right side of Figure \ref{fig:forward invariance}. 
    
    The key difference between type-$\alpha$ and type-$\beta$ is that conditions for type-$\beta$ $k$-spirals give us the following linear relations when we apply Lemma \ref{lem:triangle lemma} with \eqref{eqn:delta k,1 formula} on $P'_{j}$ for $j \in \{i,i+1,i+2,i+k,i+k+1\}$:
    \begin{equation} \label{eqn:Sk beta invariance}
        \begin{aligned}
            & P'_{i} = (1-t_1) P_{i+1} + t_1 P_{i+k+1}; &
            & P'_{i+1} = (1-s_1) P_{i+1} + s_1 P_{i+k+1}; \\
            & P'_{i+1} = (1-t_2) P_{i+2} + t_2 P_{i+k+2}; &
            & P'_{i+2} = (1-s_2) P_{i+2} + s_2 P_{i+k+2}; \\ 
            & P'_{i+k} = (1-t_3) P_{i+k+1} + t_3 P_{i+2k+1}; &
            & P'_{i+k+1} = (1-s_3) P_{i+k+1} + s_3 P_{i+2k+1},
        \end{aligned}
    \end{equation}
    where $s_1,s_2,s_3 \in (0,1)$ and $t_1,t_2,t_3 \in (1,\infty)$. We can see that $P'_{i+k} \not\in P'_i P'_{i+1}$, so the three points $P'_i, P'_{i+1}, P'_{i+k}$ are in general position. 

    A very similar computation as Equation \eqref{eqn:Sk alpha invariance 2} shows positivity of $(P'_i, P'_{i+1}, P'_{i+2})$, so we will omit it. 
    Next, let $r_1 = \frac{t_1 - 1}{t_1 - s_1}$ and $r_2 = \frac{s_3}{t_3}$. Notice that $(1-r_2)(1-r_1)$, $(1-r_2)r_1$, and $r_2$ are all in $(0,1)$ and sum up to $1$. Also, Equation \eqref{eqn:Sk beta invariance} implies 
    \begin{equation*}
        P'_{i+k+1} 
        = (1 - r_2)(1 - r_1) P'_{i} + (1 - r_2)r_1 P'_{i+1} + r_2 P'_{i+k}.
    \end{equation*}
    This shows $P'_{i+k+1} \in \Int(P'_{i}, P'_{i+1}, P'_{i+k})$. Finally, positivity of $(P'_{i}. P'_{i+1}, P'_{i+k})$ follows from a similar computation as Equation \eqref{eqn:Sk alpha invariance 3}, $P'_{i+k+1} \in \Int(P'_{i}, P'_{i+1}, P'_{i+k})$, the points $P'_{i}, P'_{i+1}, P'_{i+k}$ are in general position, and Proposition \ref{prop:orientation lemma}. 
\end{proof}

\subsection{Invariance of Backward Orbit} \label{subsec:proof of backward invariance}

In this section, we complete the proof of Theorem \ref{thm:spiral polygon invariance} by showing that $\cS_{k,n}^\alpha$ and $\cS_{k,n}^\beta$ are $T_k^{-1}$-invariant. One can derive a formula for $T_k^{-1}$ from Equation \eqref{eqn:delta k,1 formula}. Given any $k$-nice twisted $n$-gon $P'$, $P = T_k^{-1}(P')$ is given by 
\begin{equation} \label{eqn:Tk inverse formula}
    P_i = P'_{i-k-1} P'_{i-k} \cap P'_{i-1} P'_{i}.
\end{equation}

Proposition \ref{prop:spirals are k-nice} implies $T_k^{-1}$ is well-defined on $\cS_{k,n}^\alpha$ and $\cS_{k,n}^\beta$. In general, $T_k^{-1}$ needs not preserve $k$-niceness of twisted polygons. 

\begin{proposition} \label{prop:Sk alpha inverse invariance}
    For all $k \geq 2$ and $n \geq 2$, $T_k^{-1}(S_{k,n}^\alpha) \subset S_{k,n}^\alpha$.
\end{proposition}

\begin{proof}
    Given $P'$ a type-$\alpha$ $N$-representative, we will show that $P = T_k^{-1}(P')$ is a type-$\alpha$ $(N+k+1)$-representative by proving that for all $i \geq N+k+1$, $(P_i, P_{i+1}, P_{i+2})$ is positive, $(P_i, P_{i+1}, P_{i+k+1})$ is positive, the four points $P_i, P_{i+1}, P_{i+k}, P_{i+k+1}$ are in general position, and $P_{i+k} \in \Int(P_i, P_{i+1}, P_{i+k+1})$. See the left side of Figure \ref{fig:backward invariance} for configurations of relevant vertices of $P'$ and $P$. 

    Let $i \geq N+k+1$ be fixed. Since $P'$ is a type-$\alpha$ $N$-representative, we must have $P'_{j+k} \in \Int(P'_j, P'_{j+1}, P'_{j+k+1})$ for all $j \geq N$. Applying Lemma \ref{lem:triangle lemma} with Equation \eqref{eqn:Tk inverse formula} on $P_j$ for $j \in \{i, i+1, i+2, i+k, i+k+1\}$ gives us 
    \begin{equation} \label{eqn:Sk alpha inverse formula}
        \begin{aligned}
            & P_i = (1 - s_1) P'_{i-k} + s_1 P'_{i-k-1}; &
            & P_i = (1 - t_1) P'_{i} + t_1 P'_{i-1}; \\
            & P_{i+1} = (1 - s_2) P'_{i-k+1} + s_2 P'_{i-k}; &
            & P_{i+1} = (1 - t_2) P'_{i+1} + t_3 P'_{i}; \\
            & P_{i+2} = (1 - s_3) P'_{i-k+2} + s_3 P'_{i-k+1}; &
            & P_{i+k} = (1 - s_4) P'_i + s_4 P'_{i-1}; \\
            & P_{i+k+1} = (1 - s_5) P'_{i+1} + s_5 P'_i, & &
        \end{aligned}
    \end{equation}
    where $s_1, s_2, s_3, s_4, s_5 \in (0,1)$ and $t_1, t_2 \in (1, \infty)$. 

    We first show that $(P_i, P_{i+1}, P_{i+k+1})$ is positive. From Equation \eqref{eqn:Sk alpha inverse formula} we have 
    \begin{equation*} \label{eqn:Sk alpha inverse computation}
        \cO(P_i, P_{i+1}, P_{i+k+1})
        = (t_1t_2(1 - s_5) - t_1(1 - t_2)s_5) \cO(P'_{i-1}, P'_i, P'_{i+1}).
    \end{equation*}
    It follows that $\cO(P_i, P_{i+1}, P_{i+k+1}) > 0$, so $(P_i, P_{i+1}, P_{i+k+1})$ is positive. 

    Next, we show that $P_{i+k} \in \Int(P_i, P_{i+1}, P_{i+k+1})$. Let $r_1 = \frac{t_2 - 1}{t_2 - s_5}$ and $r_2 = \frac{s_4}{t_1}$. Equation \eqref{eqn:Sk alpha inverse formula} implies $r_1, r_2 \in (0,1)$ and 
    \begin{equation*}
        P_{i+k} = (1 - r_2)(1 - r_1) P_{i+1} + (1 - r_2)r_1 P_{i+k+1} + r_2 P_{i}.
    \end{equation*}
    Observe that the coefficients $(1 - r_2)(1 - r_1)$, $(1 - r_2)r_1$, and $r_2$ are all in $(0,1)$ and sum up to 1, so $P_{i+k} \in \Int(P_i, P_{i+1}, P_{i+k+1})$. 

    Finally, we check $(P_i, P_{i+1}, P_{i+2})$ is positive. We aim to invoke Lemma \ref{lem:orientation of pentagon} with the following choices of vertices:
    \begin{equation} \label{eqn:Sk alpha inverse choice of vertices}
        O = P_{i+1}; \ \ 
        A = P_{i}; \ \ 
        B = P_{i+k+1}; \ \ 
        C = P_{i+2}; \ \ 
        D = P'_{i-k+1}.
    \end{equation}
    Positivity of $(A, O, B)$ is a direct consequence of the above argument. Positivity of $(B, O, C)$ follows from positivity of $(P_{i+1}, P_{i+2}, P_{i+k+2})$, $P_{i+k+1} \in \Int(P_{i+1}, P_{i+2}, P_{i+k+2})$, and Proposition \ref{prop:orientation lemma}. Next, observe that 
    \begin{equation} \label{eqn:Sk alpha inverse positivity}
        \begin{aligned}
            & \cO(A, O, D)
            = s_1 s_2 \cO(P'_{i-k-1}, P'_{i-k}, P'_{i-k+1}); \\
            & \cO(C, O, D)
            = (1 - s_3) s_2 \cO(P'_{i-k}, P'_{i-k+1}, P'_{i-k+2}); \\
            & \cO(B, O, D)
            = s_2(1 - s_5) \cO(P'_{i-k}, P'_{i-k+1}, P'_{i+1}) + s_2s_5 \cO(P'_{i-k}, P'_{i-k+1}, P'_{i}); 
        \end{aligned}
    \end{equation}
    Then, positivity of $(A, O, D)$ and $(C, O, D)$ follows from positivity of $(P'_{i-k-1}, P'_{i-k}, P'_{i-k+1})$ and $(P'_{i-k}, P'_{i-k+1}, P'_{i-k+2})$. To see that $(B, O, D)$ is positive, apply Proposition \ref{prop:orientation lemma} on $(P'_{i-k}, P'_{i-k+1}, P'_{i+1})$ positive and $P'_{i} \in \Int(P'_{i-k}, P'_{i-k+1}, P'_{i+1})$ to get $(P'_{i-k}, P'_{i-k+1}, P'_i)$ positive. The backward direction of Lemma \ref{lem:orientation of pentagon} then implies $(P_i, P_{i+1}, P_{i+2})$ is positive. We conclude that $P$ is a type-$\alpha$ $(N+k+1)$-representative. 
\end{proof}

\begin{figure}[ht]
    \centering
    \scriptsize
    
    \def\svgwidth{0.85\columnwidth}
    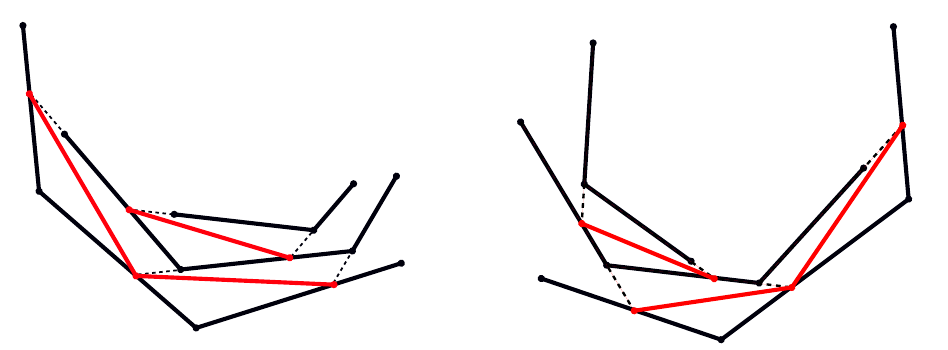

    \caption{Left: Proposition \ref{prop:Sk alpha inverse invariance} configuration. Right: Proposition \ref{prop:Sk beta inverse invariance} configuration.}
    \label{fig:backward invariance}
\end{figure}

\begin{proposition} \label{prop:Sk beta inverse invariance}
    For all $k \geq 2$ and $n \geq 2$, $T_k^{-1}(S_{k,n}^\beta) \subset S_{k,n}^\beta$.
\end{proposition}

\begin{proof}
    The proof is similar to that of Lemma \ref{prop:Sk alpha inverse invariance} (See right side of Figure \ref{fig:backward invariance}). We will point out some key differences. Replacing $\alpha$ with $\beta$, we may work with the setup in the proof of Proposition \ref{prop:Sk alpha inverse invariance}. Applying Lemma \ref{lem:triangle lemma} with \eqref{eqn:Tk inverse formula} on $P_j$ for $j \in \{i, i+1, i+2, i+k, i+k+1\}$ gives us 
    \begin{equation} \label{eqn:Sk beta inverse formula}
        \begin{aligned}
            & P_i = (1 - s_1) P'_{i-k} + s_1 P'_{i-k-1}; &
            & P_i = (1 - t_1) P'_{i-1} + t_1 P'_{i}; \\
            & P_{i+1} = (1 - s_2) P'_{i-k+1} + s_2 P'_{i-k}; &
            & P_{i+1} = (1 - t_2) P'_{i} + t_3 P'_{i+1}; \\
            & P_{i+2} = (1 - s_3) P'_{i-k+2} + s_3 P'_{i-k+1}; &
            & P_{i+k} = (1 - s_4) P'_{i-1} + s_4 P'_{i}; \\
            & P_{i+k+1} = (1 - s_5) P'_{i} + s_5 P'_{i+1}, & &
        \end{aligned}
    \end{equation}
    where $s_1, s_2, s_3, s_4, s_5 \in (0,1)$ and $t_1, t_2 \in (1, \infty)$. 
    Positivity of $(P_i, P_{i+1}, P_{i+k})$ follows from a similar computation as in \eqref{eqn:Sk alpha inverse computation}. 
    Next, let $r_1 = \frac{1 - s_4}{t_1 - s_4}$ and $r_2 = \frac{s_5}{t_2}$. Equation \eqref{eqn:Sk beta inverse formula} implies
    \begin{equation*}
        P_{i+k+1} = (1 - r_2)(1 - r_1) P_{i+k} + (1 - r_2)r_1 P_{i} + r_2 P_{i+1}.
    \end{equation*}
    Observe that the coefficients $(1 - r_2)(1 - r_1)$, $(1 - r_2)r_1$, and $r_2$ are all in $(0,1)$ and sum up to 1, so $P_{i+k+1} \in \Int(P_i, P_{i+1}, P_{i+k})$. 

    Finally, assign $O,A,B,C,D$ to be the same vertices as in \eqref{eqn:Sk alpha inverse choice of vertices}. Positivity of $(A, O, B)$, $(B, O, C)$, $(C, O, D)$, $(A, O, D)$, and $(B, O, D)$ follows from a very similar proof as that of Proposition \ref{prop:Sk alpha inverse invariance}, with \eqref{eqn:Sk alpha inverse positivity} replaced by 
    \begin{equation*}
        \begin{aligned}
            & \cO(A, O, D)
            = s_1 s_2 \cO(P'_{i-k-1}, P'_{i-k}, P'_{i-k+1}); \\
            & \cO(C, O, D)
            = (1 - s_3) s_2 \cO(P'_{i-k}, P'_{i-k+1}, P'_{i-k+2}); \\
            & \cO(B, O, D)
            = s_2 (1 - s_5) \cO(P'_{i-k}, P'_{i-k+1}, P'_{i}) + s_2 s_5 \cO(P'_{i-k}, P'_{i-k+1}, P'_{i+1}). \\
        \end{aligned}
    \end{equation*}
    The backward direction of Proposition \ref{lem:orientation of pentagon} then implies $(P_i, P_{i+1}, P_{i+2})$ is positive. 
\end{proof}

We conclude this section by stating that Proposition \ref{prop:Sk alpha invariance}, \ref{prop:Sk beta invariance}, \ref{prop:Sk alpha inverse invariance}, \ref{prop:Sk beta inverse invariance} together prove Theorem \ref{thm:spiral polygon invariance}.

%% file: figure/spiral_witness.pdf_tex
\begingroup%
  \makeatletter%
  \providecommand\color[2][]{%
    \errmessage{(Inkscape) Color is used for the text in Inkscape, but the package 'color.sty' is not loaded}%
    \renewcommand\color[2][]{}%
  }%
  \providecommand\transparent[1]{%
    \errmessage{(Inkscape) Transparency is used (non-zero) for the text in Inkscape, but the package 'transparent.sty' is not loaded}%
    \renewcommand\transparent[1]{}%
  }%
  \providecommand\rotatebox[2]{#2}%
  \newcommand*\fsize{\dimexpr\f@size pt\relax}%
  \newcommand*\lineheight[1]{\fontsize{\fsize}{#1\fsize}\selectfont}%
  \ifx\svgwidth\undefined%
    \setlength{\unitlength}{627.49079186bp}%
    \ifx\svgscale\undefined%
      \relax%
    \else%
      \setlength{\unitlength}{\unitlength * \real{\svgscale}}%
    \fi%
  \else%
    \setlength{\unitlength}{\svgwidth}%
  \fi%
  \global\let\svgwidth\undefined%
  \global\let\svgscale\undefined%
  \makeatother%
  \begin{picture}(1,0.41228989)%
    \lineheight{1}%
    \setlength\tabcolsep{0pt}%
    \put(0,0){\includegraphics[width=\unitlength,page=1]{spiral_witness.pdf}}%
  \end{picture}%
\endgroup%

%% file: figure/spiral_witness_flags.pdf_tex
\begingroup%
  \makeatletter%
  \providecommand\color[2][]{%
    \errmessage{(Inkscape) Color is used for the text in Inkscape, but the package 'color.sty' is not loaded}%
    \renewcommand\color[2][]{}%
  }%
  \providecommand\transparent[1]{%
    \errmessage{(Inkscape) Transparency is used (non-zero) for the text in Inkscape, but the package 'transparent.sty' is not loaded}%
    \renewcommand\transparent[1]{}%
  }%
  \providecommand\rotatebox[2]{#2}%
  \newcommand*\fsize{\dimexpr\f@size pt\relax}%
  \newcommand*\lineheight[1]{\fontsize{\fsize}{#1\fsize}\selectfont}%
  \ifx\svgwidth\undefined%
    \setlength{\unitlength}{627.49079186bp}%
    \ifx\svgscale\undefined%
      \relax%
    \else%
      \setlength{\unitlength}{\unitlength * \real{\svgscale}}%
    \fi%
  \else%
    \setlength{\unitlength}{\svgwidth}%
  \fi%
  \global\let\svgwidth\undefined%
  \global\let\svgscale\undefined%
  \makeatother%
  \begin{picture}(1,0.41228989)%
    \lineheight{1}%
    \setlength\tabcolsep{0pt}%
    \put(0,0){\includegraphics[width=\unitlength,page=1]{spiral_witness_flags.pdf}}%
  \end{picture}%
\endgroup%

%% file: figure/quadrilateral_config.pdf_tex
\begingroup%
  \makeatletter%
  \providecommand\color[2][]{%
    \errmessage{(Inkscape) Color is used for the text in Inkscape, but the package 'color.sty' is not loaded}%
    \renewcommand\color[2][]{}%
  }%
  \providecommand\transparent[1]{%
    \errmessage{(Inkscape) Transparency is used (non-zero) for the text in Inkscape, but the package 'transparent.sty' is not loaded}%
    \renewcommand\transparent[1]{}%
  }%
  \providecommand\rotatebox[2]{#2}%
  \newcommand*\fsize{\dimexpr\f@size pt\relax}%
  \newcommand*\lineheight[1]{\fontsize{\fsize}{#1\fsize}\selectfont}%
  \ifx\svgwidth\undefined%
    \setlength{\unitlength}{463.38061235bp}%
    \ifx\svgscale\undefined%
      \relax%
    \else%
      \setlength{\unitlength}{\unitlength * \real{\svgscale}}%
    \fi%
  \else%
    \setlength{\unitlength}{\svgwidth}%
  \fi%
  \global\let\svgwidth\undefined%
  \global\let\svgscale\undefined%
  \makeatother%
  \begin{picture}(1,0.36578735)%
    \lineheight{1}%
    \setlength\tabcolsep{0pt}%
    \put(0,0){\includegraphics[width=\unitlength,page=1]{quadrilateral_config.pdf}}%
    \put(0.24811429,0.00449976){\color[rgb]{0,0,0}\makebox(0,0)[lt]{\lineheight{1.25}\smash{\begin{tabular}[t]{l}$O$\end{tabular}}}}%
    \put(-0.00430684,0.11873632){\color[rgb]{0,0,0}\makebox(0,0)[lt]{\lineheight{1.25}\smash{\begin{tabular}[t]{l}$A$\end{tabular}}}}%
    \put(0.081891,0.30344359){\color[rgb]{0,0,0}\makebox(0,0)[lt]{\lineheight{1.25}\smash{\begin{tabular}[t]{l}$B$\end{tabular}}}}%
    \put(0.27117407,0.35439621){\color[rgb]{0,0,0}\makebox(0,0)[lt]{\lineheight{1.25}\smash{\begin{tabular}[t]{l}$C$\end{tabular}}}}%
    \put(0.44216132,0.24910258){\color[rgb]{0,0,0}\makebox(0,0)[lt]{\lineheight{1.25}\smash{\begin{tabular}[t]{l}$D$\end{tabular}}}}%
    \put(0.72657288,-0.00333285){\color[rgb]{0,0,0}\makebox(0,0)[lt]{\lineheight{1.25}\smash{\begin{tabular}[t]{l}$O$\end{tabular}}}}%
    \put(0.4963285,0.27731023){\color[rgb]{0,0,0}\makebox(0,0)[lt]{\lineheight{1.25}\smash{\begin{tabular}[t]{l}$A$\end{tabular}}}}%
    \put(0.66904966,0.1574892){\color[rgb]{0,0,0}\makebox(0,0)[lt]{\lineheight{1.25}\smash{\begin{tabular}[t]{l}$B$\end{tabular}}}}%
    \put(0.88955605,0.34906023){\color[rgb]{0,0,0}\makebox(0,0)[lt]{\lineheight{1.25}\smash{\begin{tabular}[t]{l}$C$\end{tabular}}}}%
    \put(0.97172473,0.12407228){\color[rgb]{0,0,0}\makebox(0,0)[lt]{\lineheight{1.25}\smash{\begin{tabular}[t]{l}$D$\end{tabular}}}}%
  \end{picture}%
\endgroup%

%% file: figure/spiral_flags.pdf_tex
\begingroup%
  \makeatletter%
  \providecommand\color[2][]{%
    \errmessage{(Inkscape) Color is used for the text in Inkscape, but the package 'color.sty' is not loaded}%
    \renewcommand\color[2][]{}%
  }%
  \providecommand\transparent[1]{%
    \errmessage{(Inkscape) Transparency is used (non-zero) for the text in Inkscape, but the package 'transparent.sty' is not loaded}%
    \renewcommand\transparent[1]{}%
  }%
  \providecommand\rotatebox[2]{#2}%
  \newcommand*\fsize{\dimexpr\f@size pt\relax}%
  \newcommand*\lineheight[1]{\fontsize{\fsize}{#1\fsize}\selectfont}%
  \ifx\svgwidth\undefined%
    \setlength{\unitlength}{475.08061927bp}%
    \ifx\svgscale\undefined%
      \relax%
    \else%
      \setlength{\unitlength}{\unitlength * \real{\svgscale}}%
    \fi%
  \else%
    \setlength{\unitlength}{\svgwidth}%
  \fi%
  \global\let\svgwidth\undefined%
  \global\let\svgscale\undefined%
  \makeatother%
  \begin{picture}(1,0.34107295)%
    \lineheight{1}%
    \setlength\tabcolsep{0pt}%
    \put(0,0){\includegraphics[width=\unitlength,page=1]{spiral_flags.pdf}}%
    \put(0.24386469,0.01015631){\color[rgb]{0,0,0}\makebox(0,0)[lt]{\lineheight{1.25}\smash{\begin{tabular}[t]{l}$P_{i+1}$\end{tabular}}}}%
    \put(0.01236135,0.31107196){\color[rgb]{0,0,0}\makebox(0,0)[lt]{\lineheight{1.25}\smash{\begin{tabular}[t]{l}$P_{i-1}$\end{tabular}}}}%
    \put(0.05036688,0.22416606){\color[rgb]{0,0,0}\makebox(0,0)[lt]{\lineheight{1.25}\smash{\begin{tabular}[t]{l}$P_{i+k-1}$\end{tabular}}}}%
    \put(0.1952738,0.06652723){\color[rgb]{0,0,0}\makebox(0,0)[lt]{\lineheight{1.25}\smash{\begin{tabular}[t]{l}$P_{i+k}$\end{tabular}}}}%
    \put(0.37937833,0.08212586){\color[rgb]{0,0,0}\makebox(0,0)[lt]{\lineheight{1.25}\smash{\begin{tabular}[t]{l}$P_{i+k+1}$\end{tabular}}}}%
    \put(0.13513951,0.16613971){\color[rgb]{0,0,0}\makebox(0,0)[lt]{\lineheight{1.25}\smash{\begin{tabular}[t]{l}$P_{i+2k-1}$\end{tabular}}}}%
    \put(0.30441146,0.12242691){\color[rgb]{0,0,0}\makebox(0,0)[lt]{\lineheight{1.25}\smash{\begin{tabular}[t]{l}$P_{i+2k}$\end{tabular}}}}%
    \put(0.37976445,0.19935428){\color[rgb]{0,0,0}\makebox(0,0)[lt]{\lineheight{1.25}\smash{\begin{tabular}[t]{l}$P_{i+2k+1}$\end{tabular}}}}%
    \put(0.0338133,0.10429485){\color[rgb]{0,0,0}\makebox(0,0)[lt]{\lineheight{1.25}\smash{\begin{tabular}[t]{l}$P_{i}$\end{tabular}}}}%
    \put(0.60301594,0.00873899){\color[rgb]{0,0,0}\makebox(0,0)[lt]{\lineheight{1.25}\smash{\begin{tabular}[t]{l}$P_{i-1}$\end{tabular}}}}%
    \put(0.86334796,0.0627152){\color[rgb]{0,0,0}\makebox(0,0)[lt]{\lineheight{1.25}\smash{\begin{tabular}[t]{l}$P_{i}$\end{tabular}}}}%
    \put(0.91598116,0.31234194){\color[rgb]{0,0,0}\makebox(0,0)[lt]{\lineheight{1.25}\smash{\begin{tabular}[t]{l}$P_{i+1}$\end{tabular}}}}%
    \put(0.6787349,0.06510347){\color[rgb]{0,0,0}\makebox(0,0)[lt]{\lineheight{1.25}\smash{\begin{tabular}[t]{l}$P_{i+k}$\end{tabular}}}}%
    \put(0.50374411,0.13463316){\color[rgb]{0,0,0}\makebox(0,0)[lt]{\lineheight{1.25}\smash{\begin{tabular}[t]{l}$P_{i+k-1}$\end{tabular}}}}%
    \put(0.82288088,0.1959336){\color[rgb]{0,0,0}\makebox(0,0)[lt]{\lineheight{1.25}\smash{\begin{tabular}[t]{l}$P_{i+k+1}$\end{tabular}}}}%
    \put(0.62748885,0.15638372){\color[rgb]{0,0,0}\makebox(0,0)[lt]{\lineheight{1.25}\smash{\begin{tabular}[t]{l}$P_{i+2k}$\end{tabular}}}}%
    \put(0.54905386,0.26450389){\color[rgb]{0,0,0}\makebox(0,0)[lt]{\lineheight{1.25}\smash{\begin{tabular}[t]{l}$P_{i+2k-1}$\end{tabular}}}}%
    \put(0.73608022,0.1590019){\color[rgb]{0,0,0}\makebox(0,0)[lt]{\lineheight{1.25}\smash{\begin{tabular}[t]{l}$P_{i+2k+1}$\end{tabular}}}}%
  \end{picture}%
\endgroup%

%% file: figure/twisted_k_nice.pdf_tex
\begingroup%
  \makeatletter%
  \providecommand\color[2][]{%
    \errmessage{(Inkscape) Color is used for the text in Inkscape, but the package 'color.sty' is not loaded}%
    \renewcommand\color[2][]{}%
  }%
  \providecommand\transparent[1]{%
    \errmessage{(Inkscape) Transparency is used (non-zero) for the text in Inkscape, but the package 'transparent.sty' is not loaded}%
    \renewcommand\transparent[1]{}%
  }%
  \providecommand\rotatebox[2]{#2}%
  \newcommand*\fsize{\dimexpr\f@size pt\relax}%
  \newcommand*\lineheight[1]{\fontsize{\fsize}{#1\fsize}\selectfont}%
  \ifx\svgwidth\undefined%
    \setlength{\unitlength}{362.03240582bp}%
    \ifx\svgscale\undefined%
      \relax%
    \else%
      \setlength{\unitlength}{\unitlength * \real{\svgscale}}%
    \fi%
  \else%
    \setlength{\unitlength}{\svgwidth}%
  \fi%
  \global\let\svgwidth\undefined%
  \global\let\svgscale\undefined%
  \makeatother%
  \begin{picture}(1,0.41886188)%
    \lineheight{1}%
    \setlength\tabcolsep{0pt}%
    \put(0,0){\includegraphics[width=\unitlength,page=1]{twisted_k_nice.pdf}}%
    \put(0.22696423,0.33415693){\color[rgb]{0,0,0}\makebox(0,0)[lt]{\lineheight{1.25}\smash{\begin{tabular}[t]{l}$P_{i+1}$\end{tabular}}}}%
    \put(0.90224107,0.11642774){\color[rgb]{0,0,0}\makebox(0,0)[lt]{\lineheight{1.25}\smash{\begin{tabular}[t]{l}$P_{i+k+1}$\end{tabular}}}}%
    \put(0.89236916,0.21775483){\color[rgb]{0,0,0}\makebox(0,0)[lt]{\lineheight{1.25}\smash{\begin{tabular}[t]{l}$P_{i+k}$\end{tabular}}}}%
    \put(0.01518661,0.1230842){\color[rgb]{0,0,0}\makebox(0,0)[lt]{\lineheight{1.25}\smash{\begin{tabular}[t]{l}$P_{i}$\end{tabular}}}}%
    \put(0.63655446,0.2258932){\color[rgb]{1,0,0}\makebox(0,0)[lt]{\lineheight{1.25}\smash{\begin{tabular}[t]{l}$P'_{i}$\end{tabular}}}}%
  \end{picture}%
\endgroup%

%% file: figure/forward_invariance.pdf_tex
\begingroup%
  \makeatletter%
  \providecommand\color[2][]{%
    \errmessage{(Inkscape) Color is used for the text in Inkscape, but the package 'color.sty' is not loaded}%
    \renewcommand\color[2][]{}%
  }%
  \providecommand\transparent[1]{%
    \errmessage{(Inkscape) Transparency is used (non-zero) for the text in Inkscape, but the package 'transparent.sty' is not loaded}%
    \renewcommand\transparent[1]{}%
  }%
  \providecommand\rotatebox[2]{#2}%
  \newcommand*\fsize{\dimexpr\f@size pt\relax}%
  \newcommand*\lineheight[1]{\fontsize{\fsize}{#1\fsize}\selectfont}%
  \ifx\svgwidth\undefined%
    \setlength{\unitlength}{427.20769447bp}%
    \ifx\svgscale\undefined%
      \relax%
    \else%
      \setlength{\unitlength}{\unitlength * \real{\svgscale}}%
    \fi%
  \else%
    \setlength{\unitlength}{\svgwidth}%
  \fi%
  \global\let\svgwidth\undefined%
  \global\let\svgscale\undefined%
  \makeatother%
  \begin{picture}(1,0.33873459)%
    \lineheight{1}%
    \setlength\tabcolsep{0pt}%
    \put(0,0){\includegraphics[width=\unitlength,page=1]{forward_invariance.pdf}}%
    \put(0.0299444,0.32468449){\color[rgb]{0,0,0}\makebox(0,0)[lt]{\lineheight{1.25}\smash{\begin{tabular}[t]{l}$P_0$\end{tabular}}}}%
    \put(0.02534062,0.11005154){\color[rgb]{0,0,0}\makebox(0,0)[lt]{\lineheight{1.25}\smash{\begin{tabular}[t]{l}$P_{1}$\end{tabular}}}}%
    \put(0.21710546,0.00956257){\color[rgb]{0,0,0}\makebox(0,0)[lt]{\lineheight{1.25}\smash{\begin{tabular}[t]{l}$P_{2}$\end{tabular}}}}%
    \put(0.41453881,0.03026302){\color[rgb]{0,0,0}\makebox(0,0)[lt]{\lineheight{1.25}\smash{\begin{tabular}[t]{l}$P_{3}$\end{tabular}}}}%
    \put(0.07137628,0.22032168){\color[rgb]{0,0,0}\makebox(0,0)[lt]{\lineheight{1.25}\smash{\begin{tabular}[t]{l}$P_{k}$\end{tabular}}}}%
    \put(0.16894431,0.09552146){\color[rgb]{0,0,0}\makebox(0,0)[lt]{\lineheight{1.25}\smash{\begin{tabular}[t]{l}$P_{k+1}$\end{tabular}}}}%
    \put(0.36798802,0.06133069){\color[rgb]{0,0,0}\makebox(0,0)[lt]{\lineheight{1.25}\smash{\begin{tabular}[t]{l}$P_{k+2}$\end{tabular}}}}%
    \put(0.43664987,0.20884409){\color[rgb]{0,0,0}\makebox(0,0)[lt]{\lineheight{1.25}\smash{\begin{tabular}[t]{l}$P_{k+3}$\end{tabular}}}}%
    \put(0.15478793,0.15499127){\color[rgb]{0,0,0}\makebox(0,0)[lt]{\lineheight{1.25}\smash{\begin{tabular}[t]{l}$P_{2k}$\end{tabular}}}}%
    \put(0.25526932,0.11140655){\color[rgb]{0,0,0}\makebox(0,0)[lt]{\lineheight{1.25}\smash{\begin{tabular}[t]{l}$P_{2k+1}$\end{tabular}}}}%
    \put(0.37768189,0.16753844){\color[rgb]{0,0,0}\makebox(0,0)[lt]{\lineheight{1.25}\smash{\begin{tabular}[t]{l}$P_{2k+2}$\end{tabular}}}}%
    \put(0.09626543,0.11410461){\color[rgb]{1,0,0}\makebox(0,0)[lt]{\lineheight{1.25}\smash{\begin{tabular}[t]{l}$P'_0$\end{tabular}}}}%
    \put(0.26805781,0.05833572){\color[rgb]{1,0,0}\makebox(0,0)[lt]{\lineheight{1.25}\smash{\begin{tabular}[t]{l}$P'_{1}$\end{tabular}}}}%
    \put(0.21974376,0.09990728){\color[rgb]{1,0,0}\makebox(0,0)[lt]{\lineheight{1.25}\smash{\begin{tabular}[t]{l}$P'_{k}$\end{tabular}}}}%
    \put(0.43689438,0.09727702){\color[rgb]{1,0,0}\makebox(0,0)[lt]{\lineheight{1.25}\smash{\begin{tabular}[t]{l}$P'_{2}$\end{tabular}}}}%
    \put(0.34341872,0.11425428){\color[rgb]{1,0,0}\makebox(0,0)[lt]{\lineheight{1.25}\smash{\begin{tabular}[t]{l}$P'_{k+1}$\end{tabular}}}}%
    \put(0.5931057,0.01203359){\color[rgb]{0,0,0}\makebox(0,0)[lt]{\lineheight{1.25}\smash{\begin{tabular}[t]{l}$P_0$\end{tabular}}}}%
    \put(0.80765663,0.01142849){\color[rgb]{0,0,0}\makebox(0,0)[lt]{\lineheight{1.25}\smash{\begin{tabular}[t]{l}$P_{1}$\end{tabular}}}}%
    \put(0.97148026,0.12949838){\color[rgb]{0,0,0}\makebox(0,0)[lt]{\lineheight{1.25}\smash{\begin{tabular}[t]{l}$P_{2}$\end{tabular}}}}%
    \put(0.9511841,0.32693913){\color[rgb]{0,0,0}\makebox(0,0)[lt]{\lineheight{1.25}\smash{\begin{tabular}[t]{l}$P_{3}$\end{tabular}}}}%
    \put(0.53019045,0.16471528){\color[rgb]{0,0,0}\makebox(0,0)[lt]{\lineheight{1.25}\smash{\begin{tabular}[t]{l}$P_{k}$\end{tabular}}}}%
    \put(0.66523341,0.06979428){\color[rgb]{0,0,0}\makebox(0,0)[lt]{\lineheight{1.25}\smash{\begin{tabular}[t]{l}$P_{k+1}$\end{tabular}}}}%
    \put(0.79993502,0.0874853){\color[rgb]{0,0,0}\makebox(0,0)[lt]{\lineheight{1.25}\smash{\begin{tabular}[t]{l}$P_{k+2}$\end{tabular}}}}%
    \put(0.89562147,0.19611568){\color[rgb]{0,0,0}\makebox(0,0)[lt]{\lineheight{1.25}\smash{\begin{tabular}[t]{l}$P_{k+3}$\end{tabular}}}}%
    \put(0.60127795,0.25784771){\color[rgb]{0,0,0}\makebox(0,0)[lt]{\lineheight{1.25}\smash{\begin{tabular}[t]{l}$P_{2k}$\end{tabular}}}}%
    \put(0.62955821,0.14416036){\color[rgb]{0,0,0}\makebox(0,0)[lt]{\lineheight{1.25}\smash{\begin{tabular}[t]{l}$P_{2k+1}$\end{tabular}}}}%
    \put(0.71445644,0.09831138){\color[rgb]{0,0,0}\makebox(0,0)[lt]{\lineheight{1.25}\smash{\begin{tabular}[t]{l}$P_{2k+2}$\end{tabular}}}}%
    \put(0.56627407,0.07683071){\color[rgb]{1,0,0}\makebox(0,0)[lt]{\lineheight{1.25}\smash{\begin{tabular}[t]{l}$P'_0$\end{tabular}}}}%
    \put(0.75978428,0.05053496){\color[rgb]{1,0,0}\makebox(0,0)[lt]{\lineheight{1.25}\smash{\begin{tabular}[t]{l}$P'_{1}$\end{tabular}}}}%
    \put(0.90978007,0.13006004){\color[rgb]{1,0,0}\makebox(0,0)[lt]{\lineheight{1.25}\smash{\begin{tabular}[t]{l}$P'_{2}$\end{tabular}}}}%
    \put(0.56087986,0.22718831){\color[rgb]{1,0,0}\makebox(0,0)[lt]{\lineheight{1.25}\smash{\begin{tabular}[t]{l}$P'_{k}$\end{tabular}}}}%
    \put(0.64398558,0.10744855){\color[rgb]{1,0,0}\makebox(0,0)[lt]{\lineheight{1.25}\smash{\begin{tabular}[t]{l}$P'_{k+1}$\end{tabular}}}}%
  \end{picture}%
\endgroup%

%% file: figure/backward_invariance.pdf_tex
\begingroup%
  \makeatletter%
  \providecommand\color[2][]{%
    \errmessage{(Inkscape) Color is used for the text in Inkscape, but the package 'color.sty' is not loaded}%
    \renewcommand\color[2][]{}%
  }%
  \providecommand\transparent[1]{%
    \errmessage{(Inkscape) Transparency is used (non-zero) for the text in Inkscape, but the package 'transparent.sty' is not loaded}%
    \renewcommand\transparent[1]{}%
  }%
  \providecommand\rotatebox[2]{#2}%
  \newcommand*\fsize{\dimexpr\f@size pt\relax}%
  \newcommand*\lineheight[1]{\fontsize{\fsize}{#1\fsize}\selectfont}%
  \ifx\svgwidth\undefined%
    \setlength{\unitlength}{454.9835974bp}%
    \ifx\svgscale\undefined%
      \relax%
    \else%
      \setlength{\unitlength}{\unitlength * \real{\svgscale}}%
    \fi%
  \else%
    \setlength{\unitlength}{\svgwidth}%
  \fi%
  \global\let\svgwidth\undefined%
  \global\let\svgscale\undefined%
  \makeatother%
  \begin{picture}(1,0.36871532)%
    \lineheight{1}%
    \setlength\tabcolsep{0pt}%
    \put(0,0){\includegraphics[width=\unitlength,page=1]{backward_invariance.pdf}}%
    \put(0.00162356,0.35614148){\color[rgb]{0,0,0}\makebox(0,0)[lt]{\lineheight{1.25}\smash{\begin{tabular}[t]{l}$P'_{-k-1}$\end{tabular}}}}%
    \put(0.00506492,0.1530663){\color[rgb]{0,0,0}\makebox(0,0)[lt]{\lineheight{1.25}\smash{\begin{tabular}[t]{l}$P'_{-k}$\end{tabular}}}}%
    \put(0.18468974,0.00466415){\color[rgb]{0,0,0}\makebox(0,0)[lt]{\lineheight{1.25}\smash{\begin{tabular}[t]{l}$P'_{-k+1}$\end{tabular}}}}%
    \put(0.41350595,0.06765211){\color[rgb]{0,0,0}\makebox(0,0)[lt]{\lineheight{1.25}\smash{\begin{tabular}[t]{l}$P'_{-k+2}$\end{tabular}}}}%
    \put(0.069639,0.23544777){\color[rgb]{0,0,0}\makebox(0,0)[lt]{\lineheight{1.25}\smash{\begin{tabular}[t]{l}$P'_{-1}$\end{tabular}}}}%
    \put(0.18798656,0.09986102){\color[rgb]{0,0,0}\makebox(0,0)[lt]{\lineheight{1.25}\smash{\begin{tabular}[t]{l}$P'_0$\end{tabular}}}}%
    \put(0.34833629,0.11671771){\color[rgb]{0,0,0}\makebox(0,0)[lt]{\lineheight{1.25}\smash{\begin{tabular}[t]{l}$P'_1$\end{tabular}}}}%
    \put(0.4112875,0.19821017){\color[rgb]{0,0,0}\makebox(0,0)[lt]{\lineheight{1.25}\smash{\begin{tabular}[t]{l}$P'_2$\end{tabular}}}}%
    \put(0.17031065,0.15793618){\color[rgb]{0,0,0}\makebox(0,0)[lt]{\lineheight{1.25}\smash{\begin{tabular}[t]{l}$P'_{k-1}$\end{tabular}}}}%
    \put(0.3046994,0.14073749){\color[rgb]{0,0,0}\makebox(0,0)[lt]{\lineheight{1.25}\smash{\begin{tabular}[t]{l}$P'_{k}$\end{tabular}}}}%
    \put(0.36211295,0.18853649){\color[rgb]{0,0,0}\makebox(0,0)[lt]{\lineheight{1.25}\smash{\begin{tabular}[t]{l}$P'_{k+1}$\end{tabular}}}}%
    \put(0.00125788,0.27172409){\color[rgb]{1,0,0}\makebox(0,0)[lt]{\lineheight{1.25}\smash{\begin{tabular}[t]{l}$P_0$\end{tabular}}}}%
    \put(0.1194646,0.06277857){\color[rgb]{1,0,0}\makebox(0,0)[lt]{\lineheight{1.25}\smash{\begin{tabular}[t]{l}$P_1$\end{tabular}}}}%
    \put(0.35243927,0.05229874){\color[rgb]{1,0,0}\makebox(0,0)[lt]{\lineheight{1.25}\smash{\begin{tabular}[t]{l}$P_2$\end{tabular}}}}%
    \put(0.10442139,0.14990091){\color[rgb]{1,0,0}\makebox(0,0)[lt]{\lineheight{1.25}\smash{\begin{tabular}[t]{l}$P_k$\end{tabular}}}}%
    \put(0.30228098,0.07872377){\color[rgb]{1,0,0}\makebox(0,0)[lt]{\lineheight{1.25}\smash{\begin{tabular}[t]{l}$P_{k+1}$\end{tabular}}}}%
    \put(0.64608569,0.02427061){\color[rgb]{1,0,0}\makebox(0,0)[lt]{\lineheight{1.25}\smash{\begin{tabular}[t]{l}$P_0$\end{tabular}}}}%
    \put(0.83993616,0.05186053){\color[rgb]{1,0,0}\makebox(0,0)[lt]{\lineheight{1.25}\smash{\begin{tabular}[t]{l}$P_1$\end{tabular}}}}%
    \put(0.9612838,0.24114681){\color[rgb]{1,0,0}\makebox(0,0)[lt]{\lineheight{1.25}\smash{\begin{tabular}[t]{l}$P_2$\end{tabular}}}}%
    \put(0.58365244,0.1257921){\color[rgb]{1,0,0}\makebox(0,0)[lt]{\lineheight{1.25}\smash{\begin{tabular}[t]{l}$P_k$\end{tabular}}}}%
    \put(0.72475275,0.06021423){\color[rgb]{1,0,0}\makebox(0,0)[lt]{\lineheight{1.25}\smash{\begin{tabular}[t]{l}$P_{k+1}$\end{tabular}}}}%
    \put(0.53393008,0.0605018){\color[rgb]{0,0,0}\makebox(0,0)[lt]{\lineheight{1.25}\smash{\begin{tabular}[t]{l}$P'_{-k-1}$\end{tabular}}}}%
    \put(0.74800588,-0.00781505){\color[rgb]{0,0,0}\makebox(0,0)[lt]{\lineheight{1.25}\smash{\begin{tabular}[t]{l}$P'_{-k}$\end{tabular}}}}%
    \put(0.9694138,0.15171576){\color[rgb]{0,0,0}\makebox(0,0)[lt]{\lineheight{1.25}\smash{\begin{tabular}[t]{l}$P'_{-k+1}$\end{tabular}}}}%
    \put(0.93145195,0.3516517){\color[rgb]{0,0,0}\makebox(0,0)[lt]{\lineheight{1.25}\smash{\begin{tabular}[t]{l}$P'_{-k+2}$\end{tabular}}}}%
    \put(0.52754272,0.25618748){\color[rgb]{0,0,0}\makebox(0,0)[lt]{\lineheight{1.25}\smash{\begin{tabular}[t]{l}$P'_{-1}$\end{tabular}}}}%
    \put(0.61260693,0.08097193){\color[rgb]{0,0,0}\makebox(0,0)[lt]{\lineheight{1.25}\smash{\begin{tabular}[t]{l}$P'_0$\end{tabular}}}}%
    \put(0.78153293,0.08078902){\color[rgb]{0,0,0}\makebox(0,0)[lt]{\lineheight{1.25}\smash{\begin{tabular}[t]{l}$P'_1$\end{tabular}}}}%
    \put(0.8848609,0.19765081){\color[rgb]{0,0,0}\makebox(0,0)[lt]{\lineheight{1.25}\smash{\begin{tabular}[t]{l}$P'_2$\end{tabular}}}}%
    \put(0.61624157,0.33446186){\color[rgb]{0,0,0}\makebox(0,0)[lt]{\lineheight{1.25}\smash{\begin{tabular}[t]{l}$P'_{k-1}$\end{tabular}}}}%
    \put(0.6256033,0.17869003){\color[rgb]{0,0,0}\makebox(0,0)[lt]{\lineheight{1.25}\smash{\begin{tabular}[t]{l}$P'_{k}$\end{tabular}}}}%
    \put(0.72463402,0.10614581){\color[rgb]{0,0,0}\makebox(0,0)[lt]{\lineheight{1.25}\smash{\begin{tabular}[t]{l}$P'_{k+1}$\end{tabular}}}}%
  \end{picture}%
\endgroup%

%% file: 4partition.tex
\section{Coordinate Representation of 3-Spirals} \label{sec:partition}

\subsection{The Tic-Tac-Toe Grids} \label{subsec:definition of tic tac toe}

Recall the intervals $I = (-\infty, 0)$, $J = (0,1)$, $K = (1,\infty)$ from $\S$\ref{subsec:tic tac toe}. One can partition $\RR^2$ into a $3 \times 3$ grid. See Figure \ref{fig:tic-tac-toe grid}. We make the following definition:

\begin{definition} \label{def:checkerboard polygon}
    For $n \geq 2$, let $S_n(I, J)$ be the subset of $\cP_n$ that satisfies the following: given $[P] \in S_n(I, J)$, for all $i \in \{0, \ldots, n-1\}$, $(x_{2i}, x_{2i+1}) \in I \times J$. We similarly define $S_n(K, J)$, $S_n(J, I)$, and $S_n(J, K)$. 
\end{definition}

The following symmetries of the four grids follow directly from Definition \ref{def:checkerboard polygon}. 

\begin{proposition} \label{prop:tic tac toe symmetry}
    For $i \in \ZZ$, define the map $\sigma_i: \ZZ \rightarrow \ZZ$ by $\sigma_i(x) = x + i$. Define the map $\iota: \ZZ \rightarrow \ZZ$ by $\iota(x) = -x$. 
    Given $[P] \in \cP_n$, the following are true:
    \begin{itemize}
        \item If $[P] \in S_n(I,J)$, then $[P \circ \sigma_i] \in S_n(I,J)$ for all $i \in \ZZ$. This also holds for $S_n(K, J)$, $S_n(J, I)$, and $S_n(J, K)$. 
        \item $[P] \in S_n(I,J)$ if and only if $[P \circ \iota] \in S_n(J,I)$. 
        \item $[P] \in S_n(K,J)$ if and only if $[P \circ \iota] \in S_n(J,K)$. 
    \end{itemize}
\end{proposition}

To understand the geometry implied by the corner invariants, we need to examine what happens when the corner invariants take value from $0,1,\infty$. 

\begin{proposition} \label{prop:P2 position and corner invariants}
    For all $[P] \in \cP_n$ with corner invariants $x_j = x_j(P)$ and $i \in \ZZ$, we have the following correspondence between the position of $P_{i+2}$ and the values of $x_{2i}$ and $x_{2i+1}$:
    \begin{center}
        \begin{tabular}{c|c||c|c}
            Configuration & Coordinates &
            Configuration & Coordinates \\
            \hline 
            $P_{i+2} \in P_{i+1} P_{i}$ & $x_{2i} = 0$ &
            $P_{i+2} \in P_{i-1} P_{i+1}$ & $x_{2i+1} = 0$ \\
            $P_{i+2} \in P_{i+1} P_{i-2}$ & $x_{2i} = 1$ &
            $P_{i+2} \in P_{i-1} P_{i-2}$ & $x_{2i+1} = 1$ \\
            $P_{i+2} \in P_{i+1} P_{i-1}$ & $x_{2i} = \infty$ &
            $P_{i+2} \in P_{i-1} P_{i}$ & $x_{2i+1} = \infty $ 
        \end{tabular}
    \end{center}
\end{proposition}

\begin{figure}[ht]
    \centering
    \footnotesize
    
    \def\svgwidth{0.35\columnwidth}
    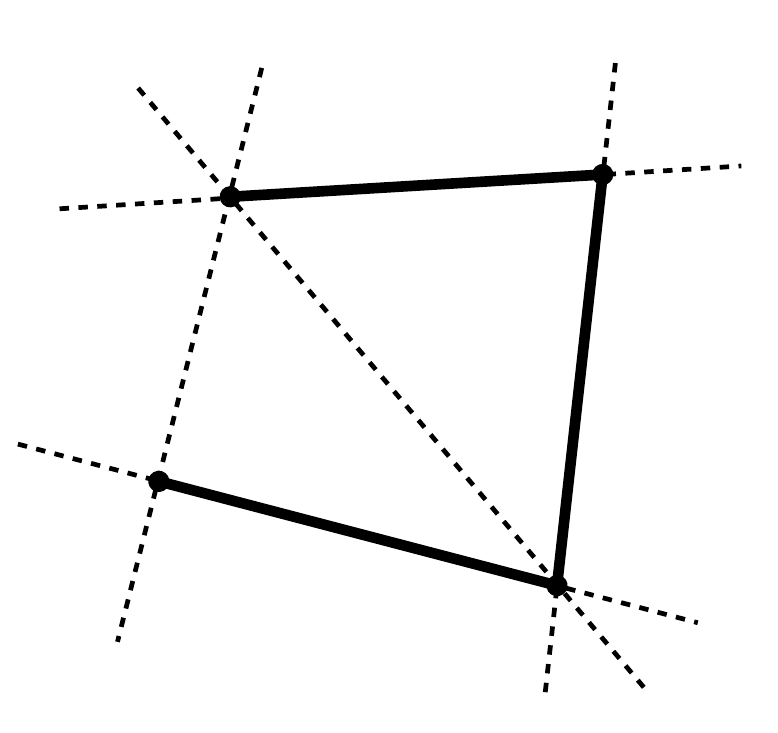

    \caption{Configurations of points and lines in the proof of Proposition \ref{prop:P2 position and corner invariants}.}
    \label{fig:geometry_of_corner_inv}
\end{figure}

\begin{proof}
    Consider the following lines: 
    \begin{equation*}
        \begin{aligned}
            l_1 &= P_{i+1} P_{i-2}; & l_2 &= P_{i+1} P_{i-1}; &
            l_3 &= P_{i+1} P_{i}; & l_4 &= P_{i+1} P_{i+2}; \\
            m_1 &= P_{i-1} P_{i+2}; & m_2 &= P_{i-1} P_{i+1}; &
            m_3 &= P_{i-1} P_{i}; & m_4 &= P_{i-1} P_{i-2}.
        \end{aligned}
    \end{equation*}
    See Figure \ref{fig:geometry_of_corner_inv} for a visualization of the configurations of points and lines. 
    Equation \eqref{eqn:corner invariant lines def} implies $x_{2i} = \chi(l_1,l_2,l_3,l_4)$ and $x_{2i+1} = \chi(m_1,m_2,m_3,m_4)$. This yields
    \begin{center}
        \begin{tabular}{c|c|c||c|c|c}
            Configuration & Lines & Coordinates &
            Configuration & Lines & Coordinates \\
            \hline 
            $P_{i+2} \in P_{i+1} P_{i}$ & $l_4 = l_3$ & $x_{2i} = 0$ &
            $P_{i+2} \in P_{i-1} P_{i+1}$ & $m_1 = m_2$ & $x_{2i+1} = 0$ \\ 
            $P_{i+2} \in P_{i+1} P_{i-2}$ & $l_4 = l_1$ & $x_{2i} = 1$ &
            $P_{i+2} \in P_{i-1} P_{i-2}$ & $m_1 = m_4$ & $x_{2i+1} = 1$ \\
            $P_{i+2} \in P_{i+1} P_{i-1}$ & $l_4 = l_2$ & $x_{2i} = \infty$ &
            $P_{i+2} \in P_{i-1} P_{i}$ & $m_1 = m_3$ & $x_{2i+1} = \infty$ 
        \end{tabular}
    \end{center}
    which is precisely the relationship described in the proposition.
\end{proof}

\begin{remark} \label{rmk:P2 position and corner invariants}
    Proposition \ref{prop:P2 position and corner invariants} also gives us a way to determine the position of $P_{i+2}$ when neither $x_{2i}$ nor $x_{2i+1}$ takes value in $0,1,\infty$. Suppose the four points $P_{i-2}, P_{i-1}, P_i, P_{i+1}$ are in general position. For $i, j, k \in \{1,2,3\}$ distinct, we define $U_{i,j}$ to be the connected component of $\RR\PP^2 - (l_i \cup l_j)$ that does not intersect $l_k$. For $i,j,k \in \{2,3,4\}$ distinct, we define $V_{i,j}$ to be the connected component of $\RR\PP^2 - (m_i \cup m_j)$ that does not intersect $m_k$. See Figure \ref{fig:corner_inv_partition} for a visualization of the $U_{i,j}$'s and $V_{i,j}$'s using the point configurations given in Figure \ref{fig:geometry_of_corner_inv}. By Proposition \ref{prop:P2 position and corner invariants} and continuity of $\chi$, we have the following:
    \begin{table}[h!]
        \centering
        \begin{tabular}{c|c||c|c}
            Configuration & Coordinates & 
            Configuration & Coordinates \\
            \hline 
            $P_{i+2} \in U_{2,3}$ & $x_{2i} = I$ &
            $P_{i+2} \in V_{2,3}$ & $x_{2i+1} = I$ \\
            $P_{i+2} \in U_{1,3}$ & $x_{2i} = J$ &
            $P_{i+2} \in V_{2,4}$ & $x_{2i+1} = J$ \\
            $P_{i+2} \in U_{1,2}$ & $x_{2i} = K$ &
            $P_{i+2} \in V_{3,4}$ & $x_{2i+1} = K $ 
        \end{tabular}
    \end{table}
\end{remark}

\begin{corollary} \label{cor:grids are 3 nice}
    Given $[P] \in \cP_n$ with corner invariants $x_j = x_j(P)$, if $x_j \not\in \{0,1,\infty\}$ for all $j$, then $P$ is $3$-nice. Moreover, every four consecutive points of $P$ are in general position.
\end{corollary}

\begin{proof}
    Using Proposition \ref{prop:P2 position and corner invariants} we may check that 
    \begin{table}[h!]
        \centering
        \begin{tabular}{c|c||c|c}
            Collinearity & Coordinates & Collinearity & Coordinates \\
            \hline
            $P_{i-2}$, $P_{i-1}$, $P_{i+1}$ & $x_{2i-1} = \infty$ & 
            $P_{i-1}$, $P_{i}$, $P_{i+2}$ & $x_{2i+1} = \infty$ \\
            $P_{i-2}$, $P_{i-1}$, $P_{i+2}$ & $x_{2i+1} = 1$ & 
            $P_{i-1}$, $P_{i+1}$, $P_{i+2}$ & $x_{2i+1} = 0$ \\
            $P_{i-2}$, $P_{i+1}$, $P_{i+2}$ & $x_{2i} = 1$ &
            $P_{i}$, $P_{i+1}$, $P_{i+2}$ & $x_{2i} = 0$ \\
            $P_{i-1}$, $P_{i}$, $P_{i+1}$ & $x_{2i-2} = 0$ &
            & \\
        \end{tabular}
    \end{table}

    All seven cases contradict the assumption in the corollary. Therefore, the four points $P_{i-2}$, $P_{i-1}$, $P_{i+1}$, $P_{i+2}$ are in general position, and the four consecutive points $P_{i-1}$, $P_i$, $P_{i+1}$, $P_{i+2}$ are in general position for all $i \in \ZZ$. This shows $P$ is 3-nice, and every four consecutive points of $P$ are in general position.
\end{proof}

\begin{figure}[ht]
    \centering
    \footnotesize
    
    \def\svgwidth{0.8\columnwidth}
    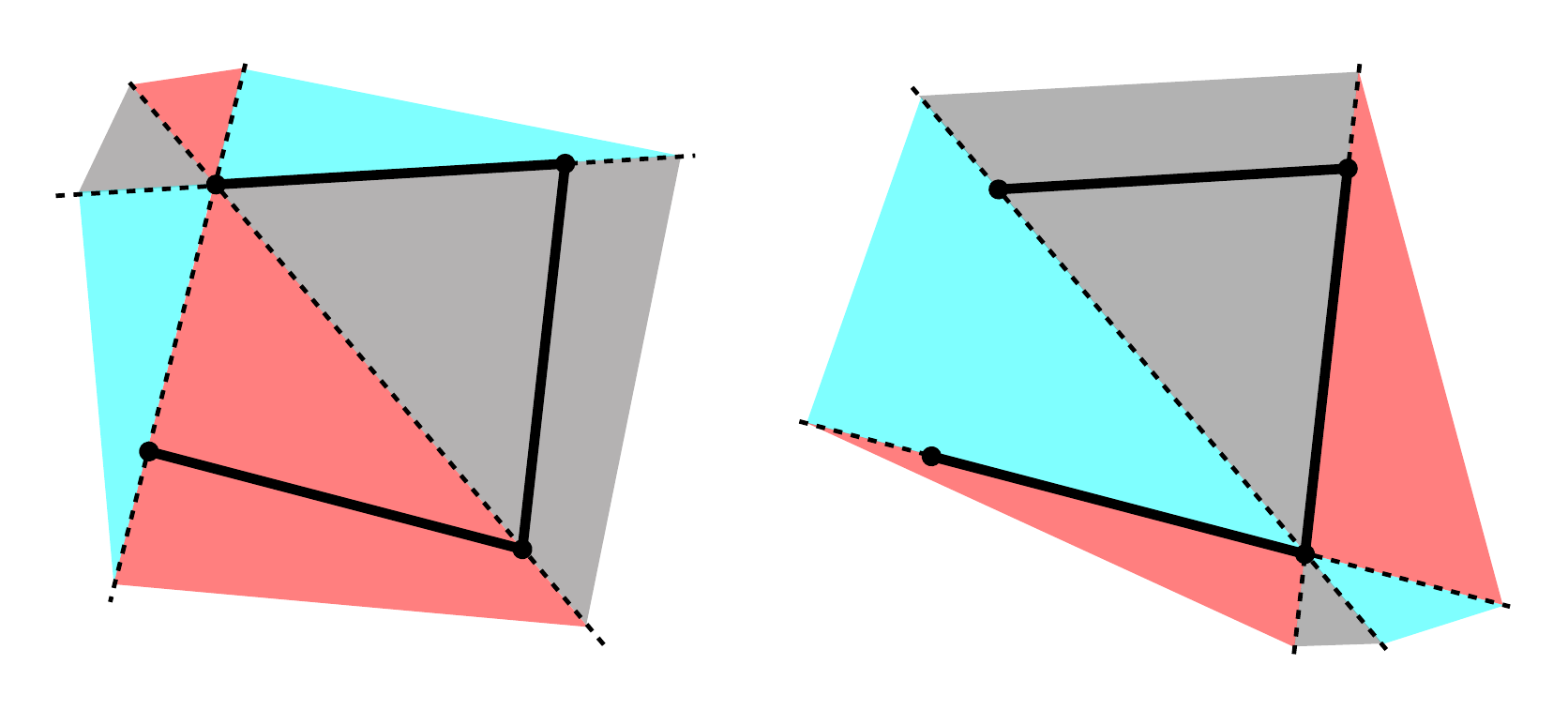

    \caption{The connected components $U_{i,j}$'s and $V_{i,j}$'s in Remark \ref{rmk:P2 position and corner invariants}. The corner invariants value in $I$ if $P_{i+2}$ lies in the black-shaded region, $J$ if $P_{i+2}$ lies in the red-shaded region, and $K$ if $P_{i+2}$ lies in the cyan-shaded region.}
    \label{fig:corner_inv_partition}
\end{figure}

Our goal of this section is to prove the following correspondence theorem:

\begin{theorem} \label{thm:tic-tac-toe and 3-spirals}
    For all $n \geq 2$, $\cS_{3,n}^\alpha = S_n(J, I)$, $\cS_{3,n}^\beta = S_n(K,J)$. 
\end{theorem}

This theorem immediately produces the following important corollary.

\begin{corollary} \label{cor:tic-tac-toe and 3-spirals}
    For all $n \geq 2$, the four cells $S_n(I,J)$, $S_n(K,J)$, $S_n(J,I)$, $S_n(J,K)$ are both forward and backward invariant under $T_3$.
\end{corollary}

\begin{proof}
    The case $S_n(J, I)$ and $S_n(K, J)$ follows immediately from Theorem \ref{thm:spiral polygon invariance} and \ref{thm:tic-tac-toe and 3-spirals}. We will prove the case $S_n(I, J)$. The case $S_n(J,K)$ is completely analogous, so we will omit.
    
    Fix $[P] \in S_n(I, J)$. Recall the maps $\sigma_i$ and $\iota$ from Proposition \ref{prop:tic tac toe symmetry}. Equation \eqref{eqn:delta k,1 formula} implies $T_3(P \circ \iota) = T_3(P) \circ \iota \circ \sigma_4$. Then, Proposition \ref{prop:tic tac toe symmetry} implies $[P \circ \iota] \in S_n(J, I)$, so $[T_3(P \circ \iota)] \in S_n(J, I)$. Finally, observe that 
    \begin{equation*}
        T_3(P)
        = (T_3(P) \circ \iota \circ \sigma_4) \circ (\sigma_{-4} \circ \iota)
        = T_3(P \circ \iota) \circ (\sigma_{-4} \circ \iota).
    \end{equation*}
    It follows that $[T_3(P)] \in S_n(I, J)$. We omit the proof of $[T_3^{-1}(P)] \in S_n(I, J)$. 
\end{proof}

\subsection{The Correspondence of \texorpdfstring{$\cS_{3,n}^\alpha$}{S3n alpha} and \texorpdfstring{$S_n(J,I)$}{Sn(J,I)}} \label{subsec:correspondence alpha and IJ}

Here we show that $\cS_{3,n}^\alpha$ is equivalent to $S_n(J, I)$. We will first show that the corner invariants of a $0$-representative $P$ of some $[P] \in \cS_{3,n}^\alpha$ satisfies $S_n(J, I)$. Then, we will show that we can find type-$\alpha$ $N$-representatives for all $N \in \ZZ$ given any $[P] \in S_n(J, I)$. 

\begin{lemma} \label{lem:S3 alpha in S(J,I)}
    If $P$ is an $N$-representative of $[P] \in \cS_{3,n}^\alpha$, then $P_{i+2} \in \Int(P_{i-1}, P_{i}, P_{i+1})$ for all $i > N+1$. 
\end{lemma}

\begin{proof}
    Since $(P_{i-1}, P_{i}, P_{i+1})$ is positive, we may normalize with $\Aff_2^+(\RR)$ so that $P_{i-1} = (-1,0)$, $P_{i} = (0,0)$, and $P_{i+1} = (0,1)$. Let $P_{i+2} = (x,y)$. It suffices to show that $x < 0$, $y > 0$, and $y - x < 1$. We get $x < 0$ from positivity of $(P_{i}, P_{i+1}, P_{i+2})$, and we get $y > 0$ from positivity of $(P_{i-1}, P_{i}, P_{i+2})$. Finally, since $(P_{i-2}, P_{i-1}, P_{i+2})$ is positive and $P_{i+1} \in \Int (P_{i-2}, P_{i-1}, P_{i+2})$, Proposition \ref{prop:orientation lemma} implies $(P_{i+1}, P_{i-1}, P_{i+2})$ is positive, which gives us $y - x < 1$ as desired. 
\end{proof}

\begin{proposition} \label{prop:S3 alpha in S(J,I)}
    For all $n \geq 2$, $\cS_{3,n}^\alpha \subset S_n(J, I)$.
\end{proposition}

\begin{proof}
    Fix $i \in \ZZ$. Let $P$ be an $(i-3)$-representative of $[P] \in \cS_{3,n}^\alpha$ with corner invariants $x_j = x_j(P)$. 
    Normalize with $\Aff_2^+(\RR)$ so that $P_{i-1} = (-1,0)$, $P_{i} = (0,0)$, and $P_{i+1} = (0,1)$. Let $s_{a,b}$ denote the slope of the line $P_{i+a} P_{i+b}$. See Figure \ref{fig:alpha JI} for the configuration of points. 

    We want to show that $(x_{2i}, x_{2i+1}) \in I \times J$. 
    By Lemma \ref{lem:S3 alpha in S(J,I)}, $P_{i+1} \in \Int (P_{i-2}, P_{i-1}, P_i)$. This implies $s_{1,-2} > s_{-1,-2} > 1$. On the other hand, since $P_{i+1} \in \Int(P_{i-2}, P_{i-1}, P_{i+2})$, we have $s_{1,2} > s_{1,-2} > 1$, and $s_{-1,2} \in (0,1)$. This gives us
    \begin{equation*}
        \begin{aligned}
            x_{2i}
            &= \frac{(s_{1,-2} - s_{1,-1}) (s_{1,0} - s_{1,2})}{(s_{1,-2} - s_{1,0}) (s_{1,-1} - s_{1,2})}
            = \frac{s_{1,-2} - 1}{s_{1,2} - 1} \in J  \hspace{20pt} \text{and} \\
            x_{2i+1}
            &= \frac{(s_{-1,2} - s_{-1,1}) (s_{-1,0} - s_{-1,-2})}{(s_{-1,2} - s_{-1,0}) (s_{-1,1} - s_{-1,-2})}
            = \frac{s_{-1,-2} (s_{-1,2} - 1)}{s_{-1,2} (s_{-1,-2} - 1)} \in I.
        \end{aligned}
    \end{equation*}
    This concludes the proof. 
\end{proof}

\begin{figure}[ht]
    \centering
    \footnotesize
    
    \def\svgwidth{0.35\columnwidth}
    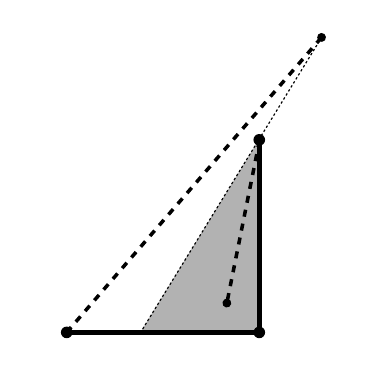

    \caption{Configuration of Proposition \ref{prop:S3 alpha in S(J,I)} and \ref{prop:S3 alpha equals in S(J,I)}.}
    \label{fig:alpha JI}
\end{figure}

\begin{proposition} \label{prop:S3 alpha equals in S(J,I)}
    For all $n \geq 2$, $\cS_{3,n}^\alpha = S_n(I,J)$.
\end{proposition}

\begin{proof}
    Proposition \ref{prop:S3 alpha in S(J,I)} implies we only need to show $\cS_{3,n}^\alpha \supset S_n(I,J)$. Given $[P] \in S_n(I,J)$, let $P$ be a representative that satisfies $P_{-1} = (1,4)$, $P_{0} = (-1,0)$, $P_{1} = (0,0)$, $P_{2} = (0,1)$. Say that $P$ satisfies condition $(*)_i$ if the three triangles $(P_{i-1}, P_{i}, P_{i+1})$, $(P_{i}, P_{i+1}, P_{i+2})$, $(P_{i-1}, P_{i}, P_{i+2})$ are all positive, $P_{i+2} \in \Int(P_{i-1}, P_i, P_{i+1})$, and $P_{i+2} \in \Int(P_{i-2}, P_{i-1}, P_{i+1})$. 

    We show that for all $i > 0$, if $P$ satisfies $(*)_{i-1}$, then $P$ satisfies $(*)_{i}$. Since $(P_{i-1}, P_i, P_{i+1})$ is positive, we can normalize with $\Aff_2^+(\RR)$ so that $P_{i-1} = (-1,0)$, $P_i = (0,0)$, and $P_{i+1} = (0,1)$. Let $s_{a,b}$ denote the slope of $P_{i+a} P_{i+b}$. Since $P_{i+1} \in \Int(P_{i-2}, P_{i-1}, P_{i})$, we know that $s_{1,-2} > s_{-1,-2} > 1$. Then, $x_{2i} \in J$ implies $0 < \frac{s_{1,-2} - 1}{s_{1,2} - 1} < 1$. This gives us $s_{1,2} > s_{1,-2} > 1$. On the other hand, $x_{2i+1} \in I$ implies $\frac{s_{-1,-2} (s_{-1,2} - 1)}{s_{-1,2} (s_{-1,-2} - 1)} < 0$. Since $s_{-1,-2} > 1$, this is equivalent to $1 - \frac{1}{s_{-1,2}} < 0$, which implies $s_{-1,2} \in (0,1)$. Thus, the two lines $P_{i-1} P_{i+2}$ and $P_{i+1} P_{i+2}$ must meet in the shaded triangle in Figure \ref{fig:alpha JI}, which implies $(P_{i}, P_{i+1}, P_{i+2})$, $(P_{i-1}, P_i, P_{i+2})$ are positive, $P_{i+2} \in \Int(P_{i-1}, P_{i}, P_{i+1})$, and $P_{i+2} \in \Int(P_{i-2}, P_{i-1}, P_{i+1})$, so $P$ satisfies $(*)_i$. Finally, since $P$ clearly satisfies $(*)_0$, by induction $P$ satisfies $(*)_i$ for all $i \geq 0$, so $P$ is a type-$\alpha$ $0$-representative of a 3-spiral. We conclude that $[P] \in \cS_{3,n}^\alpha$. 
\end{proof}

\subsection{The Correspondence of \texorpdfstring{$\cS_{3,n}^\beta$}{S3n beta} and \texorpdfstring{$S_n(K,J)$}{Sn(K,J)}}

Here we show that $\cS_{3,n}^\beta$ is equivalent to $S_n(K,J)$. The ideas behind the proofs are essentially the same as the ones in $\S$\ref{subsec:correspondence alpha and IJ}. We will focus on explaining how to modify the details of the proofs in $\S$\ref{subsec:correspondence alpha and IJ} for type-$\beta$ 3-spirals and $S_n(K,J)$. 

\begin{lemma} \label{lem:S3 beta in S(K,J)}
    If $P$ is an $N$-representative of $[P] \in \cS_{3,n}^\beta$, then the quadrilateral joined by vertices $(P_i, P_{i+1}, P_{i+2}, P_{i+3})$ is convex for all $i > N$. 
\end{lemma}

\begin{proof}
    Normalize with $\Aff_2^+(\RR)$ so that $P_i = (-1,0)$, $P_{i+1} = (0,0)$, $P_{i+2} = (0,1)$, and $P_{i+3} = (x,y)$. Positivity of $(P_{i+1}, P_{i+2}, P_{i+3})$ and $(P_{i}, P_{i+1}, P_{i+3})$ implies $x < 0$ and $y > 0$. Positivity of $(P_{i-1}, P_i, P_{i+2})$, $P_{i+3} \in \Int(P_{i-1}, P_i, P_{i+2})$, and Proposition \ref{prop:orientation lemma} shows $y - x > 1$. 
\end{proof}

\begin{proposition} \label{prop:S3 beta in S(K,J)}
    For all $n \geq 2$, $\cS_{3,n}^\beta \subset S_n(K,J)$. 
\end{proposition}

\begin{proof}
    Let $P$ be a $(-3)$-representative of $[P] \in \cS_{3,n}^\beta$ with corner invariants $x_j = x_j(P)$. Lemma \ref{lem:S3 beta in S(K,J)} implies the quadrilateral $(P_{i-2}, P_{i-1}, P_i, P_{i+1})$ is convex. Next, since $P$ is a type-$\beta$ $(-3)$-representative, $P_i \in \Int(P_{i-2}, P_{i-1}, P_{i+1})$ for all $i \geq 0$ (See Figure \ref{fig:beta KJ}). Referring back to Remark \ref{rmk:P2 position and corner invariants}, convexity of $(P_{i-2}, P_{i-1}, P_i, P_{i+1})$ implies $P_{i} P_{i+1}$ doesn't go through $P_{i-2}, P_{i-1}, P_{i+1}$, so $(x_{2i}, x_{2i+1}) \in K \times J$ whenever $P_{i+2} \in \Int(P_{i-2}, P_{i-1}, P_{i+1})$. 
\end{proof}

\begin{figure}[ht]
    \centering
    \footnotesize
    
    \def\svgwidth{0.3\columnwidth}
    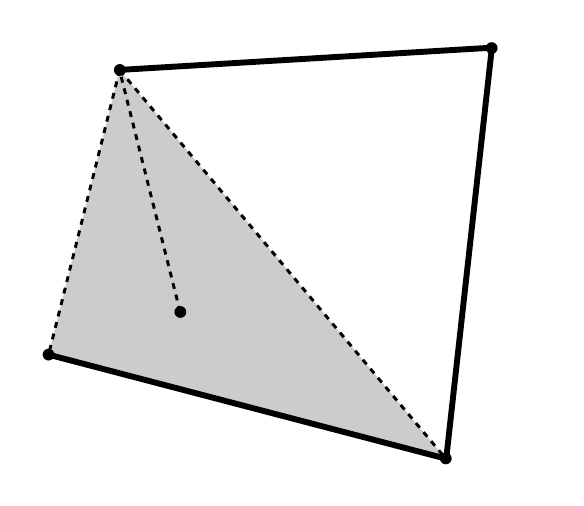

    \caption{Configuration of Proposition \ref{prop:S3 beta in S(K,J)} and Lemma \ref{lem:S3 beta induction}.}
    \label{fig:beta KJ}
\end{figure}

\begin{lemma} \label{lem:S3 beta induction}
    Given a $3$-nice sequence $P: \ZZ \rightarrow \RR\PP^2$ and an integer $i \in \ZZ$, let $x_{2i} = x_{2i}(P)$ and $x_{2i+1} = x_{2i+1}(P)$ be the corner invariants of $P$. If the following conditions are true:
    \begin{itemize}
        \item $(P_{i-2}, P_{i-1}, P_i)$ and $(P_{i-1}, P_i, P_{i+1})$ are both positive;
        \item The quadrilateral $(P_{i-2}, P_{i-1}, P_i, P_{i+1})$ is convex; 
        \item $(x_{2i}, x_{2i+1}) \in K \times J$.
    \end{itemize}
    Then, the following hold:
    \begin{itemize}
        \item $P_{i+2} \in \Int(P_{i-2}, P_{i-1}, P_{i+1})$;
        \item The quadrilateral $(P_{i-1}, P_i, P_{i+1}, P_{i+2})$ is convex;
        \item $(P_{i}, P_{i+1}, P_{i+2})$ and $(P_{i-1}, P_{i}, P_{i+2})$ are both positive. 
    \end{itemize}
\end{lemma}

\begin{proof}
    Recall that from the proof of Proposition \ref{prop:S3 beta in S(K,J)}, we claimed that if the quadrilateral $(P_{i-2}, P_{i-1}, P_i, P_{i+1})$ is convex, then the line $P_i P_{i+1}$ doesn't go through $(P_{i-2}, P_{i-1}, P_{i+1})$. Since $(x_{2i}, x_{2i+1}) \in K \times J$, Remark \ref{rmk:P2 position and corner invariants} implies $P_{i+2} \in \Int(P_{i-2}, P_{i-1}, P_{i+1})$, in which case all conclusions of this lemma will hold. See Figure \ref{fig:beta KJ} for a visualization of the five points. 
\end{proof}

\begin{proposition} \label{prop:S3 beta equals S(J,K)}
    $\cS_{3,n}^\beta = S_n(K,J)$. 
\end{proposition}

\begin{proof}
    Proposition \ref{prop:S3 beta in S(K,J)} gives us $\cS_{3,n}^\beta \subset S_n(K,J)$, so we show the other containment. 
    Given $[P] \in S_n(K,J)$, we can find a representative $P$ that satisfies $P_{N} = (0,0)$, $P_{N+1} = (1,0)$, $P_{N+2} = (1,1)$, $P_{N+3} = (0,1)$. Corollary  \ref{cor:grids are 3 nice} shows that $P$ is 3-nice. To see that $(P_i, P_{i+1}, P_{i+2})$, $(P_{i}, P_{i+1}, P_{i+3})$ are positive, and $P_{i+4} \in \Int(P_{i}, P_{i+1}, P_{i+3})$, we may inductively apply Lemma \ref{lem:S3 beta induction}. This implies $[P] \in \cS_{3,n}^\beta$. 
\end{proof}

%% file: figure/geometry_of_corner_inv.pdf_tex
\begingroup%
  \makeatletter%
  \providecommand\color[2][]{%
    \errmessage{(Inkscape) Color is used for the text in Inkscape, but the package 'color.sty' is not loaded}%
    \renewcommand\color[2][]{}%
  }%
  \providecommand\transparent[1]{%
    \errmessage{(Inkscape) Transparency is used (non-zero) for the text in Inkscape, but the package 'transparent.sty' is not loaded}%
    \renewcommand\transparent[1]{}%
  }%
  \providecommand\rotatebox[2]{#2}%
  \newcommand*\fsize{\dimexpr\f@size pt\relax}%
  \newcommand*\lineheight[1]{\fontsize{\fsize}{#1\fsize}\selectfont}%
  \ifx\svgwidth\undefined%
    \setlength{\unitlength}{373.02904985bp}%
    \ifx\svgscale\undefined%
      \relax%
    \else%
      \setlength{\unitlength}{\unitlength * \real{\svgscale}}%
    \fi%
  \else%
    \setlength{\unitlength}{\svgwidth}%
  \fi%
  \global\let\svgwidth\undefined%
  \global\let\svgscale\undefined%
  \makeatother%
  \begin{picture}(1,0.94225565)%
    \lineheight{1}%
    \setlength\tabcolsep{0pt}%
    \put(0,0){\includegraphics[width=\unitlength,page=1]{geometry_of_corner_inv.pdf}}%
    \put(0.84717148,0.02171805){\color[rgb]{0,0,0}\makebox(0,0)[lt]{\lineheight{1.25}\smash{\begin{tabular}[t]{l}$m_2$\end{tabular}}}}%
    \put(0.6326101,0.00510679){\color[rgb]{0,0,0}\makebox(0,0)[lt]{\lineheight{1.25}\smash{\begin{tabular}[t]{l}$m_3$\end{tabular}}}}%
    \put(0.92884328,0.12692231){\color[rgb]{0,0,0}\makebox(0,0)[lt]{\lineheight{1.25}\smash{\begin{tabular}[t]{l}$m_4$\end{tabular}}}}%
    \put(0.0531037,0.26167445){\color[rgb]{0,0,0}\makebox(0,0)[lt]{\lineheight{1.25}\smash{\begin{tabular}[t]{l}$P_{i-2}$\end{tabular}}}}%
    \put(0.75125978,0.22383375){\color[rgb]{0,0,0}\makebox(0,0)[lt]{\lineheight{1.25}\smash{\begin{tabular}[t]{l}$P_{i-1}$\end{tabular}}}}%
    \put(0.79613548,0.74623313){\color[rgb]{0,0,0}\makebox(0,0)[lt]{\lineheight{1.25}\smash{\begin{tabular}[t]{l}$P_i$\end{tabular}}}}%
    \put(0.15204513,0.6093087){\color[rgb]{0,0,0}\makebox(0,0)[lt]{\lineheight{1.25}\smash{\begin{tabular}[t]{l}$P_{i+1}$\end{tabular}}}}%
    \put(0.33435798,0.91101059){\color[rgb]{0,0,0}\makebox(0,0)[lt]{\lineheight{1.25}\smash{\begin{tabular}[t]{l}$l_1$\end{tabular}}}}%
    \put(0.11515595,0.87357171){\color[rgb]{0,0,0}\makebox(0,0)[lt]{\lineheight{1.25}\smash{\begin{tabular}[t]{l}$l_2$\end{tabular}}}}%
    \put(0.00056514,0.67236519){\color[rgb]{0,0,0}\makebox(0,0)[lt]{\lineheight{1.25}\smash{\begin{tabular}[t]{l}$l_3$\end{tabular}}}}%
  \end{picture}%
\endgroup%

%% file: figure/corner_inv_partition.pdf_tex
\begingroup%
  \makeatletter%
  \providecommand\color[2][]{%
    \errmessage{(Inkscape) Color is used for the text in Inkscape, but the package 'color.sty' is not loaded}%
    \renewcommand\color[2][]{}%
  }%
  \providecommand\transparent[1]{%
    \errmessage{(Inkscape) Transparency is used (non-zero) for the text in Inkscape, but the package 'transparent.sty' is not loaded}%
    \renewcommand\transparent[1]{}%
  }%
  \providecommand\rotatebox[2]{#2}%
  \newcommand*\fsize{\dimexpr\f@size pt\relax}%
  \newcommand*\lineheight[1]{\fontsize{\fsize}{#1\fsize}\selectfont}%
  \ifx\svgwidth\undefined%
    \setlength{\unitlength}{802.51986766bp}%
    \ifx\svgscale\undefined%
      \relax%
    \else%
      \setlength{\unitlength}{\unitlength * \real{\svgscale}}%
    \fi%
  \else%
    \setlength{\unitlength}{\svgwidth}%
  \fi%
  \global\let\svgwidth\undefined%
  \global\let\svgscale\undefined%
  \makeatother%
  \begin{picture}(1,0.46210801)%
    \lineheight{1}%
    \setlength\tabcolsep{0pt}%
    \put(0,0){\includegraphics[width=\unitlength,page=1]{corner_inv_partition.pdf}}%
    \put(0.15061838,0.43762097){\color[rgb]{0.02352941,0,0}\makebox(0,0)[lt]{\lineheight{1.25}\smash{\begin{tabular}[t]{l}$l_1$\end{tabular}}}}%
    \put(0.06237264,0.4288426){\color[rgb]{0,0,0}\makebox(0,0)[lt]{\lineheight{1.25}\smash{\begin{tabular}[t]{l}$l_2$\end{tabular}}}}%
    \put(-0.00060639,0.33495638){\color[rgb]{0,0,0}\makebox(0,0)[lt]{\lineheight{1.25}\smash{\begin{tabular}[t]{l}$l_3$\end{tabular}}}}%
    \put(0.8054468,0.01071491){\color[rgb]{0,0,0}\makebox(0,0)[lt]{\lineheight{1.25}\smash{\begin{tabular}[t]{l}$m_3$\end{tabular}}}}%
    \put(0.89369048,0.02596159){\color[rgb]{0,0,0}\makebox(0,0)[lt]{\lineheight{1.25}\smash{\begin{tabular}[t]{l}$m_2$\end{tabular}}}}%
    \put(0.9732481,0.06371709){\color[rgb]{0,0,0}\makebox(0,0)[lt]{\lineheight{1.25}\smash{\begin{tabular}[t]{l}$m_4$\end{tabular}}}}%
    \put(0.20425046,0.3794492){\color[rgb]{0,0,0}\makebox(0,0)[lt]{\lineheight{1.25}\smash{\begin{tabular}[t]{l}$U_{1,3}$\end{tabular}}}}%
    \put(0.25860686,0.26855027){\color[rgb]{0,0,0}\makebox(0,0)[lt]{\lineheight{1.25}\smash{\begin{tabular}[t]{l}$U_{2,3}$\end{tabular}}}}%
    \put(0.15722501,0.20743627){\color[rgb]{0,0,0}\makebox(0,0)[lt]{\lineheight{1.25}\smash{\begin{tabular}[t]{l}$U_{1,2}$\end{tabular}}}}%
    \put(0.87684363,0.18568228){\color[rgb]{0,0,0}\makebox(0,0)[lt]{\lineheight{1.25}\smash{\begin{tabular}[t]{l}$V_{3,4}$\end{tabular}}}}%
    \put(0.7531338,0.26830186){\color[rgb]{0,0,0}\makebox(0,0)[lt]{\lineheight{1.25}\smash{\begin{tabular}[t]{l}$V_{2,3}$\end{tabular}}}}%
    \put(0.63304034,0.20818129){\color[rgb]{0,0,0}\makebox(0,0)[lt]{\lineheight{1.25}\smash{\begin{tabular}[t]{l}$V_{2,4}$\end{tabular}}}}%
  \end{picture}%
\endgroup%

%% file: figure/alpha_JI.pdf_tex
\begingroup%
  \makeatletter%
  \providecommand\color[2][]{%
    \errmessage{(Inkscape) Color is used for the text in Inkscape, but the package 'color.sty' is not loaded}%
    \renewcommand\color[2][]{}%
  }%
  \providecommand\transparent[1]{%
    \errmessage{(Inkscape) Transparency is used (non-zero) for the text in Inkscape, but the package 'transparent.sty' is not loaded}%
    \renewcommand\transparent[1]{}%
  }%
  \providecommand\rotatebox[2]{#2}%
  \newcommand*\fsize{\dimexpr\f@size pt\relax}%
  \newcommand*\lineheight[1]{\fontsize{\fsize}{#1\fsize}\selectfont}%
  \ifx\svgwidth\undefined%
    \setlength{\unitlength}{186.10863633bp}%
    \ifx\svgscale\undefined%
      \relax%
    \else%
      \setlength{\unitlength}{\unitlength * \real{\svgscale}}%
    \fi%
  \else%
    \setlength{\unitlength}{\svgwidth}%
  \fi%
  \global\let\svgwidth\undefined%
  \global\let\svgscale\undefined%
  \makeatother%
  \begin{picture}(1,0.98415635)%
    \lineheight{1}%
    \setlength\tabcolsep{0pt}%
    \put(0,0){\includegraphics[width=\unitlength,page=1]{alpha_JI.pdf}}%
    \put(0.79082774,0.92775804){\color[rgb]{0,0,0}\makebox(0,0)[lt]{\lineheight{1.25}\smash{\begin{tabular}[t]{l}$P_{i-2}$\end{tabular}}}}%
    \put(0.04360665,0.04178767){\color[rgb]{0,0,0}\makebox(0,0)[lt]{\lineheight{1.25}\smash{\begin{tabular}[t]{l}$P_{i-1}$\end{tabular}}}}%
    \put(0.72745879,0.0440676){\color[rgb]{0,0,0}\makebox(0,0)[lt]{\lineheight{1.25}\smash{\begin{tabular}[t]{l}$P_i$\end{tabular}}}}%
    \put(0.72183984,0.60733854){\color[rgb]{0,0,0}\makebox(0,0)[lt]{\lineheight{1.25}\smash{\begin{tabular}[t]{l}$P_{i+1}$\end{tabular}}}}%
    \put(0.45639161,0.21321632){\color[rgb]{0,0,0}\makebox(0,0)[lt]{\lineheight{1.25}\smash{\begin{tabular}[t]{l}$P_{i+2}$\end{tabular}}}}%
  \end{picture}%
\endgroup%

%% file: figure/beta_KJ.pdf_tex
\begingroup%
  \makeatletter%
  \providecommand\color[2][]{%
    \errmessage{(Inkscape) Color is used for the text in Inkscape, but the package 'color.sty' is not loaded}%
    \renewcommand\color[2][]{}%
  }%
  \providecommand\transparent[1]{%
    \errmessage{(Inkscape) Transparency is used (non-zero) for the text in Inkscape, but the package 'transparent.sty' is not loaded}%
    \renewcommand\transparent[1]{}%
  }%
  \providecommand\rotatebox[2]{#2}%
  \newcommand*\fsize{\dimexpr\f@size pt\relax}%
  \newcommand*\lineheight[1]{\fontsize{\fsize}{#1\fsize}\selectfont}%
  \ifx\svgwidth\undefined%
    \setlength{\unitlength}{280.31079871bp}%
    \ifx\svgscale\undefined%
      \relax%
    \else%
      \setlength{\unitlength}{\unitlength * \real{\svgscale}}%
    \fi%
  \else%
    \setlength{\unitlength}{\svgwidth}%
  \fi%
  \global\let\svgwidth\undefined%
  \global\let\svgscale\undefined%
  \makeatother%
  \begin{picture}(1,0.86614102)%
    \lineheight{1}%
    \setlength\tabcolsep{0pt}%
    \put(0,0){\includegraphics[width=\unitlength,page=1]{beta_KJ.pdf}}%
    \put(0.01167914,0.19421597){\color[rgb]{0,0,0}\makebox(0,0)[lt]{\lineheight{1.25}\smash{\begin{tabular}[t]{l}$P_{i-2}$\end{tabular}}}}%
    \put(0.72523349,0.01675112){\color[rgb]{0,0,0}\makebox(0,0)[lt]{\lineheight{1.25}\smash{\begin{tabular}[t]{l}$P_{i-1}$\end{tabular}}}}%
    \put(0.85126968,0.8169463){\color[rgb]{0,0,0}\makebox(0,0)[lt]{\lineheight{1.25}\smash{\begin{tabular}[t]{l}$P_i$\end{tabular}}}}%
    \put(0.14334738,0.79131343){\color[rgb]{0,0,0}\makebox(0,0)[lt]{\lineheight{1.25}\smash{\begin{tabular}[t]{l}$P_{i+1}$\end{tabular}}}}%
    \put(0.26805291,0.26001356){\color[rgb]{0,0,0}\makebox(0,0)[lt]{\lineheight{1.25}\smash{\begin{tabular}[t]{l}$P_{i+2}$\end{tabular}}}}%
  \end{picture}%
\endgroup%

%% file: 5formula.tex
\section{A Birational Formula for \texorpdfstring{$T_3$}{T(3)}} \label{sec:formula}

Given two spaces $X$ and $Y$, a \textit{rational map} $f: X \dashrightarrow Y$ is an equivalence class of maps $f_U: U \rightarrow Y$ where $U$ is a dense open in $X$, and the equivalence relation is given by $f_U \sim f_V$ if they restrict to the same map on $U \cap V$. A map $f: X \dashrightarrow Y$ is \textit{birational} if there exists a rational map $g: Y \dashrightarrow X$ such that $g \circ f$ restricts to the identity on a dense open of $X$ and $f \circ g$ restricts to an identity on a dense open of $Y$. 

In this section, we show that $T_3: \cP_{n} \dashrightarrow \cP_{n}$ is a birational map by finding an explicit formula using the corner invariants. 

\subsection{The Formula} \label{subsec:T3 map formula and invariants}

Let $P$ be a twisted $n$-gon, and $P' = T_3(P)$. In this section, we use a different labeling convention:
\begin{equation} \label{eqn:delta 3,1 formula}
    P'_i = P_{i-2} P_{i+1} \cap P_{i-1} P_{i+2}.
\end{equation}

We let $x_j = x_j(P)$ and $x_j' = x_j(P')$ denote the corner invariants of $P$ and $P'$ respectively. Our goal is to show that $T_3$ is a birational map over the corner invariants. I discovered it using computer algebra and the reconstruction formula in \cite[Equation (19)]{schwartz2008discretemonodromypentagramsmethod}.

\begin{proposition} \label{prop:coordinate formula}
    Given $[P] \in \cP_{3,n}$, the following formula holds \textnormal{(}indices taken modulo $2n$\textnormal{)}:
    \begin{equation} \label{eqn:31 corner}
        \left\{
            \begin{aligned}
                x'_{2i} = x_{2i-2} \cdot \frac{(x_{2i-4} + x_{2i-1} - 1)}{x_{2i-2}x_{2i-1} - (1 - x_{2i+1})(1 - x_{2i-4})}; \\
                x'_{2i+1} = x_{2i+3} \cdot \frac{(x_{2i+2} + x_{2i+5} - 1)}{x_{2i+2}x_{2i+3} - (1 - x_{2i+5})(1 - x_{2i})}.
            \end{aligned}
        \right.
    \end{equation}
\end{proposition}

One can verify Equation \eqref{eqn:31 corner} with the following procedure: Given the corner invariants of $[P]$, use the reconstruction formula from \cite[Equation (19)]{schwartz2008discretemonodromypentagramsmethod} to obtain a representative $P$. Apply $T_3$ on $P$ as in Equation \eqref{eqn:delta 3,1 formula} to get $P' = T_3(P)$. Then, compute the corner invariants of $P'$. 
We present a geometric proof of Equation \eqref{eqn:31 corner} using cross-ratio identities. We start with the following lemma, which is a classical observation in projective geometry called ``quadrangular sets.'' 

\begin{lemma} \label{lem:quadrangle}
    Let $Q_1, Q_2, Q_3, Q_4$ be four points in general position, and let $\omega$ be a line that contains none of the four points. For all $i \neq j$, let $l_{ij} = Q_iQ_j$ and $S_{ij} = \omega \cap l_{ij}$ Then, 
    \begin{equation*}
        \chi(S_{12}, S_{13}, S_{14}, S_{24}) = \chi(S_{23}, S_{13}, S_{34}, S_{24}).
    \end{equation*}
\end{lemma}

\begin{figure}[ht]
    \centering
    \footnotesize
    
    \def\svgwidth{0.7\columnwidth}
    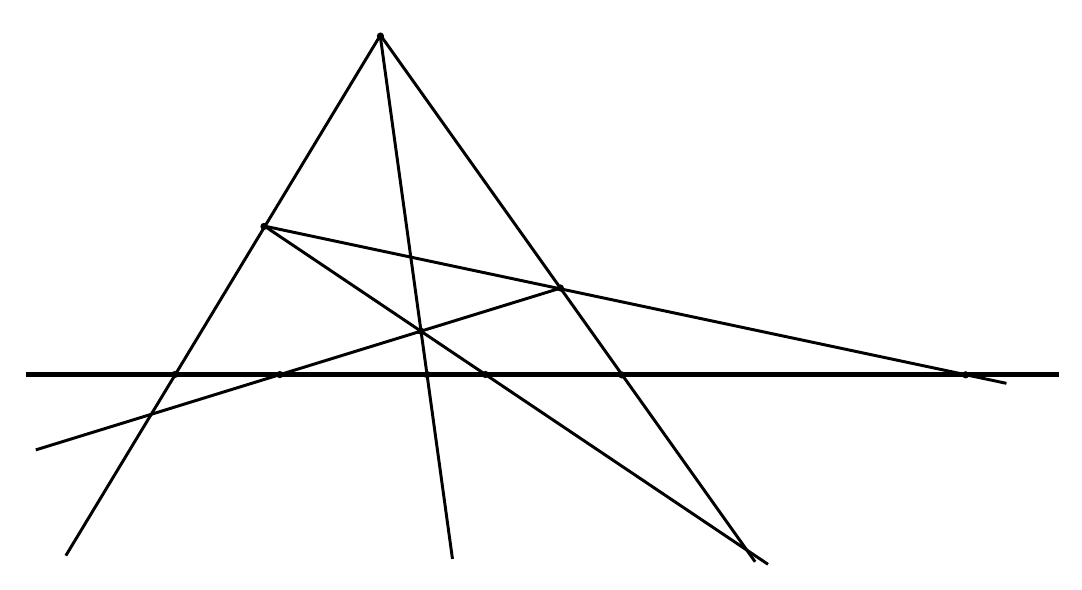

    \caption{Point configurations of Lemma \ref{lem:quadrangle}}
    \label{fig:quadrangle section}
\end{figure}

\begin{proof}
    Let $O = l_{13} \cap l_{24}$. See Figure \ref{fig:quadrangle section} for an example of the point configurations. Applying Equation \eqref{eqn:chi prop} on $(l_{12}, l_{13}, l_{14}, Q_1D)$ with respect to $\omega$ and $Q_2Q_4$ gives us 
    \begin{equation*}
        \chi(S_{12}, S_{13}, S_{14}, S_{24}) 
        \overset{\omega}{=} \chi(l_{12}, l_{13}, l_{14}, Q_1D)
        \overset{l_{24}}{=} \chi(Q_2, O, Q_4, S_{24}).
    \end{equation*}
    Next, applying Equation \eqref{eqn:chi prop} twice on $(l_{23}, l_{13}, l_{34}, Q_3D)$ with respect to $l_{24}$ and $\omega$ gives us 
    \begin{equation*}
        \chi(Q_2, O, Q_4, S_{24})
        \overset{l_{24}}{=} \chi(l_{23}, l_{13}, l_{34}, Q_3D) 
        \overset{\omega}{=} \chi(S_{23}, S_{13}, S_{34}, S_{24}). 
    \end{equation*}
    Combining the above two equations completes the proof. 
\end{proof}

\begin{figure}[ht]
    \centering
    \scriptsize
    
    \def\svgwidth{0.6\columnwidth}
    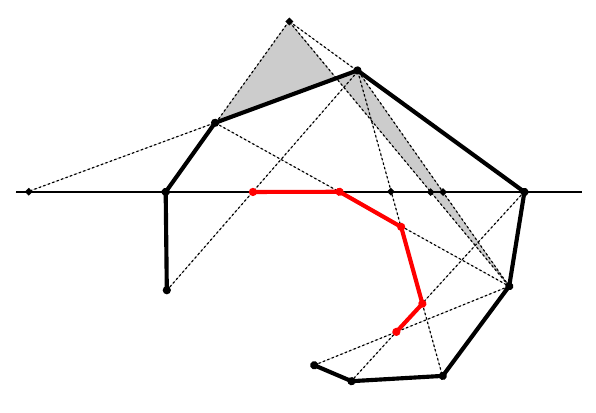

    \caption{Visualization of Points Assigned in Equation \eqref{eqn:eight points}. The thick black line segments are edges connecting vertices of $P$, and the thick red line segments are edges connecting vertices of $P'$.}
    \label{fig:31 formula config}
\end{figure}

\begin{proof}[Proof of Proposition \ref{prop:coordinate formula}]
    From the symmetry of Equation \eqref{eqn:31 corner}, it suffices to prove the formula for $x'_0$. That is,
    \begin{equation} \label{eqn:coordinate formula x0}
        x'_0 = \frac{x_{-2} (x_{-4} + x_{-1} - 1)}{x_{-2}x_{-1} - (1 - x_{-4})(1 - x_1)}.
    \end{equation}

    Let $l_{i,j} = P_i P_j$ and $O = l_{-3,-2} \cap l_{-1,0}$. We label points as follows:
    \begin{equation} \label{eqn:eight points}
        \begin{aligned}
            & A = P'_{-2} ; \\ 
            & B = P'_{-1} ;
        \end{aligned}
        \hspace{20pt}
        \begin{aligned}
            & C = l_{-3,0} \cap l_{-2,1} ; \\ 
            & D = P_0 ;
        \end{aligned}
        \hspace{20pt}
        \begin{aligned}
            & E = P_{-3} ; \\ 
            & F = l_{-3,0} \cap l_{-1,1} ;
        \end{aligned}
        \hspace{20pt}
        \begin{aligned}
            & G = l_{-3,0} \cap l_{-2,-1} ; \\
            & H = l_{-3,0} \cap O P_1 .
        \end{aligned}
    \end{equation}

    Since $[P] \in \cP_{3,n}$, every five consecutive points of $[P]$ are in general position. This ensures that point $O$ and the points in Equation \eqref{eqn:eight points} are all distinct. See Figure \ref{fig:31 formula config} for a visualization of the assignment of labels to these points.

    It follows from Equation \eqref{eqn:corner invariant def} that $x'_0 = \chi(A, B, C, D)$. Using Equation \eqref{eqn:corner invariant lines def}, we have  
    \begin{equation} \label{eqn:corner invariants sub}
        \begin{aligned}
            x_{-4} 
            &= \chi(l_{-1,-4},\, l_{-1,-3},\, l_{-1,-2},\, l_{-1,0})
            \overset{l_{-3,0}}{=} \chi(A, E, G, D); \\
            x_{-2} 
            &= \chi(l_{0,-3},\, l_{0,-2},\, l_{0,-1},\, l_{0,1})
            \overset{l_{-3,0}}{=} \chi(E, B, H, D); \\
            x_{-1} 
            &= \chi(l_{-2,1},\, l_{-2,0},\, l_{-2,-1},\, l_{-2,-3})
            \overset{l_{-3,0}}{=} \chi(B, D, G, E); \\
            x_1 
            &= \chi(l_{-1,2},\, l_{-1,1},\, l_{-1,0},\, l_{-1,-2})
            \overset{l_{-3,0}}{=} \chi(C, F, D, G).
        \end{aligned}
    \end{equation}
    
    We may further invoke Lemma \ref{lem:quadrangle} with $Q_1 = P_{-2}$, $Q_2 = O$, $Q_3 = P_1$, $Q_4 = P_{-1}$, and $\omega = l_{-3,0}$. This gives us 
    \begin{equation} \label{eqn:x_-2 sub}
        x_{-2} = \chi(E, B, H, D) = \chi(G, B, F, D). 
    \end{equation}

    The rest of the proof is just algebraic verification. Normalize with a projective transformation so that $l_{-3,0}$ is the $x$-axis of $\AA^2$. Let $a$, $b$, $c$, $d$, $e$, $f$, $g$, $h$ be coordinates of $A$, $B$, $C$, $D$, $E$, $F$, $G$, $H$ respectively. Plugging \eqref{eqn:corner invariants sub} and \eqref{eqn:x_-2 sub} into the numerator of \eqref{eqn:coordinate formula x0} gives us
    \begin{equation*} \label{eqn:num}
        \begin{aligned}
            x_{-2} (x_{-4} + x_{-1} - 1) 
            &= \chi(G, B, F, D) \left( \chi(A, E, G, D) + \chi(B, D, G, E)- 1 \right) \\
            &= \, \frac{(g - b)(f - d)}{(g - f)(b - d)} 
            \left( \frac{(a - e)(g - d)}{(a - g)(e - d)} + \frac{(b - e)(g - d)}{(b - g)(e - d)} \right) \\
            &= \, \frac{(a - b)(g - d)(e - g)(d - f)}{(a - g)(b - d)(e - d)(g - f)} .
        \end{aligned}
    \end{equation*}
    The denominator can be computed similarly. We skip the computation and list the results: 
    \begin{equation*} \label{eqn:denom}
        x_{-2}x_{-1} - (1 - x_{-4})(1 - x_1)
        = \frac{(a - c)(g - d)(d - f)(e - g)}{(a - g)(c - d)(d - e)(f - g)}.
    \end{equation*}
    Combining the above two equations gives us 
    \begin{equation*}
        \frac{x_{-2}(x_{-4} + x_{-1} - 1)}{x_{-2}x_{-1} - (1 - x_{-4})(1 - x_{1})} 
        = \frac{(a - b)(c - d)}{(a - c)(b - d)} 
        = \chi(A, B, C, D) 
        = x'_0 ,
    \end{equation*}
    which is precisely Equation \eqref{eqn:coordinate formula x0}.
\end{proof}

Next, we provide a formula for the inverse of $T_3$. 

\begin{proposition} \label{prop:coordinate inverse formula}
    The map $T_3: \cP_n \dashrightarrow \cP_n$ is birational. Its inverse is given by 
    \begin{equation} \label{eqn:T3 inverse corner}
        \left\{
            \begin{aligned}
                x_{2i} = x'_{2i+2} \cdot \frac{(x'_{2i+4} + x'_{2i+1} - 1)}{x'_{2i+1} x'_{2i+2} - (1 - x'_{2i-1})(1 - x'_{2i+4})}; \\
                x_{2i+1} = x'_{2i-1} \cdot \frac{(x'_{2i-3} + x'_{2i} - 1)}{x'_{2i} x'_{2i-1} - (1 - x'_{2i+2})(1 - x'_{2i-3})}.
            \end{aligned}
        \right.
    \end{equation}
\end{proposition}

We will give an algebraic proof. Consider two families of rational maps $\{\mu_{(s,t)}: \RR^{2n} \dashrightarrow \RR^{2n}\}_{(s,t) \in \ZZ^2}$ and $\{\nu_{(s,t)}: \RR^{2n} \dashrightarrow \RR^{2n}\}_{(s,t) \in \ZZ^2}$. Write $(a_0, \ldots, a_{2n-1}) = \mu_{(s,t)}(x_0, \ldots, x_{2n-1})$ and $(b_0, \ldots, b_{2n-1}) = \nu_{(s,t)}(x_0, \ldots, x_{2n-1})$. Then, we set 
\begin{equation} \label{eqn:mu nu def}
    \left\{ 
        \begin{aligned}
            & a_{2i} = \frac{1 - x_{2i+s}}{x_{2i+s+t}} \\
            & a_{2i+1} = \frac{1 - x_{2i+1-s}}{x_{2i+1-s-t}};
        \end{aligned}
    \right. 
    \ \ \ \ 
    \left\{ 
        \begin{aligned}
            & b_{2i} = \frac{1 - x_{2i+s}}{1 - x_{2i+s} x_{2i+s+t}} \\
            & b_{2i+1} = \frac{1 - x_{2i+1-s}}{1 - x_{2i+1-s} x_{2i+1-s-t}}.
        \end{aligned}
    \right. 
\end{equation}

\begin{lemma} \label{lem:mu nu birational}
    Let $\varphi: \ZZ^2 \rightarrow \ZZ^2$ be the map given by 
    \begin{equation}
        \varphi(s,t) = ((-1)^{s+1}s,\, (-1)^s (2s + t)).
    \end{equation}
    Then, $\varphi$ is an involution. Moreover, when $t$ is odd, $\mu_{(s,t)}^{-1} = \nu_{\varphi(s,t)}$ and $\nu_{(s,t)}^{-1} = \mu_{\varphi(s,t)}$. 
\end{lemma}

\begin{proof}
    To see $\varphi$ is an involution, a direct computation shows that
    \begin{equation*}
        \begin{aligned}
            \varphi^2(s,t)
            &= \varphi((-1)^{s+1}s,\, (-1)^s (2s + t)) \\
            &= \left( (-1)^{(-1)^{s+1}s+s+2}s, \, (-1)^{(-1)^{s+1}s} \, (2(-1)^{s+1}s + (-1)^s(2s + t)) \right) = (s,t).
        \end{aligned}
    \end{equation*}
    Next, we show that when $t$ is odd, $\mu^{-1}_{(s,t)} = \nu_{\varphi(s,t)}$. We will show by direct computation that $\mu_{(s,t)} \circ \nu_{\varphi(s,t)}$ is the identity on the $2i$-th coordinate when $s$ is even. First, when $s$ is even, $\varphi(s,t) = (-s, 2s+t)$. The $2i$-th coordinate of $\mu_{(s,t)} \circ \nu_{\varphi(s,t)}$ is given by 
    \begin{equation*}
        \begin{aligned}
            &\left( 1 - \frac{1 - x_{2i+s+(-s)}}{1 - x_{2i+s+(-s)} x_{2i + s + (-s) + (2s + t)}} \right) \cdot \left( \frac{1 - x_{2i + s + t - (-s)}}{1 - x_{2i + s + t - (-s)} x_{2i + s + t - (-s) - (2s + t)}} \right)^{-1} \\
            &= \left( 1 - \frac{1 - x_{2i}}{1 - x_{2i} x_{2i + 2s + t}} \right) \left( \frac{1 - x_{2i + 2s + t} x_{2i}}{1 - x_{2i + 2s + t}} \right) = x_{2i}.
        \end{aligned}
    \end{equation*}
    This is precisely what we want. One can similarly carry out the computation of $\nu_{\varphi(s,t)} \circ \mu_{s,t}$ for the $(2i+1)$-th coordinate, and $s$ odd. We will omit these heavy computations and conclude that $\mu_{(s,t)}^{-1} = \nu_{\varphi(s,t)}$. Finally, to see $\nu_{(s,t)}^{-1} = \mu_{\varphi(s,t)}$, observe that $(-1)^s(2s + t)$ is odd iff $t$ is odd. Therefore, $\nu_{s,t} \circ \mu_{\varphi(s,t)} = \nu_{\varphi^2(s,t)} \circ \mu_{\varphi(s,t)}$ is the identity map by the previous argument. The same argument shows that $\mu_{\varphi(s,t)} \circ \nu_{(s,t)}$ is the identity. 
\end{proof}

The following corollary is immediate. We omit the proof. 

\begin{corollary} \label{cor:mu nu birational}
    For all $(s,t) \in \ZZ^2$ such that $t$ is odd, $\mu_{(s,t)}$ and $\nu_{(s,t)}$ are birational maps. 
\end{corollary}

\begin{proof}[Proof of Proposition \ref{prop:coordinate inverse formula}]
    We first claim that $T_3 = \nu_{(-1, -1)} \circ \mu_{(3, -3)}$. We will provide the computation for even coordinates. Let $(a_0, \ldots, a_{2n-1})$ denote the image of $(x_0, \ldots, x_{2n-1})$ under $\mu_{(3, -3)}$, and let $(b_0, \ldots, b_{2n-1})$ denote the image of $(a_0, \ldots, a_{2n-1})$ under $\nu_{(-1,-1)}$. Then, we have 
    \begin{equation*}
        \begin{aligned}
            b_{2i} 
            &= \frac{1 - a_{2i-1}}{1 - a_{2i-1} a_{2i-2}}
            = \left( 1 - \frac{1 - x_{2i-4}}{x_{2i - 1}} \right) \cdot \left( 1 - \frac{(1 - x_{2i-4})(1 - x_{2i+1})}{x_{2i - 1} x_{2i-2}} \right)^{-1} \\
            &= \frac{x_{2i-2} (x_{2i-1} + x_{2i-4} + 1)}{x_{2i-1} x_{2i-2} - (1 - x_{2i-4})(1 - x_{2i+1})}.
        \end{aligned}
    \end{equation*}
    Observe that this is precisely the first line of \eqref{eqn:31 corner}. The computation for $b_{2i+1}$ is analogous, thus omitted. Then, by Corollary \ref{cor:mu nu birational}, $T_3^{-1} = \nu_{(3, -3)} \circ \mu_{(-1, 3)}$. Finally, Equation \eqref{eqn:T3 inverse corner} follows from a direct computation of $\nu_{(3, -3)} \circ \mu_{(-1, 3)}$ using Equation \eqref{eqn:mu nu def}, which we will omit.
\end{proof}

\subsection{Conjugated Corner Invariants and Its \texorpdfstring{$T_3$}{T3} Formula}

To relate Equation \eqref{eqn:31 corner} to parameters $(y_r)_{r \in \ZZ^2}$ in \cite{GP2016ymeshes}, it is convenient to consider another coordinate system of $\cP_n$, which we define below. 

\begin{definition} \label{def:conj corner inv}
    Given $[P] \in \cP_n$, define the \textit{conjugated corner invariants} to be coordinate functions $\tilde x_0(P), \ldots, \tilde x_{2n-1}(P)$ given by $\tilde x_j(P) = \frac{x_j(P)}{x_j(P)-1}$. 
\end{definition}

The conjugated corner invariants can be viewed as the image of the corner invariants under a birational map $\lambda: \RR^{2n} \dashrightarrow \RR^{2n}$ sending each coordinate $x_j \mapsto \frac{x_j}{x_j-1}$. Observe that $\lambda^2$ restricted to the dense open set $(\RR - \{0,1\})^{2n}$ is the identity map, so $\tilde x_j(P)$ is also a coordinate system for $\cP_n$. Geometrically, the map $\lambda$ corresponds to a different choice of permutation in the cross-ratio. 

Throughout this section, we will use $\tilde x_j = \tilde x_j(P)$ and $\tilde x_j' = \tilde x_j(P')$ to denote the conjugate corner invariants of $P$ and $P'$. 
We start by observing some symmetries of conjugating our factorization maps $\mu_{(s,t)}$ and $\nu_{(s,t)}$ from Equation \eqref{eqn:mu nu def}. 

\begin{lemma} \label{lem:lambda conjugate mu and nu}
    For all $(s,t) \in \ZZ^2$, we have $\lambda \circ \mu_{(s,t)} \circ \lambda = \nu_{(s+t,-t)}$. 
\end{lemma}

\begin{proof} 
    We can check this by direct computation. We show that the equation holds on even coordinates. The $2i$-th coordinate of $\mu_{(s,t)} \circ \lambda$ is given by
    \begin{equation*}
        \frac{1 - x_{2i+s} \cdot (x_{2i+s} - 1)^{-1}}{x_{2i+s+t} \cdot (x_{2i+s+t} - 1)^{-1}}
        = \frac{1 - x_{2i+s+t}}{x_{2i+s+t} (x_{2i+s} - 1)}
    \end{equation*}
    The $2i$-th coordinate of $\lambda \circ \mu_{(s,t)} \circ \lambda$ is given by
    \begin{equation*}
        \left( \frac{1 - x_{2i+s+t}}{x_{2i+s+t} (x_{2i+s} - 1)} \right) \cdot \left( \frac{1 - x_{2i+s+t}}{x_{2i+s+t} (x_{2i+s} - 1)} - 1 \right)^{-1}
        = \frac{1 - x_{2i+s+t}}{1 - x_{2i+s+t}x_{2i+s}},
    \end{equation*}
    which is precisely the $2i$-th coordinate of $\nu_{(s+t, -t)}$. The computation for the odd coordinates is similar.
\end{proof}

Since $\lambda$ is an involution, it immediately follows that $\lambda \circ \nu_{(s,t)} \circ \lambda = \mu_{(s+t, -t)}$. This allows us to obtain a formula for $T_3$ with respect to the conjugated corner invariants.

\begin{proposition} \label{prop:T3 formula conj corner inv}
    Given any 3-nice twisted $n$-gon $P$, the following formula holds (indices taken modulo $2n$):
    \begin{equation} \label{eqn:T3 formula conj corner inv}
        \left\{ 
            \begin{aligned}
                & \tilde x'_{2i} = \tilde x_{2i-2} \cdot \frac{(1 - \tilde x_{2i-1}\tilde x_{2i-4})(1 - \tilde x_{2i+1})}{(1 - \tilde x_{2i+1}\tilde x_{2i-2})(1 - \tilde x_{2i-1})} ; \\
                & \tilde x'_{2i+1} = \tilde x_{2i+3} \cdot \frac{(1 - \tilde x_{2i+2}\tilde x_{2i+5})(1 - \tilde x_{2i})}{(1 - \tilde x_{2i}\tilde x_{2i+3})(1 - \tilde x_{2i+2})} . 
            \end{aligned}
        \right.
    \end{equation}
\end{proposition}

\begin{proof}
    From the proof of Proposition \ref{prop:coordinate inverse formula}, we saw that the formula for $T_3$ on the corner invariants is given by $\nu_{(-1,-1)} \circ \mu_{(3,-3)}$. It follows that the formula for conjugated corner invariants is $\lambda \circ \left( \nu_{(-1,-1)} \circ \mu_{(3,-3)} \right) \circ \lambda$. By Lemma \ref{lem:lambda conjugate mu and nu}, 
    \begin{equation*}
        \lambda \circ \left( \nu_{(-1,-1)} \circ \mu_{(3,-3)} \right) \circ \lambda
        = \left( \lambda \circ \nu_{(-1,-1)} \circ \lambda \right) \circ \left( \lambda \circ \mu_{(3,-3)} \circ \lambda \right)
        = \mu_{(-2,1)} \circ \nu_{(0,3)} .
    \end{equation*}
    It remains to check that $\mu_{(-2,1)} \circ \nu_{(0,3)}$ agrees with Equation \eqref{eqn:T3 formula conj corner inv}. The $2i$-th coordinate of $\mu_{(-2,1)} \circ \nu_{(0,3)}$ is given by 
    \begin{equation*}
        \left( 1 - \frac{1 - \tilde x_{2i-2}}{1 - \tilde x_{2i-2} \tilde x_{2i+1}} \right) \cdot \left( \frac{1 - \tilde x_{2i-1}}{1 - \tilde x_{2i-1} \tilde x_{2i-4}} \right)^{-1}
        = \frac{\tilde x_{2i-2} (1 - \tilde x_{2i-1} \tilde x_{2i-4})(1 - \tilde x_{2i+1})}{(1 - \tilde x_{2i-2} \tilde x_{2i+1})(1 - \tilde x_{2i-1})}.
    \end{equation*}
    This is precisely $\tilde x'_{2i}$ from Equation \eqref{eqn:T3 formula conj corner inv}. The computation for odd coordinates is omitted. 
\end{proof}

Using Lemma \ref{lem:mu nu birational}, we can easily compute the formula of $T_3^{-1}$ with respect to the conjugated corner invariants. The proof is again a direct computation, so we omit it.

\begin{corollary} \label{cor:T3 inverse formula conj corner inv}
    The formula for $T_3^{-1}$ with conjugated corner invariants is given by $\mu_{(0,3)} \circ \nu_{(2, -3)}$. More specifically, 
    \begin{equation} \label{eqn:T3 inverse formula conj corner inv}
        \left\{ 
            \begin{aligned}
                & \tilde x_{2i} = \tilde x'_{2i+2} \cdot \frac{(1 - \tilde x'_{2i+1} \tilde x'_{2i+4})(1 - \tilde x'_{2i-1})}{(1 - \tilde x'_{2i-1}\tilde x'_{2i+2})(1 - \tilde x'_{2i+1})} ; \\
                & \tilde x_{2i+1} = \tilde x'_{2i-1} \cdot \frac{(1 - \tilde x'_{2i} \tilde x'_{2i-3})(1 - \tilde x'_{2i+2})}{(1 - \tilde x'_{2i+2} \tilde x'_{2i-1})(1 - \tilde x'_{2i})} .
            \end{aligned}
        \right.
    \end{equation}
\end{corollary}

\subsection{Relation to \texorpdfstring{$Y$}{Y}-Variables} \label{subsec:Y-variable comparison}

In this section, we discuss how Equation \eqref{eqn:delta 3,1 formula} generalizes the results from \cite{GP2016ymeshes}. 
The propositions in this section hold for all four cells $S_n(I,J)$, $S_n(J,I)$, $S_n(K,J)$, $S_n(J,K)$. For notational convenience, our statements will only mention $S_n(J, I)$. The readers may assume that the propositions hold for the other three cells with the same proof. 

The map $T_3$ along with the labeling convention of Equation \eqref{eqn:delta 3,1 formula} corresponds to the following construction in \cite{GP2016ymeshes}. Let $a, b, c, d \in \ZZ^2$ be distinct and assume $a_2 \leq b_2 \leq c_2 \leq d_2$. Say that $S = \{a, b, c, d\}$ is a \textit{$Y$-pin} if $b_2 < c_2$ and the vectors $b - a$, $c - a$, $d - a$ generate all of $\ZZ^2$. 

\begin{definition}[{\cite[Definition 1.4]{GP2016ymeshes}}] \label{def:Y-mesh}
    Let $S = \{a, b, c, d\}$ be a $Y$-pin and suppose $D \geq 2$. A \textit{$Y$-mesh of type $S$ and dimension $D$} is a grid of points $\hat P_{i,j}$ in $\RR\PP^D$ with $i, j \in \ZZ$ which together span all of $\RR\PP^D$ and such that 
    \begin{itemize}
        \item $\hat P_{r+a}$, $\hat P_{r+b}$, $\hat P_{r+c}$, $\hat P_{r+d}$ are distinct for all $r \in \ZZ^2$.
        \item Let $L_r = \hat P_{r+a} \hat P_{r+b}$. Then, $\hat P_{r+a}$, $\hat P_{r+b}$, $\hat P_{r+c}$, $\hat P_{r+d}$ all lie on $L_r$ for all $r \in \ZZ^2$.
        \item The four lines $L_{r-a}$, $L_{r-b}$, $L_{r-c}$, $L_{r-d}$ (all of which contain $\hat P_r$) are distinct for all $r \in \ZZ^2$.
    \end{itemize}
\end{definition}

Let $S = \{(-1,0), (2,0), (0,1), (1,1)\}$, which is a $Y$-pin. Given a representative $P$ of some $[P] \in S_n(J,I)$, we can consider a grid $(\hat P_{i,j})_{(i,j) \in \ZZ^2}$ where $\hat P_{i,j}$ is the $i$-th vertex of $T_3^j (P)$. 

\begin{proposition} \label{prop:Pij is Y-mesh}
    $(\hat P_{i,j})$ is a $Y$-mesh of type $S$ and dimension 2. 
\end{proposition}

\begin{proof}
    The first two conditions of Definition \ref{def:Y-mesh} are straightforward to verify using the identification $S_n(J,I) = \cS_{3,n}^\alpha$ from Proposition \ref{prop:S3 alpha equals in S(J,I)}. For the third condition, let $P^{(j)} = T_3^j(P)$. Then, we have 
    \begin{equation*}
        L_{r-a} = P^{(j)}_{i-1} P^{(j)}_{i+2}, \hspace{10pt}
        L_{r-b} = P^{(j)}_{i-1} P^{(j)}_{i-4}, \hspace{10pt}
        L_{r-c} = P^{(j)}_{i-1} P^{(j)}_{i}, \hspace{10pt}
        L_{r-d} = P^{(j)}_{i-1} P^{(j)}_{i-2}.
    \end{equation*}
    Notice also that $L_{r-a} = P^{(j+1)}_{i} P^{(j+1)}_{i+1}$ and $L_{r-b} = P^{(j+1)}_{i-2} P^{(j+1)}_{i-1}$, so 3-niceness of $P^{(j+1)}$ implies they are distinct. The other pairings are distinct because of 3-niceness of $P^{(j)}$. 
\end{proof}

\cite{GP2016ymeshes} then introduces the parameters $y_r(\hat P)$ associated to a $Y$-mesh. Fix a $Y$-pin $S = \{a,b,c,d\}$ and a $Y$-mesh $\hat P$ of type $S$ and dimension $D$. For all $r \in \ZZ^2$, consider 
\begin{equation} \label{eqn:def of y-variables}
    y_r(\hat P) = - \chi(\hat P_{r+a}, \hat P_{r+c}, \hat P_{r+d}, \hat P_{r+b}).
\end{equation}
See the left side of Figure \ref{fig:y-variables} for the setup using the $Y$-mesh from Proposition \ref{prop:Pij is Y-mesh}. \cite[Theorem 1.6]{GP2016ymeshes} give us the following relation on $y_r$:
\begin{equation} \label{eqn:y-mesh formula}
    y_{i+1,j} \, y_{i+1,j+2}
    = \frac{(1 + y_{i-1,j+1})(1 + y_{i+3,j+1})}{(1 + y_{i,j+1}^{-1})(1 + y_{i+2,j+1}^{-1})}.
\end{equation}

\begin{figure}[ht]
    \centering
    \footnotesize
    
    \def\svgwidth{0.8\columnwidth}
    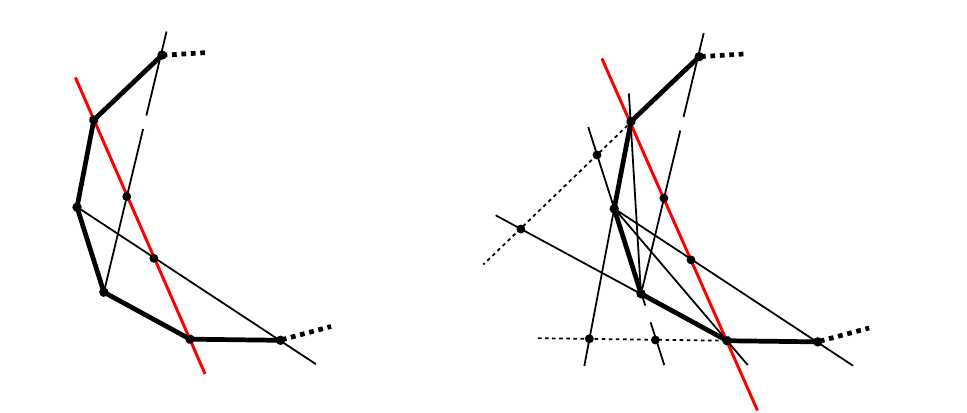

    \caption{Left: Definition of $y_r(\hat P)$ for the $Y$-mesh from Proposition \ref{prop:Pij is Y-mesh}. Right: Relationship between $y_r(\hat P)$ and conjugated corner invariants.}
    \label{fig:y-variables}
\end{figure}

\begin{lemma}
    Given a representative $P$ of $[P] \in S_n(I,J)$ with conjugated corner invariants $\tilde x_j = \tilde x_j(P)$. Let $(\hat P_{i,j})$ be its corresponding $Y$-mesh with $y_r = y_r(\hat P)$ for all $r \in \ZZ^2$. Then, for all $i \in \ZZ$, 
    \begin{equation} \label{eqn:y-variable and corner inv}
        y_{i,0} = - \tilde x_{2i} \tilde x_{2i+3}.
    \end{equation}
\end{lemma}

\begin{proof}
    Let $l_{a,b} = P_{i+a}P_{i+b}$. See right side of Figure \ref{fig:y-variables} for the setup. Equation \eqref{eqn:corner invariant lines def} gives us 
    \begin{equation*}
        \tilde x_{2i} = \frac{x_{2i}}{x_{2i} - 1}
        = \chi(l_{1,-2}, l_{1,-1}, l_{1,2}, l_{1,0}); \hspace{10pt}
        \tilde x_{2i+3} = \frac{x_{2i+3}}{x_{2i+3} - 1}
        = \chi(l_{0,3}, l_{0,2}, l_{0,-1}, l_{0,1}).
    \end{equation*}
    Notice that $(l_{1,-1} \cap l_{0,-1}) \cap (l_{1,2} \cap l_{0,2}) = P_{i-1} P_{i+2}= l_{-1,2}$. Then, from elementary cross ratio identities we have 
    \begin{equation*}
        \begin{aligned}
            \tilde x_{2i} \tilde x_{2i+3}
            &= \chi(l_{1,-1} \cap l_{-1,2}, l_{1,-2} \cap l_{-1,2}, l_{0,3} \cap l_{-1,2}, l_{0,2} \cap l_{-1,2}) \\
            &= \chi (\hat P_{i-1,0}, \hat P_{i,1}, \hat P_{i+1,1}, \hat P_{i+2,0}) = - y_{i,0},
        \end{aligned}
    \end{equation*}
    which is precisely Equation \eqref{eqn:y-variable and corner inv}. 
\end{proof}

\begin{remark}
    Equation \eqref{eqn:y-variable and corner inv} is very similar to the correspondence of $y_r$ and corner invariants in the $T_2$ case. Let $P$ be an arbitrary twisted $n$-gon with $P' = T_2(P)$. If we use the labeling convention $P'_i = P_{i-1} P_{i+1} \cap P_i P_{i+2}$, then the $T_2$-orbit $(\hat P_{i,j})_{(i,j) \in \ZZ^2}$, where $\hat P_{i,j}$ is the $i$-th vertex of $T_2^j(P)$, is a $Y$-mesh of type $S = \{(-1,0), (1,0), (-1,1), (0,1)\}$. Denote by $x_j = x_j(P)$ the corner invariants of $P$. Then, for all $i \in \ZZ$, 
    \begin{equation} \label{eqn:y-variable of T2}
        y_{i,0} = -x_{2i+1} x_{2i+2} .
    \end{equation}
    For the proof of Equation \eqref{eqn:y-variable of T2}, see \cite[Equation (2.2)]{glick2011Ypatterns}.
\end{remark}

\begin{theorem}
    For the $Y$-pin $S = \{(-1,0), (2,0), (0,1), (1,1)\}$, the transformation formula of $y_r$ from \textnormal{\cite[Theorem 1.6]{GP2016ymeshes}} is a direct consequence of the birational formula for the conjugated corner invariants under $T_3$.
\end{theorem}

\begin{proof}
    It suffices to show that we can use Equation \eqref{eqn:y-variable and corner inv} to derive \eqref{eqn:y-mesh formula} for $j = -1$. We first compute $y_{i+1,-1}$ and $y_{i+1, 1}$ using Equation \eqref{eqn:T3 formula conj corner inv} and \eqref{eqn:T3 inverse formula conj corner inv}:
    \begin{equation}
        \begin{aligned}
            y_{i+1,-1} 
            &= - \frac{\tilde x_{2i+4} (1 - \tilde x_{2i+3} \tilde x_{2i+6}) (1 - \tilde x_{2i+1})}{(1 - \tilde x_{2i+1} \tilde x_{2i+4}) (1 - \tilde x_{2i+3})} 
            \cdot \frac{\tilde x_{2i+3} (1 - \tilde x_{2i+1} \tilde x_{2i+4}) (1 - \tilde x_{2i+6})}{(1 - \tilde x_{2i+3} \tilde x_{2i+6}) (1 - \tilde x_{2i+4})} \\
            &= - \frac{\tilde x_{2i+3} \tilde x_{2i+4} (1 - \tilde x_{2i+1}) (1 - \tilde x_{2i+6})}{(1 - \tilde x_{2i+3}) (1 - \tilde x_{2i+4})}; \\
            y_{i+1,1} 
            &= - \frac{\tilde x_{2i} (1 - \tilde x_{2i-2} \tilde x_{2i+1}) (1 - \tilde x_{2i+3})}{(1 - \tilde x_{2i} \tilde x_{2i+3}) (1 - \tilde x_{2i+1})} \cdot \frac{\tilde x_{2i+7} (1 - \tilde x_{2i+6} \tilde x_{2i+9}) (1 - \tilde x_{2i+4})}{(1 - \tilde x_{2i+4} \tilde x_{2i+7}) (1 - \tilde x_{2i+6})} \\
            &= - \frac{\tilde x_{2i} \tilde x_{2i+7} (1 + y_{i-1,0}) (1 + y_{i+3,0}) (1 - \tilde x_{2i+3}) (1 - \tilde x_{2i+4})}{(1 + y_{i,0}) (1 + y_{i+2,0}) (1 - \tilde x_{2i+1}) (1 - \tilde x_{2i+6})} .
        \end{aligned}
    \end{equation}
    It follows that 
    \begin{equation*}
        \begin{aligned}
            y_{i+1,-1} \, y_{i+1,1}
            &= \frac{\tilde x_{2i} \tilde x_{2i+3} \tilde x_{2i+4} \tilde x_{2i+7} (1 + y_{i-1,0}) (1 + y_{i+3,0})}{(1 + y_{i,0}) (1 + y_{i+2,0})} \\
            &= \frac{y_{i,0} y_{i+2,0} (1 + y_{i-1,0}) (1 + y_{i+3,0})}{(1 + y_{i,0}) (1 + y_{i+2,0})} 
            = \frac{(1 + y_{i-1,0}) (1 + y_{i+3,0})}{(1 + y_{i,0}^{-1}) (1 + y_{i+2,0}^{-1})}.
        \end{aligned}
    \end{equation*}
    This concludes the proof. 
\end{proof}

%% file: figure/quadrangle_section.pdf_tex
\begingroup%
  \makeatletter%
  \providecommand\color[2][]{%
    \errmessage{(Inkscape) Color is used for the text in Inkscape, but the package 'color.sty' is not loaded}%
    \renewcommand\color[2][]{}%
  }%
  \providecommand\transparent[1]{%
    \errmessage{(Inkscape) Transparency is used (non-zero) for the text in Inkscape, but the package 'transparent.sty' is not loaded}%
    \renewcommand\transparent[1]{}%
  }%
  \providecommand\rotatebox[2]{#2}%
  \newcommand*\fsize{\dimexpr\f@size pt\relax}%
  \newcommand*\lineheight[1]{\fontsize{\fsize}{#1\fsize}\selectfont}%
  \ifx\svgwidth\undefined%
    \setlength{\unitlength}{519.04115692bp}%
    \ifx\svgscale\undefined%
      \relax%
    \else%
      \setlength{\unitlength}{\unitlength * \real{\svgscale}}%
    \fi%
  \else%
    \setlength{\unitlength}{\svgwidth}%
  \fi%
  \global\let\svgwidth\undefined%
  \global\let\svgscale\undefined%
  \makeatother%
  \begin{picture}(1,0.54476182)%
    \lineheight{1}%
    \setlength\tabcolsep{0pt}%
    \put(0,0){\includegraphics[width=\unitlength,page=1]{quadrangle_section.pdf}}%
    \put(0.36218091,0.50989139){\color[rgb]{0,0,0}\makebox(0,0)[lt]{\lineheight{1.25}\smash{\begin{tabular}[t]{l}$Q_1$\\\end{tabular}}}}%
    \put(0.51572352,0.28899364){\color[rgb]{0,0,0}\makebox(0,0)[lt]{\lineheight{1.25}\smash{\begin{tabular}[t]{l}$Q_4$\end{tabular}}}}%
    \put(0.41347851,0.22310698){\color[rgb]{0,0,0}\makebox(0,0)[lt]{\lineheight{1.25}\smash{\begin{tabular}[t]{l}$Q_3$\end{tabular}}}}%
    \put(0.19760268,0.3375617){\color[rgb]{0,0,0}\makebox(0,0)[lt]{\lineheight{1.25}\smash{\begin{tabular}[t]{l}$Q_2$\end{tabular}}}}%
    \put(0.42409088,0.02834464){\color[rgb]{0,0,0}\makebox(0,0)[lt]{\lineheight{1.25}\smash{\begin{tabular}[t]{l}$l_{13}$\end{tabular}}}}%
    \put(0.6470668,0.10460488){\color[rgb]{0,0,0}\makebox(0,0)[lt]{\lineheight{1.25}\smash{\begin{tabular}[t]{l}$l_{14}$\end{tabular}}}}%
    \put(0.55292888,0.09097019){\color[rgb]{0,0,0}\makebox(0,0)[lt]{\lineheight{1.25}\smash{\begin{tabular}[t]{l}$l_{23}$\end{tabular}}}}%
    \put(0.03976505,0.10442688){\color[rgb]{0,0,0}\makebox(0,0)[lt]{\lineheight{1.25}\smash{\begin{tabular}[t]{l}$l_{34}$\end{tabular}}}}%
    \put(0.73430312,0.24324211){\color[rgb]{0,0,0}\makebox(0,0)[lt]{\lineheight{1.25}\smash{\begin{tabular}[t]{l}$l_{24}$\end{tabular}}}}%
    \put(0.09293827,0.05482745){\color[rgb]{0,0,0}\makebox(0,0)[lt]{\lineheight{1.25}\smash{\begin{tabular}[t]{l}$l_{12}$\end{tabular}}}}%
    \put(0.09313512,0.16940163){\color[rgb]{0,0,0}\makebox(0,0)[lt]{\lineheight{1.25}\smash{\begin{tabular}[t]{l}$S_{12}$\end{tabular}}}}%
    \put(0.34947459,0.16936917){\color[rgb]{0,0,0}\makebox(0,0)[lt]{\lineheight{1.25}\smash{\begin{tabular}[t]{l}$S_{13}$\end{tabular}}}}%
    \put(0.24072241,0.16932775){\color[rgb]{0,0,0}\makebox(0,0)[lt]{\lineheight{1.25}\smash{\begin{tabular}[t]{l}$S_{34}$\end{tabular}}}}%
    \put(0.42801782,0.16849997){\color[rgb]{0,0,0}\makebox(0,0)[lt]{\lineheight{1.25}\smash{\begin{tabular}[t]{l}$S_{23}$\end{tabular}}}}%
    \put(0.59555338,0.17013225){\color[rgb]{0,0,0}\makebox(0,0)[lt]{\lineheight{1.25}\smash{\begin{tabular}[t]{l}$S_{14}$\end{tabular}}}}%
    \put(0.88420087,0.16860626){\color[rgb]{0,0,0}\makebox(0,0)[lt]{\lineheight{1.25}\smash{\begin{tabular}[t]{l}$S_{24}$\end{tabular}}}}%
    \put(0.0374407,0.20462497){\color[rgb]{0,0,0}\makebox(0,0)[lt]{\lineheight{1.25}\smash{\begin{tabular}[t]{l}$\omega$\end{tabular}}}}%
    \put(0,0){\includegraphics[width=\unitlength,page=2]{quadrangle_section.pdf}}%
    \put(0.37946522,0.31702253){\color[rgb]{0,0,0}\makebox(0,0)[lt]{\lineheight{1.25}\smash{\begin{tabular}[t]{l}$O$\end{tabular}}}}%
  \end{picture}%
\endgroup%

%% file: figure/31_formula_config.pdf_tex
\begingroup%
  \makeatletter%
  \providecommand\color[2][]{%
    \errmessage{(Inkscape) Color is used for the text in Inkscape, but the package 'color.sty' is not loaded}%
    \renewcommand\color[2][]{}%
  }%
  \providecommand\transparent[1]{%
    \errmessage{(Inkscape) Transparency is used (non-zero) for the text in Inkscape, but the package 'transparent.sty' is not loaded}%
    \renewcommand\transparent[1]{}%
  }%
  \providecommand\rotatebox[2]{#2}%
  \newcommand*\fsize{\dimexpr\f@size pt\relax}%
  \newcommand*\lineheight[1]{\fontsize{\fsize}{#1\fsize}\selectfont}%
  \ifx\svgwidth\undefined%
    \setlength{\unitlength}{286.04200709bp}%
    \ifx\svgscale\undefined%
      \relax%
    \else%
      \setlength{\unitlength}{\unitlength * \real{\svgscale}}%
    \fi%
  \else%
    \setlength{\unitlength}{\svgwidth}%
  \fi%
  \global\let\svgwidth\undefined%
  \global\let\svgscale\undefined%
  \makeatother%
  \begin{picture}(1,0.69140435)%
    \lineheight{1}%
    \setlength\tabcolsep{0pt}%
    \put(0,0){\includegraphics[width=\unitlength,page=1]{31_formula_config.pdf}}%
    \put(0.41318027,0.3309244){\color[rgb]{1,0,0}\makebox(0,0)[lt]{\lineheight{1.25}\smash{\begin{tabular}[t]{l}$A$\end{tabular}}}}%
    \put(0.28619563,0.49954168){\color[rgb]{0,0,0}\makebox(0,0)[lt]{\lineheight{1.25}\smash{\begin{tabular}[t]{l}$P_{-2}$\end{tabular}}}}%
    \put(0.23181578,0.33050013){\color[rgb]{0,0,0}\makebox(0,0)[lt]{\lineheight{1.25}\smash{\begin{tabular}[t]{l}$E$\end{tabular}}}}%
    \put(0.01944028,0.33325743){\color[rgb]{0,0,0}\makebox(0,0)[lt]{\lineheight{1.25}\smash{\begin{tabular}[t]{l}$G$\end{tabular}}}}%
    \put(0.54393601,0.32986644){\color[rgb]{1,0,0}\makebox(0,0)[lt]{\lineheight{1.25}\smash{\begin{tabular}[t]{l}$B$\end{tabular}}}}%
    \put(0.61164268,0.38222884){\color[rgb]{0,0,0}\makebox(0,0)[lt]{\lineheight{1.25}\smash{\begin{tabular}[t]{l}$C$\end{tabular}}}}%
    \put(0.68909358,0.33365255){\color[rgb]{0,0,0}\makebox(0,0)[lt]{\lineheight{1.25}\smash{\begin{tabular}[t]{l}$H$\end{tabular}}}}%
    \put(0.74457959,0.38407557){\color[rgb]{0,0,0}\makebox(0,0)[lt]{\lineheight{1.25}\smash{\begin{tabular}[t]{l}$F$\end{tabular}}}}%
    \put(0.88739071,0.3831994){\color[rgb]{0,0,0}\makebox(0,0)[lt]{\lineheight{1.25}\smash{\begin{tabular}[t]{l}$D$\end{tabular}}}}%
    \put(0.62164867,0.29324277){\color[rgb]{1,0,0}\makebox(0,0)[lt]{\lineheight{1.25}\smash{\begin{tabular}[t]{l}$P'_0$\end{tabular}}}}%
    \put(0.64046009,0.18623816){\color[rgb]{1,0,0}\makebox(0,0)[lt]{\lineheight{1.25}\smash{\begin{tabular}[t]{l}$P'_{-1}$\end{tabular}}}}%
    \put(0.59921358,0.13791277){\color[rgb]{1,0,0}\makebox(0,0)[lt]{\lineheight{1.25}\smash{\begin{tabular}[t]{l}$P'_{-2}$\end{tabular}}}}%
    \put(0.44054729,0.67285345){\color[rgb]{0,0,0}\makebox(0,0)[lt]{\lineheight{1.25}\smash{\begin{tabular}[t]{l}$O$\end{tabular}}}}%
    \put(0.86320359,0.18542464){\color[rgb]{0,0,0}\makebox(0,0)[lt]{\lineheight{1.25}\smash{\begin{tabular}[t]{l}$P_1$\end{tabular}}}}%
    \put(0.60337564,0.58867928){\color[rgb]{0,0,0}\makebox(0,0)[lt]{\lineheight{1.25}\smash{\begin{tabular}[t]{l}$P_{-1}$\end{tabular}}}}%
    \put(0.73903151,0.0255703){\color[rgb]{0,0,0}\makebox(0,0)[lt]{\lineheight{1.25}\smash{\begin{tabular}[t]{l}$P_2$\end{tabular}}}}%
    \put(0.57008195,0.01445781){\color[rgb]{0,0,0}\makebox(0,0)[lt]{\lineheight{1.25}\smash{\begin{tabular}[t]{l}$P_3$\end{tabular}}}}%
    \put(0.47143299,0.06271826){\color[rgb]{0,0,0}\makebox(0,0)[lt]{\lineheight{1.25}\smash{\begin{tabular}[t]{l}$P_4$\end{tabular}}}}%
    \put(0.23016612,0.17849235){\color[rgb]{0,0,0}\makebox(0,0)[lt]{\lineheight{1.25}\smash{\begin{tabular}[t]{l}$P_{-4}$\end{tabular}}}}%
    \put(0.93780111,0.3366587){\color[rgb]{0,0,0}\makebox(0,0)[lt]{\lineheight{1.25}\smash{\begin{tabular}[t]{l}$l_{-3,0}$\end{tabular}}}}%
  \end{picture}%
\endgroup%

%% file: figure/y-variables.pdf_tex
\begingroup%
  \makeatletter%
  \providecommand\color[2][]{%
    \errmessage{(Inkscape) Color is used for the text in Inkscape, but the package 'color.sty' is not loaded}%
    \renewcommand\color[2][]{}%
  }%
  \providecommand\transparent[1]{%
    \errmessage{(Inkscape) Transparency is used (non-zero) for the text in Inkscape, but the package 'transparent.sty' is not loaded}%
    \renewcommand\transparent[1]{}%
  }%
  \providecommand\rotatebox[2]{#2}%
  \newcommand*\fsize{\dimexpr\f@size pt\relax}%
  \newcommand*\lineheight[1]{\fontsize{\fsize}{#1\fsize}\selectfont}%
  \ifx\svgwidth\undefined%
    \setlength{\unitlength}{463.96673728bp}%
    \ifx\svgscale\undefined%
      \relax%
    \else%
      \setlength{\unitlength}{\unitlength * \real{\svgscale}}%
    \fi%
  \else%
    \setlength{\unitlength}{\svgwidth}%
  \fi%
  \global\let\svgwidth\undefined%
  \global\let\svgscale\undefined%
  \makeatother%
  \begin{picture}(1,0.42689568)%
    \lineheight{1}%
    \setlength\tabcolsep{0pt}%
    \put(0,0){\includegraphics[width=\unitlength,page=1]{y-variables.pdf}}%
    \put(0.68590236,0.02602234){\color[rgb]{0,0,0}\makebox(0,0)[lt]{\lineheight{1.25}\smash{\begin{tabular}[t]{l}$l_{0,1}$\\\end{tabular}}}}%
    \put(0.7920344,0.02746518){\color[rgb]{0,0,0}\makebox(0,0)[lt]{\lineheight{1.25}\smash{\begin{tabular}[t]{l}$l_{0,2}$\end{tabular}}}}%
    \put(0.89213843,0.02750644){\color[rgb]{0,0,0}\makebox(0,0)[lt]{\lineheight{1.25}\smash{\begin{tabular}[t]{l}$l_{0,3}$\end{tabular}}}}%
    \put(0.58632412,0.02538347){\color[rgb]{0,0,0}\makebox(0,0)[lt]{\lineheight{1.25}\smash{\begin{tabular}[t]{l}$l_{0,-1}$\end{tabular}}}}%
    \put(0.48050082,0.22190197){\color[rgb]{0,0,0}\makebox(0,0)[lt]{\lineheight{1.25}\smash{\begin{tabular}[t]{l}$l_{1,2}$\end{tabular}}}}%
    \put(0.58524201,0.3096397){\color[rgb]{0,0,0}\makebox(0,0)[lt]{\lineheight{1.25}\smash{\begin{tabular}[t]{l}$l_{1,0}$\end{tabular}}}}%
    \put(0.63965301,0.34704726){\color[rgb]{0,0,0}\makebox(0,0)[lt]{\lineheight{1.25}\smash{\begin{tabular}[t]{l}$l_{1,-1}$\end{tabular}}}}%
    \put(0.73487226,0.40825964){\color[rgb]{0,0,0}\makebox(0,0)[lt]{\lineheight{1.25}\smash{\begin{tabular}[t]{l}$l_{1,-2}$\end{tabular}}}}%
    \put(0.58110548,0.38888653){\color[rgb]{1,0,0}\makebox(0,0)[lt]{\lineheight{1.25}\smash{\begin{tabular}[t]{l}$l_{-1,2}$\end{tabular}}}}%
    \put(0.6669763,0.29582568){\color[rgb]{0,0,0}\makebox(0,0)[lt]{\lineheight{1.25}\smash{\begin{tabular}[t]{l}$\hat P_{i-1,0}$\end{tabular}}}}%
    \put(0.69890465,0.22261034){\color[rgb]{0,0,0}\makebox(0,0)[lt]{\lineheight{1.25}\smash{\begin{tabular}[t]{l}$\hat P_{i,1}$\end{tabular}}}}%
    \put(0.73353355,0.15658507){\color[rgb]{0,0,0}\makebox(0,0)[lt]{\lineheight{1.25}\smash{\begin{tabular}[t]{l}$\hat P_{i+1,1}$\end{tabular}}}}%
    \put(0.75216928,0.08407117){\color[rgb]{0,0,0}\makebox(0,0)[lt]{\lineheight{1.25}\smash{\begin{tabular}[t]{l}$\hat P_{i+2,0}$\end{tabular}}}}%
    \put(0.58627503,0.20729493){\color[rgb]{0,0,0.01960784}\makebox(0,0)[lt]{\lineheight{1.25}\smash{\begin{tabular}[t]{l}$\hat P_{i,0}$\end{tabular}}}}%
    \put(0.6188048,0.09997722){\color[rgb]{0,0,0.01960784}\makebox(0,0)[lt]{\lineheight{1.25}\smash{\begin{tabular}[t]{l}$\hat P_{i+1,0}$\end{tabular}}}}%
    \put(0.11123959,0.29744109){\color[rgb]{0,0,0}\makebox(0,0)[lt]{\lineheight{1.25}\smash{\begin{tabular}[t]{l}$\hat P_{i-1,0}$\end{tabular}}}}%
    \put(0.14316796,0.22422576){\color[rgb]{0,0,0}\makebox(0,0)[lt]{\lineheight{1.25}\smash{\begin{tabular}[t]{l}$\hat P_{i,1}$\end{tabular}}}}%
    \put(0.17779687,0.1582005){\color[rgb]{0,0,0}\makebox(0,0)[lt]{\lineheight{1.25}\smash{\begin{tabular}[t]{l}$\hat P_{i+1,1}$\end{tabular}}}}%
    \put(0.19643259,0.08568659){\color[rgb]{0,0,0}\makebox(0,0)[lt]{\lineheight{1.25}\smash{\begin{tabular}[t]{l}$\hat P_{i+2,0}$\end{tabular}}}}%
    \put(0.02008474,0.21467883){\color[rgb]{0,0,0.01960784}\makebox(0,0)[lt]{\lineheight{1.25}\smash{\begin{tabular}[t]{l}$\hat P_{i,0}$\end{tabular}}}}%
    \put(0.03204908,0.10359712){\color[rgb]{0,0,0.01960784}\makebox(0,0)[lt]{\lineheight{1.25}\smash{\begin{tabular}[t]{l}$\hat P_{i+1,0}$\end{tabular}}}}%
  \end{picture}%
\endgroup%

%% file: 6precompact.tex
\section{The Precompactness of \texorpdfstring{$T_3$}{T(3)} Orbits} \label{sec:precompact}

In this section, we establish four algebraic invariants of $T_3$. We then use them to prove Theorem \ref{thm:3-spiral precompact}. Having Theorem \ref{thm:tic-tac-toe and 3-spirals} in hand, we may fully work with $S_n(J,I)$ and $S_n(K,J)$. Our strategy is to use the algebraic invariants to show that the corner invariants are uniformly bounded.

\subsection{The Four Invariants}

\begin{proposition} \label{prop:31 invariants}
    Given $[P] \in \cP_n$ with corner invariants $x_j = x_j(P)$, consider the following four quantities $\cF_i = \cF_i(P)$:
    \begin{equation} \label{eqn:31 energies}
        \cF_1 = \prod_{i=0}^{n-1} \frac{x_{2i}}{x_{2i} - 1} ; \ \ 
        \cF_2 = \prod_{i=0}^{n-1} \frac{x_{2i+1}}{x_{2i+1} - 1} ; \ \ 
        \cF_3 = \prod_{i=0}^{n-1} \frac{x_{2i}}{x_{2i+1}}; \ \ 
        \cF_4 = \prod_{i=0}^{n-1} \frac{1 - x_{2i}}{1 - x_{2i+1}}.
    \end{equation}
    Then, $\cF_i$ is invariant under $T_3$ for $i = 1, 2, 3, 4$.
\end{proposition}

\begin{proof}
    We first show that $\cF_3$ is invariant under $T_3$. Let $\cF_3'$ denote the invariants obtained by plugging in $x'_i$ from Equation \eqref{eqn:T3 inverse corner}. Observe that 
    \begin{equation*} \label{eqn:F3 invariance}
        \begin{aligned}
            \cF_3' 
            &= \cF_3 \cdot \prod_{i=0}^{n-1} \frac{x_{2i-4} + x_{2i-1} - 1}{x_{2i+2} + x_{2i+5} - 1} \cdot \prod_{i=0}^{n-1} \frac{x_{2i+2} x_{2i+3} - (1 - x_{2i+5})(1 - x_{2i})}{x_{2i-2} x_{2i-1} - (1 - x_{2i+1})(1 - x_{2i-4})} \\
            &= \cF_3 \cdot \frac{\prod_{i={-3}}^{n-4} (x_{2i+2} + x_{2i+5} - 1)}{\prod_{i=0}^{n-1} (x_{2i+2} + x_{2i+5} - 1)} \cdot \frac{\prod_{i=2}^{n+1} (x_{2i-2} x_{2i-1} - (1 - x_{2i+1})(1 - x_{2i-4}))}{\prod_{i=0}^{n-1} (x_{2i-2} x_{2i-1} - (1 - x_{2i+1})(1 - x_{2i-4}))} 
            = \cF_3 .
        \end{aligned}
    \end{equation*}
    This shows $\cF'_3 = \cF_3$. 
    Next, we show that $\cF_1$ and $\cF_2$ are invariant. Using conjugated corner invariants, we see that $\cF_1 = \prod_{i=0}^{n-1} \tilde x_{2i}$ and $\cF_2 = \prod_{i=0}^{n-1} \tilde x_{2i+1}$. We let $\cF_1' = \prod_{i=0}^{n-1} \tilde x'_{2i}$ be the first invariant of $T_3(P)$. Equation \eqref{eqn:T3 formula conj corner inv} gives us
    \begin{equation*}
        \cF'_1 
        = \cF_1 \cdot 
        \prod_{i=0}^{n-1} \frac{(1 - \tilde x_{2i-1} \tilde x_{2i-4})(1 - \tilde x_{2i+1})}{(1 - \tilde x_{2i+1} \tilde x_{2i-2})(1 - \tilde x_{2i-1})} = \cF_1 ,
    \end{equation*}
    where the last equality follows from cyclically permuting the numerator. 
    This shows $\cF'_1 = \cF_1$. The proof for $\cF_2$ goes through the same computation, so we omit it. 

    Finally, observe that $\cF_4 = \frac{\cF_2 \cF_3}{\cF_1}$, so by invariance of $\cF_1, \cF_2, \cF_3$, we know that $\cF_4$ must also be invariant. This concludes the proof. 
\end{proof}

\begin{remark}
    As shown in the proof of Proposition \ref{prop:31 invariants},  $\cF_1$ and $\cF_2$ correspond to the product of conjugated corner invariants. $\cF_3$ is the ratio of the two Casimirs $\frac{O_n}{E_n}$ of the $T_2$ invariant Poisson structure on $\cP_n$. For discussions on $\cF_3$ and the Casimirs, see \cite[$\S$2.3]{schwartz2024pentagramrigiditycentrallysymmetric}. Also, since $\cF_1\cF_4 = \cF_2 \cF_3$, the four $T_3$ invariants are not algebraically independent. 
\end{remark}

Below is a direct consequence of the invariance of the $\cF_i$'s. Since the $\cF_i$'s are preserved by the forward action, it must also be preserved by the backward action. 

\begin{corollary} \label{cor:31 inverse invariants}
    The four invariants $\cF_1, \cF_2, \cF_3, \cF_4$ are also invariant under $T_3^{-1}$. 
\end{corollary}

\subsection{Proof of Theorem \ref{thm:3-spiral precompact}} \label{subsec:precompactness of T3}

Recall that a subset $A$ of a topological space $X$ is precompact if the closure of $A$ is compact. To show that the $T_3$-orbit is precompact, it suffices to show that the corner invariants of the orbit are uniformly bounded away from the singularities $0, 1, \infty$. 

In this section, we let $[n] := \{1, \ldots, n\}$. Given $[P] \in \cP_n$, for all $j,m \in \ZZ$, let $x_{j,m} = x_j(T_3^m(P))$ whenever $T_3^m(P)$ exists. Let $\cF_{i,m} = \cF_i(T_3^m(P))$ for $i = 1, 2, 3, 4$. By Proposition \ref{prop:31 invariants}, $\cF_{i,m}$ is independent of $m$. All sequences are indexed by $\ZZ_{\geq 0}$ unless specified otherwise. Finally, when we say ``$\{a_m\}$ converges/diverges on a subsequence, and $\{b_m\}$ converges/diverges on the same subsequence,'' we mean that a subsequence of $\{b_m\}$ with the same choice of indices as the subsequence of $\{a_m\}$ converges/diverges. 

\begin{lemma} \label{lem:compact J in (J, I)}
    Given $[P] \in S_n(J, I)$, there exist $a, b \in J$ such that $x_{2i,m} \in [a, b]$ for all $i \in [n]$ and $m \in \ZZ_{\geq0}$. 
\end{lemma}

\begin{proof}
    We first claim that for each $i$, the sequence $\{x_{2i,m}\}$ is bounded above uniformly by some $b_i \in J$. If not, then $x_{2i,m} \rightarrow 1$ on a subsequence, which implies $1 - x_{2i,m} \rightarrow 0$ on the same subsequence. Since $[T_3^m(P)] \in S_n(J,I)$ for all $m \in \ZZ_{\geq0}$, we must have $1 - x_{2j,m} \in (0,1)$ and $(1 - x_{2j+1,m})^{-1} \in (0,1)$ for all $j\in [n]$. This implies $\cF_{4,m} \rightarrow 0$ on the same subsequence, but that contradicts invariance of $\cF_{4,m}$. Therefore, $\{x_{2i,m}\}$ is bounded above by $b_i = \sup_m \{x_{2i,m}\} \in J$. Taking $b = \max_{i \in [n]} b_i$ satisfies the condition in the lemma.

    Next, we show $\{x_{2i,m}\}$ is bounded below uniformly by some $a_i > 0$. If not, then $x_{2i,m} \rightarrow 0$ on a subsequence, so $x_{2i,m} \cdot (x_{2i,m} - 1)^{-1} \rightarrow 0$ on the same subsequence. From the argument above, $x_{2j,m} \leq b$ for all $m \in \ZZ_{\geq0}$ and $j\in [n]$, which gives us $|x_{2j,m} \cdot (x_{2j,m} - 1)^{-1}| \leq |\frac{b}{b-1}|$, so the sequences are uniformly bounded for all $j \neq i$. This together with $|x_{2i,m} \cdot (x_{2i,m} - 1)^{-1}| \rightarrow 0$ on a subsequence implies $|\cF_{1,m}| \rightarrow 0$ on the same subsequence, but that contradicts invariance of $\cF_{1,m}$. Therefore, $\{x_{2i,m}\}$ is bounded below by $a_i = \inf_m \{x_{2i,m}\} \in J$.
    Taking $a = \min_{i \in [n]} a_i$ completes the proof. 
\end{proof}

\begin{lemma} \label{lem:compact I in (J, I)}
    With the same notation as in Lemma \ref{lem:compact J in (K, J)}, there exist $c, d \in I$ such that $x_{2i+1,m} \in [c, d]$ for all $i \in [n]$ and $m \in \ZZ_{\geq0}$. 
\end{lemma}

\begin{proof}
    We first claim that for each $i$,the sequence $\{x_{2i+1,m}\}$ is bounded above uniformly by some $d_i \in I$. If not, then, $x_{2i+1,m} \cdot (x_{2i+1,m} - 1)^{-1} \rightarrow 0$ on a subsequence. Since $x_{2j+1,m} \cdot (x_{2j+1,m} - 1)^{-1} \in (0,1)$ for all $j \in [n]$, we must have $\cF_{2,m} \rightarrow 0$ on the same subsequence, but that contradicts invariance of $\cF_{2,m}$. 

    Next, we show that $\{x_{2i+1,m}\}$ is bounded below uniformly by some $c_i \in I$. If not, then a subsequence of $\{x_{2i+1,m}\}$ must diverge, so the same subsequence of $\{1 - x_{2i+1,m}\}$ also diverges. Lemma \ref{lem:compact J in (J, I)} and $x_{2i+1,m} \leq d_i < 0$ together implies $\cF_{4,m}$ diverges on the same subsequence, but that contradicts invariance of $\cF_{4,m}$. 
    Finally, taking $c = \min_{i \in [n]} c_i$ and $d = \max_{i \in [n]} d_i$ completes the proof. 
\end{proof}

The proofs of the following two lemmas are analogous to Lemma \ref{lem:compact J in (J, I)} and \ref{lem:compact I in (J, I)}. We will omit the details and point out which invariants to use in each step.

\begin{lemma} \label{lem:compact J in (K, J)}
    Given $[P] \in S_n(K, J)$, there exist $a, b \in J$ such that $x_{2i+1,m} \in [a, b]$ for all $i \in [n]$ and $m \in \ZZ_{\geq0}$. 
\end{lemma}

\begin{proof}
    For each $i$, the sequence $\{x_{2i+1,m}\}$ is bounded below by some $a_i \in J$, for otherwise $\cF_{3,m}$ diverges on a subsequence. Next, since $\{ |\cF_{3,m}| \}$ is bounded below by $\prod_{j=0}^{n-1} a_j > 0$,  $\{x_{2i+1,m}\}$ is bounded above by some $b_i \in J$. Taking $a = \min_{i \in [n]} a_i$ and $b = \max_{i \in [n]} b_i$ completes the proof. 
\end{proof}

\begin{lemma} \label{lem:compact K in (K, J)}
    With the same notation as in Lemma \ref{lem:compact J in (K, J)}, there exist $c, d \in K$ such that $x_{2i,m} \in [c, d]$ for all $i \in [n]$ and $m \in \ZZ_{\geq0}$. 
\end{lemma}

\begin{proof}
    For each $i$, the sequence $\{x_{2i,m}\}$ must be bounded below by some $c_i \in K$, for otherwise $\cF_{1,m} \rightarrow \infty$ on a subsequence. It's also bounded above by some $d_i$. To see this, Lemma \ref{lem:compact J in (K, J)} implies all corner invariants are bounded away from 0, so if $\{x_{2i,m}\}$ is not bounded above, then $\cF_{3,m}$ diverges on a subsequence. Taking $c = \min_i c_i$ and $d = \max_i d_i$ completes the proof. 
\end{proof}

\begin{proof}[Proof of Theorem \ref{thm:3-spiral precompact}]
    We will show that the forward $T_3$ orbit of $[P] \in \cS_{3,n}^\alpha = S_n(J,I)$ has uniformly bounded corner invariants. By Proposition \ref{prop:S3 alpha equals in S(J,I)}, $[P] \in S_n(J,I)$. 
    Let $[a,b] \subset J$, $[c, d] \subset I$ be compact intervals derived from Lemma \ref{lem:compact J in (J, I)} and \ref{lem:compact I in (J, I)}. Then, the sequence $\{(x_{0,m}, \ldots, x_{2n-1,m})\}$ is contained in $\prod_{i=0}^{n-1} [a,b] \times [c,d]$, so it is uniformly bounded. To show precompactness of the backward $T_3$ orbit of $\cS_{3,n}^\alpha$, one can adapt the proofs of Lemma \ref{lem:compact J in (J, I)} and \ref{lem:compact I in (J, I)} with very few changes. We omit the details. The case $\cS_{3,n}^\beta$ follows from Lemma \ref{lem:compact J in (K, J)} and \ref{lem:compact K in (K, J)} by essentially the same argument, which we again omit.
\end{proof}

%% file: 7twospiral.tex
\section{Type-\texorpdfstring{$\beta$}{beta} 2-Spirals and Precompact \texorpdfstring{$T_2$}{T2} Orbits} \label{sec:two spiral}

\subsection{The Corner Invariants of Type-\texorpdfstring{$\beta$}{beta} 2-Spirals} \label{subsec:type beta 2 spirals corner inv}

We finish this paper by discussing the type-$\beta$ 2-spirals. Proposition \ref{prop:k spirals are not closed} implies $\cS_{2,n}^{\beta}$ is disjoint from the moduli space of closed convex polygons, so $\cS_{2,n}^{\beta}$ is a new invariant geometric construction under the pentagram map by Theorem \ref{thm:spiral polygon invariance}. In this section, we analyze the corner invariants of $\cS_{2,n}^\beta$ and show that just like the type-$\alpha$ and type-$\beta$ 3-spirals, it is cut out by linear boundaries.

\begin{proposition} \label{prop:type beta 2 spiral corner inv}
    For all $n \geq 2$, given any $[P] \in \cS_{2,n}^\beta$ with corner invariants $x_j = x_j(P)$, we have $x_{2i} > 0$ and $x_{2i+1} < 0$ for all $i \in [n]$. 
\end{proposition}

\begin{proof}
    Let $P$ be an $(i-2)$-representative of $[P]$. Normalize by $\Aff_2^+(\RR)$ so that $P_{i-1} = (-1,0)$, $P_i = (0,0)$, $P_{i+1} = (0,1)$ on the affine patch, which is possible because $(P_{i-1}, P_{i}, P_{i+1})$ is positive. Let $s_{a,b} \in \RR \cup \{\infty\}$ denote the slope of $P_{i+a} P_{i+b}$.
    Positivity of $(P_{i-2}, P_{i-1}, P_i)$ and $P_{i+1} \in \Int(P_{i-2}, P_{i-1}, P_i)$ implies $s_{-1,-2} > 1$ and $s_{1,-2} > 1$. Similarly, since $P_{i+2} \in \Int(P_{i-1}, P_{i}, P_{i+1})$, we have $s_{-1,2} \in (0,1)$ and $s_{1,2} > 1$. It follows that 
    \begin{equation} \label{eqn:type beta 2 spiral corner}
        \begin{aligned}
            x_{2i} 
            &= \frac{(s_{1,-2} - s_{1,-1}) (s_{1,0} - s_{1,2})}{(s_{1,-2} - s_{1,0}) (s_{1,-1} - s_{1,2})} 
            = - \frac{s_{1,-2} - 1}{1 - s_{1,2}} > 0 ; \\
            x_{2i+1}
            &= \frac{(s_{-1,2} - s_{-1,1}) (s_{-1,0} - s_{-1,-2})}{(s_{-1,2} - s_{-1,0}) (s_{-1,1} - s_{-1,-2})} 
            = \frac{- s_{-1,-2} (s_{-1,2} - 1)}{s_{-1,2} (1 - s_{-1,-2})}
            < 0.
        \end{aligned}
    \end{equation}
    This concludes the proof. 
\end{proof}

\begin{figure}[ht]
    \centering
    \footnotesize
    
    \def\svgwidth{0.35\columnwidth}
    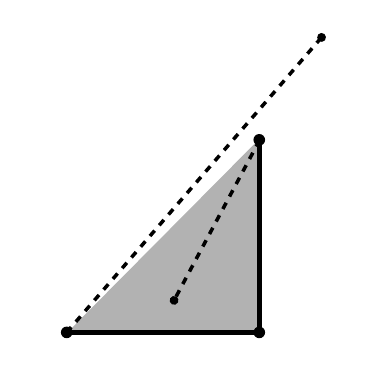

    \caption{Configuration of Proposition \ref{prop:type beta 2 spiral corner inv} and \ref{prop:type beta 2 spiral corner inv 2}.}
    \label{fig:2 spirals}
\end{figure}

\begin{proposition} \label{prop:type beta 2 spiral corner inv 2}
    For all $n \geq 2$, if $[P] \in \cP_n$ has corner invariants $x_j = x_j(P)$ such that $x_{2i} > 0$ and $x_{2i+1} < 0$ for all $i \in [n]$, then $[P] \in \cS_{2,n}^\beta$.
\end{proposition}

\begin{proof}
    Fix $N \in \ZZ$.
    Let $P$ be a representative of $[P]$ such that $P_{N-2} = (\frac{1}{3}, \frac{3}{2})$, $P_{N-1} = (-1,0)$, $P_N = (0,0)$, and $P_{N+1} = (0,1)$. 
    We say $P$ satisfies condition $(**)_i$ if $(P_{i-1}, P_{i}, P_{i+1})$ is positive and $P_{i+2} \in \Int(P_{i-1}, P_{i}, P_{i+1})$. Then, $P$ is a type-$\beta$ $N$-representative of 2-spirals iff $P$ satisfies $(**)_i$ for all $i > N$. Notice that $P$ satisfies $(**)_N$, so by an induction argument, it suffices to show that for $i \geq N$, if $P$ satisfies $(**)_i$, then $P$ satisfies $(**)_{i+1}$. 
    
    If $P$ satisfies $(**)_i$, then $(P_{i-1}, P_i, P_{i+1})$ is positive. Normalize by $\Aff_2^+(\RR)$ so that $P_{i-1} = (-1,0)$, $P_i = (0,0)$, and $P_{i} = (0,1)$. We will use the same notation $l_{a,b}$ and $s_{a,b}$ as Proposition \ref{prop:type beta 2 spiral corner inv}. Since $P_{i+1} \in \Int(P_{i-2}, P_{i-1}, P_{i})$, we have $s_{-1,-2} > 1$ and $s_{1,-2} > 1$. Then, since $x_{2i} > 0$ and $x_{2i+1} < 0$, Equation \eqref{eqn:type beta 2 spiral corner} gives us
    \begin{equation*}
        \frac{s_{1,-2} - 1}{s_{1,2} - 1} > 0 \ \text{ and } \ \frac{s_{-1,-2} (s_{-1,2} - 1)}{s_{-1,2}(s_{-1,-2} - 1)} < 0.
    \end{equation*}
    It follows that $s_{1,2} > 1$ and $1 - \frac{1}{s_{-1,2}} < 0$. The latter inequality implies $\frac{1}{s_{-1,2}} > 1$, so in particular $s_{-1,2} > 0$ and hence $s_{-1,2} \in (0,1)$. The two conditions $s_{1,2} > 1$ and $s_{-1,2} \in (0,1)$ implies $P_{i+2} \in \Int(P_{i-1}, P_{i}, P_{i+1})$ and $(P_{i}, P_{i+1}, P_{i+2})$ positive, so $P$ satisfies $(**)_{i+1}$ as desired. We conclude that $[P] \in \cS_{2,n}^\beta$.
\end{proof}

\subsection{The Precompactness of \texorpdfstring{$T_2$}{T(2)} Orbits} \label{subsec:precompactness of T2 orbits}

We adapt the argument for Theorem \ref{thm:3-spiral precompact} to give a quick proof of Theorem \ref{thm:2-spiral precompact} using the Casimir functions of the $T_2$-invariant Poisson structure on $\cP_n$ that were developed in \cite[Theorem 1.2]{schwartz2008discretemonodromypentagramsmethod}. One can find the proof of the following lemma in \cite[$\S$2.2]{schwartz2008discretemonodromypentagramsmethod}.

\begin{lemma} \label{lem:T2 invariants}
    For the map $T_2$ acting on a twisted $n$-gon $P$ with corner invariants $x_j = x_j(P)$, one has the following four invariant quantities.
    \begin{equation*}
        \begin{aligned}
            O_1(P) &= \sum_{i=0}^{n-1} (- x_{2i+1} + x_{2i-1} x_{2i} x_{2i+1}) ; &
            O_n(P) &= \prod_{i=0}^{n-1} x_{2i+1} ; \\
            E_1(P) &= \sum_{i=0}^{n-1} (- x_{2i} + x_{2i-2} x_{2i-1} x_{2i}) ; &
            E_n(P) &= \prod_{i=1}^n x_{2i} . 
        \end{aligned}
    \end{equation*}
\end{lemma}

We continue to use the notation from $\S$\ref{subsec:precompactness of T3}. In addition, we write $O_{1,m} = O_1(T_2^m(P))$. We define $O_{n,m}$, $E_{1,m}$, and $E_{n,m}$ analogously. By Lemma \ref{lem:T2 invariants}, the values of these four quantities are independent of the choice of $m$.

\begin{lemma} \label{lem:bounding even corner inv for type beta 2 spirals}
    For all $n \geq 2$, given $[P] \in \cS_{2,n}^\beta$, there exists $a,b > 0$ such that $x_{2i,m} \in [a,b]$ for all $i \in [n]$ and $m \in \ZZ_{\geq0}$.
\end{lemma}

\begin{proof}
    Fix $i \in [n]$. We first show that $x_{2i,m}$ is uniformly bounded above by some $b > 0$. Since $T_2^m(P) \in \cS_{2,n}^\beta$ for all $m \in \ZZ_{\geq0}$, we must have $E_{1,m} < -x_{2i,m} < 0$. Then, if $x_{2i,m} \rightarrow \infty$ on a subsequence, $E_{1,m}$ also diverges on the same subsequence, but that contradicts invariance of $E_{1,m}$. This implies $x_{2i,m} < b_i$ for some $b_i > 0$. Taking $b = \max_{i \in [n]} b_i$ satisfies the condition in the lemma. 

    Next, we show that $x_{2i,m}$ is uniformly bounded below by some $a > 0$. We first notice that $E_{n,m} < b_i^{n}$. This implies if $x_{2i,m} \rightarrow 0$ on a subsequence, then $E_{n,m} \rightarrow 0$ on the same subsequence, but that contradicts invariance of $E_{n,m}$. Therefore, $x_{2i,m} > a_i$ for some $a_i > 0$. Taking $a = \min_{i \in [n]} a_i$ completes the proof. 
\end{proof}

\begin{lemma} \label{lem:bounding odd corner inv for type beta 2 spirals}
    For all $n \geq 2$, given $[P] \in \cS_{2,n}^\beta$, there exists $c,d < 0$ such that $x_{2i+1,m} \in [c,d]$ for all $i \in [n]$ and $m \in \ZZ_{\geq0}$.
\end{lemma}

\begin{proof}
    The argument is analogous to the proof of Lemma \ref{lem:bounding even corner inv for type beta 2 spirals}. Fix $i \in [n]$. To find $c_i$ that bounds $\{x_{2i+1,m}\}$ uniformly from below, we use the fact that $O_{1,m} > -x_{2i+1} > 0$. We then set $c = \min_{i \in [n]} c_i$. To find $d_i$ that bounds $\{x_{2i+1,m}\}$ uniformly from above, we use the fact that $|O_n(P)| < |c^n|$. We then set $d = \max_{i \in [n]} d_i$ to complete the proof. 
\end{proof}

Lemma \ref{lem:bounding even corner inv for type beta 2 spirals} and \ref{lem:bounding odd corner inv for type beta 2 spirals} together implies that the forward $T_2$-orbit of any $[P] \in \cS_{2,n}^\beta$ is precompact in $\cP_n$. One can use the same argument to show that the backward $T_2$-orbit is also precompact. We have thus completed the proof of Theorem \ref{thm:2-spiral precompact}.

%% file: figure/2-spiral.pdf_tex
\begingroup%
  \makeatletter%
  \providecommand\color[2][]{%
    \errmessage{(Inkscape) Color is used for the text in Inkscape, but the package 'color.sty' is not loaded}%
    \renewcommand\color[2][]{}%
  }%
  \providecommand\transparent[1]{%
    \errmessage{(Inkscape) Transparency is used (non-zero) for the text in Inkscape, but the package 'transparent.sty' is not loaded}%
    \renewcommand\transparent[1]{}%
  }%
  \providecommand\rotatebox[2]{#2}%
  \newcommand*\fsize{\dimexpr\f@size pt\relax}%
  \newcommand*\lineheight[1]{\fontsize{\fsize}{#1\fsize}\selectfont}%
  \ifx\svgwidth\undefined%
    \setlength{\unitlength}{186.10863633bp}%
    \ifx\svgscale\undefined%
      \relax%
    \else%
      \setlength{\unitlength}{\unitlength * \real{\svgscale}}%
    \fi%
  \else%
    \setlength{\unitlength}{\svgwidth}%
  \fi%
  \global\let\svgwidth\undefined%
  \global\let\svgscale\undefined%
  \makeatother%
  \begin{picture}(1,0.98415635)%
    \lineheight{1}%
    \setlength\tabcolsep{0pt}%
    \put(0,0){\includegraphics[width=\unitlength,page=1]{2-spiral.pdf}}%
    \put(0.79082774,0.92775804){\color[rgb]{0,0,0}\makebox(0,0)[lt]{\lineheight{1.25}\smash{\begin{tabular}[t]{l}$P_{i-2}$\end{tabular}}}}%
    \put(0.04360665,0.04178767){\color[rgb]{0,0,0}\makebox(0,0)[lt]{\lineheight{1.25}\smash{\begin{tabular}[t]{l}$P_{i-1}$\end{tabular}}}}%
    \put(0.72745879,0.0440676){\color[rgb]{0,0,0}\makebox(0,0)[lt]{\lineheight{1.25}\smash{\begin{tabular}[t]{l}$P_i$\end{tabular}}}}%
    \put(0.72183984,0.60733854){\color[rgb]{0,0,0}\makebox(0,0)[lt]{\lineheight{1.25}\smash{\begin{tabular}[t]{l}$P_{i+1}$\end{tabular}}}}%
    \put(0.47707581,0.18736103){\color[rgb]{0,0,0}\makebox(0,0)[lt]{\lineheight{1.25}\smash{\begin{tabular}[t]{l}$P_{i+2}$\end{tabular}}}}%
  \end{picture}%
\endgroup%

%% file: 8appendix.tex
\section{Appendix} \label{sec:appendix}

\subsection{Conjectures for Invariants}

Given $[P] \in \cP_{k,n}$, we may consider the following quantity:
\begin{equation} \label{eqn:y-variables for Tk}
    y_{i}^{(k)}(P) = -\chi(P_i, P_{i} P_{i+k} \cap P_{i-1} P_{i+k-1}, P_{i} P_{i+k} \cap P_{i+1} P_{i+k+1}, P_{i+k}) .
\end{equation}
When $T_k^j(P)$ is well-defined, we write $y_{i,j}^{(k)} = y_i^{(k)}(T_k^j(P))$, or simply $y_{i,j}$ if the value of $k$ is clear from the context. Let $Y_{j}^{(k)}$ (or simply $Y_j$) denote the product $\prod_{i=0}^{n-1} y_{i,j}^{(k)}$. 

\begin{proposition} \label{prop:alg invariants for Tk}
    For all $k,n \geq 2$, given $[P] \in \cS_{k,n}^\alpha$, there exists $C \in \RR$, $C \neq 0$, such that
    \begin{equation} \label{eqn:Tk casimirs}
        Y_{j+1}^{(k)} \left( Y_j^{(k)} \right)^{-1} = C 
    \end{equation}
    for all $j \in \ZZ$. The same holds for $[P] \in \cS_{k,n}^\beta$.
\end{proposition}

\begin{proof}
    The grid $(\hat P_{i,j})_{(i,j) \in \ZZ^2}$ where $\hat P_{i,j}$ is the $i$-th vertex of $T_k^j(P)$ is a $Y$-mesh of type $S = \{(0,0), (k,0), (-1,1), (0,1)\}$ with $y_r = y_r(\hat P)$ for $r \in \ZZ^2$. The proof is essentially the same as the one for Proposition \ref{prop:Pij is Y-mesh}, so we will omit it. Then, from \cite[Theorem 1.6]{GP2016ymeshes}, we have 
    \begin{equation} \label{eqn:y-variable transform for Tk}
        y_{i+k,j} y_{i-1,j+2} 
        = \frac{(1 + y_{i-1,j+1})(1 + y_{i+k,j+1})}{(1 + y_{i,j+1}^{-1})(1 + y_{i+k-1,j+1}^{-1})} .
    \end{equation}
    It follows that 
    \begin{equation}
        \begin{aligned}
            Y_j Y_{j+2}
            &= \prod_{i=0}^{n-1} (y_{i+k,j} y_{i-1,j+2}) 
            = \prod_{i=0}^{n-1} \frac{(1 + y_{i-1,j+1})(1 + y_{i+k,j+1})}{(1 + y_{i,j+1}^{-1})(1 + y_{i+k-1,j+1}^{-1})} \\
            &= \left( \prod_{i=0}^{n-1} (y_{i,j+1} \, y_{i+k-1,j+1})  \right) \, 
            \left( \prod_{i=0}^{n-1} \frac{(1 + y_{i-1,j+1})(1 + y_{i+k,j+1})}{(1 + y_{i,j+1})(1 + y_{i+k-1,j+1})} \right) 
            = Y_{j+1}^2.
        \end{aligned}
    \end{equation}
    This implies $Y_{j+2} / Y_{j+1} = Y_{j+1} / Y_j$ for all $j \in \ZZ$. Taking $C = Y_1 / Y_0$ completes the proof. 
\end{proof}

\begin{remark}
    Combining the results of \cite{GP2016ymeshes} and \cite{gekhtman2012higherpentagrammapsweighted}, we see that
    Proposition \ref{prop:alg invariants for Tk} is equivalent to \cite[Theorem 2.1]{gekhtman2012higherpentagrammapsweighted}. Specifically, the quantity in Equation \eqref{eqn:Tk casimirs} is shown to be a Casimir function with respect to a Poisson structure that is invariant under the $y$-variable transformation of a quiver $Q_k$, which we will define below.
    
    Consider the infinite directed graph $Q_k$ whose vertices are indexed by $\ZZ \times \{0,1\}$, with directed edges $(i,0) \rightarrow (i-1,1)$, $(i,0) \rightarrow (i-k,1)$, $(i,1) \rightarrow (i,0)$, and $(i-k-1,1) \rightarrow (i,0)$ for all $i \in \ZZ$. See Figure \ref{fig:quiver} for a visual representation of this quiver. We refer the readers to \cite[$\S$9]{GP2016ymeshes} for the construction of this quiver and the proof that the $y$-variable transformations satisfy \eqref{eqn:y-variable transform for Tk}.

    \begin{figure}[ht]
        \centering
        \footnotesize
        
    \def\svgwidth{0.6\columnwidth}
    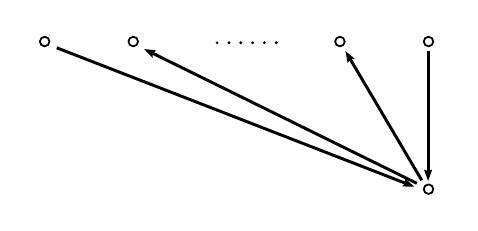

        \caption{The quiver $Q_k$. Only edges into and out of the vertex $(i,0)$ are shown.}
        \label{fig:quiver}
    \end{figure}
    
    For all $n \geq 2$, the $y$-variables corresponding to $[P] \in \cP_{k,n}$ are periodic modulo $n$. We may then identify vertices of $Q_k$ via $(i,j) \sim (i+n,j)$, and similarly identify the corresponding edges. The resulting directed graph $Q_{k,n}$ is isomorphic to the quiver $\cQ_{k,n}$ from \cite{gekhtman2012higherpentagrammapsweighted} by applying a translation to the first entry of the vertices $(i,1)$. Moreover, the $y$-variables $(y_{i,0})_{i\in[n]}$ of the quiver in \cite{GP2016ymeshes} transforms in the same way as the $p$-variables $(p_i)_{i\in[n]}$ of $\cQ_{k,n}$ under the map $\overline{T_k}$ (see \cite[$\S$2]{gekhtman2012higherpentagrammapsweighted}), and the $q$-variables $(q_i)_{i\in[n]}$ of $\cQ_{k,n}$ correspond to the multiplicative inverse of $y_{i,-1}$. As a result, $Y_{0}^{(k)} / Y_{-1}^{(k)} = \prod_{i=1}^n p_iq_i$, which by \cite[Theorem 2.1]{gekhtman2012higherpentagrammapsweighted} is invariant under $\overline{T_k}$ and forms a Casimir function with respect to a Poisson structure that is invariant under $\overline{T_k}$. 

    Both \cite{GP2016ymeshes} and \cite{gekhtman2012higherpentagrammapsweighted} demonstrate that the quiver $Q_{k,n}$ is a bipartite graph that can be embedded into a torus. For further details, see \cite[$\S$9]{GP2016ymeshes} and \cite[$\S$3]{gekhtman2012higherpentagrammapsweighted}. This connection links $Q_{k,n}$ to the \textit{Goncharov-Kenyon Dimer Integrable Systems} in \cite{GK13dimerIntegrableSystems}, where a more general definition of Casimir functions is provided.
\end{remark}

\begin{conjecture} \label{conj:alg invariants for Tk}
    The constant $C$ in Proposition \ref{prop:alg invariants for Tk} equals 1 for all $k \geq 2$.
\end{conjecture}

We prove Conjecture \ref{conj:alg invariants for Tk} for $k = 2$ and $k = 3$. Let $x_j = x_j(P)$ be the corner invariants of $[P]$. The case $k = 2$ follows from $y_{i}^{(2)}(P) = -x_{2i+1} x_{2i+2}$ (see Equation \eqref{eqn:y-variable of T2}), so 
\begin{equation*}
    \prod_{i=1}^n y_{i}^{(2)}(P) 
    = (-1)^n \prod_{i=1}^n x_{2i+1} x_{2i+2} 
    = (-1)^n \, O_n(P) \, E_n(P) ,
\end{equation*}
which is $T_2$-invariant by Lemma \ref{lem:T2 invariants}. 

For the case $k = 3$, Equation \eqref{eqn:y-variable and corner inv} implies 
\begin{equation*}
    \prod_{i=1}^n y_{i}^{(3)}(P)
    = (-1)^n \prod_{i=1}^n \frac{x_{2i} x_{2i+3}}{(x_{2i} - 1)(x_{2i+3} - 1)}
    = (-1)^n \, \cF_1(P) \, \cF_2(P) ,
\end{equation*}
which is $T_3$-invariant by Proposition \ref{prop:31 invariants}.

%% file: figure/quiver.pdf_tex
\begingroup%
  \makeatletter%
  \providecommand\color[2][]{%
    \errmessage{(Inkscape) Color is used for the text in Inkscape, but the package 'color.sty' is not loaded}%
    \renewcommand\color[2][]{}%
  }%
  \providecommand\transparent[1]{%
    \errmessage{(Inkscape) Transparency is used (non-zero) for the text in Inkscape, but the package 'transparent.sty' is not loaded}%
    \renewcommand\transparent[1]{}%
  }%
  \providecommand\rotatebox[2]{#2}%
  \newcommand*\fsize{\dimexpr\f@size pt\relax}%
  \newcommand*\lineheight[1]{\fontsize{\fsize}{#1\fsize}\selectfont}%
  \ifx\svgwidth\undefined%
    \setlength{\unitlength}{229.26890216bp}%
    \ifx\svgscale\undefined%
      \relax%
    \else%
      \setlength{\unitlength}{\unitlength * \real{\svgscale}}%
    \fi%
  \else%
    \setlength{\unitlength}{\svgwidth}%
  \fi%
  \global\let\svgwidth\undefined%
  \global\let\svgscale\undefined%
  \makeatother%
  \begin{picture}(1,0.48294594)%
    \lineheight{1}%
    \setlength\tabcolsep{0pt}%
    \put(0,0){\includegraphics[width=\unitlength,page=1]{quiver.pdf}}%
    \put(-0.00108829,0.42233226){\color[rgb]{0,0,0}\makebox(0,0)[lt]{\lineheight{1.25}\smash{\begin{tabular}[t]{l}$(i-k-1,1)$\end{tabular}}}}%
    \put(0.22146096,0.42233226){\color[rgb]{0,0,0}\makebox(0,0)[lt]{\lineheight{1.25}\smash{\begin{tabular}[t]{l}$(i-k,1)$\end{tabular}}}}%
    \put(0.64805196,0.42233226){\color[rgb]{0,0,0}\makebox(0,0)[lt]{\lineheight{1.25}\smash{\begin{tabular}[t]{l}$(i-1,1)$\end{tabular}}}}%
    \put(0.85823737,0.42233226){\color[rgb]{0,0,0}\makebox(0,0)[lt]{\lineheight{1.25}\smash{\begin{tabular}[t]{l}$(i,1)$\end{tabular}}}}%
    \put(0.86095592,0.03584679){\color[rgb]{0,0,0}\makebox(0,0)[lt]{\lineheight{1.25}\smash{\begin{tabular}[t]{l}$(i,0)$\\\end{tabular}}}}%
  \end{picture}%
\endgroup%

%% file: main.bib
@article{Schwartz1992,
  author = {R. E. Schwartz},
  title = {The pentagram map},
  volume = {1},
  journal = {Experimental Mathematics},
  number = {1},
  publisher = {A K Peters, Ltd.},
  pages = {71 -- 81},
  year = {1992}
}

@article{Schwartz2001recurrent,
  author = {Schwartz, R. E.},
  year = {2001},
  title = {The pentagram map is recurrent},
  volume = {10},
  journal = {Experimental Mathematics}, 
  pages = {519--528}
}

@article{Schwartz2007Poncelet,
  title = {The Poncelet grid},
  author = {R. E. Schwartz},
  pages = {157--175},
  volume = {7},
  number = {2},
  journal = {Advances in Geometry },
  year = {2007},
  lastchecked = {2024-09-25}
}

@Article{schwartz2008discretemonodromypentagramsmethod,
  author={Schwartz, R. E.},
  title={Discrete monodromy, pentagrams, and the method of condensation},
  journal={Journal of Fixed Point Theory and Applications},
  year={2008},
  volume={3},
  number={2},
  pages={379-409}
}

@Article{Ovsienko2010,
  author={Ovsienko, V.
  and Schwartz, R. E.
  and Tabachnikov, S.},
  title={The pentagram map: a discrete integrable system},
  journal={Communications in Mathematical Physics},
  year={2010},
  volume={299},
  number={2},
  pages={409-446}
}

@article{glick2011Ypatterns,
  title = {The pentagram map and Y-patterns},
  journal = {Advances in Mathematics},
  volume = {227},
  number = {2},
  pages = {1019-1045},
  year = {2011},
  author = {M. Glick}
}

@article{gekhtman2012higherpentagrammapsweighted,
  title = {Higher pentagram maps, weighted directed networks, and cluster dynamics},
  journal = {Electronic Research Announcements},
  volume = {19},
  number = {0},
  pages = {1-17},
  year = {2012},
  author = {M. Gekhtman and M. Shapiro and S. Tabachnikov and A. Vainshtein}
}

@misc{difrancesco2012tsystemsboundariesnetworksolutions,
  title={T-systems with boundaries from network solutions}, 
  author={P. Di Francesco and R. Kedem},
  year={2012},
  eprint={1208.4333},
  archivePrefix={arXiv},
  primaryClass={math.CO}
}

@article{GK13dimerIntegrableSystems,
     author = {Goncharov, A. B. and Kenyon, R.},
     title = {Dimers and cluster integrable systems},
     journal = {Annales scientifiques de l'\'Ecole Normale Sup\'erieure},
     pages = {747--813},
     publisher = {Soci\'et\'e math\'ematique de France},
     volume = {Ser. 4, 46},
     number = {5},
     year = {2013},
     language = {en}
}

@article{Schwartz2013spirals,
author = {Schwartz, R. E.},
title = {Pentagram spirals},
journal = {Experimental Mathematics},
volume = {22},
number = {4},
pages = {384--405},
year = {2013},
publisher = {Taylor \& Francis},
}

@article{Soloviev_2013,
   title={Integrability of the pentagram map},
   volume={162},
   pages = {2815--2853},
   number={15},
   journal={Duke Mathematical Journal},
   publisher={Duke University Press},
   author={Soloviev, F.},
   year={2013}
}

@article{Ovsienko2013Integrability, 
author = {V. Ovsienko and R. E. Schwartz and S. Tabachnikov},
title = {Liouville-Arnold integrability of the pentagram map on closed polygons},
volume = {162},
journal = {Duke Mathematical Journal},
number = {12},
publisher = {Duke University Press},
pages = {2149 -- 2196},
year = {2013}
}

@Article{Khesin2013,
  author={Khesin, B.
  and Soloviev, F.},
  title={Integrability of higher pentagram maps},
  journal={Mathematische Annalen},
  year={2013},
  volume={357},
  number={3},
  pages={1005-1047}
}

@Article{Beffa2013AGDflows,
  author={Mar{\'i} Beffa, G.},
  title={On generalizations of the pentagram map: discretizations of AGD flows},
  journal={Journal of Nonlinear Science},
  year={2013},
  volume={23},
  number={2},
  pages={303-334}
}

@article{KEDEM2015233,
  title = {T-systems and the pentagram map},
  journal = {Journal of Geometry and Physics},
  volume = {87},
  pages = {233-247},
  year = {2015},
  note = {Finite dimensional integrable systems: on the crossroad of algebra, geometry and physics},
  author = {R. Kedem and P. Vichitkunakorn}
}

@article{Beffa2015IntegrableGeneralizations,
  author = {Marí Beffa, G.},
  title = "{On integrable generalizations of the pentagram map}",
  journal = {International Mathematics Research Notices},
  volume = {2015},
  number = {12},
  pages = {3669-3693},
  year = {2014}
}

@Inbook{Fock2016,
  author="Fock, V. V.
  and Marshakov, A.",
  editor="Alvarez Consul, Luis
  and Andersen, J{\o}rgen Ellegaard
  and Mundet i Riera, Ignasi",
  title="Loop groups, clusters, dimers and integrable systems",
  bookTitle="Geometry and Quantization of Moduli Spaces",
  year="2016",
  publisher="Springer International Publishing",
  address="Cham",
  pages="1--65"
}

@article{GP2016ymeshes,
author = {Glick, M. and Pylyavskyy, P.},
title = {Y-meshes and generalized pentagram maps},
journal = {Proceedings of the London Mathematical Society},
volume = {112},
number = {4},
pages = {753-797},
year = {2016}
}

@article{Khesin2016,
  author = {Khesin, B. and Soloviev, F.},
  journal = {Journal of the European Mathematical Society},
  language = {eng},
  number = {1},
  pages = {147-179},
  publisher = {European Mathematical Society Publishing House},
  title = {The geometry of dented pentagram maps},
  volume = {018},
  year = {2016}
}

@article{Felipe2019Grassmannians,
  author = {Felipe, R. and Mar{\'\i} Beffa, G.},
  title = {The pentagram map on {Grassmannians}},
  journal = {Annales de l'Institut Fourier},
  pages = {421--456},
  publisher = {Association des Annales de l{\textquoteright}institut Fourier},
  volume = {69},
  number = {1},
  year = {2019},
  language = {en}
}

@article{Glick2020limit,
  author = {Glick, M.},
  title = "{The limit point of the pentagram map}",
  journal = {International Mathematics Research Notices},
  volume = {2020},
  number = {9},
  pages = {2818-2831},
  year = {2018}
}

@article{Nackan2021kdv,
  title = {Continuous limits of generalized pentagram maps},
  journal = {Journal of Geometry and Physics},
  volume = {167},
  pages = {104292},
  year = {2021},
  author = {D. Nackan and R. Speciel}
}

@misc{schwartz2021textbookcasepentagramrigidity,
  title={A textbook case of pentagram rigidity}, 
  author={R. E. Schwartz},
  year={2021},
  eprint={2108.07604},
  archivePrefix={arXiv},
  primaryClass={math.DS}
}

@article{Izosimov2022poisson,
  title = {Pentagram maps and refactorization in Poisson-Lie groups},
  journal = {Advances in Mathematics},
  volume = {404},
  pages = {108476},
  year = {2022},
  author = {A. Izosimov}
}

@article{Izosimov2022poncelet, 
  title={The pentagram map, Poncelet polygons, and commuting difference operators}, 
  volume={158},
  number={5}, 
  journal={Compositio Mathematica}, 
  author={Izosimov, A.}, 
  year={2022}, 
  pages={1084–1124}
}

@article{izosimov2023longdiagonalpentagrammaps,
  author = {Izosimov, A. and Khesin, B.},
  title = {Long-diagonal pentagram maps},
  journal = {Bulletin of the London Mathematical Society},
  volume = {55},
  number = {3},
  pages = {1314-1329},
  year = {2023}
}

@article{Weinreich2023, 
  title={The algebraic dynamics of the pentagram map}, 
  volume={43}, 
  number={10}, 
  journal={Ergodic Theory and Dynamical Systems}, 
  author={Weinreich, M. H.}, 
  year={2023}, 
  pages={3460-3505}
}

@article{schwartz2024pentagramrigiditycentrallysymmetric,
  author = {R. E. Schwartz},
  title = "{Pentagram rigidity for centrally symmetric octagons}",
  journal = {International Mathematics Research Notices},
  volume = {2024},
  number = {12},
  pages = {9535-9561},
  year = {2024}
}

@article{schwartz2024flappingbirdspentagramzoo,
  title={The flapping birds in the pentagram zoo}, 
  author={R. E. Schwartz},
  journal={Arnold Mathematical Journal},
  year={2025},
  pages={to appear}
}
